\documentclass[a4paper]{amsart}
\author[Raghavan]{Dilip Raghavan}
\thanks{This paper was completed during the Thematic
 Program on Set Theoretic Methods in Algebra, Dynamics and Geometry held at the Fields Institute.
The authors thank the Fields Institute for its kind hospitality.}
\address[Raghavan]{Department of Mathematics\\
National University of Singapore\\
Singapore 119076.}
\email{\href{dilip.raghavan@protonmail.com}{dilip.raghavan@protonmail.com}}
\urladdr{\url{https://dilip-raghavan.github.io/}}
\author[Stepr{\= a}ns]{Juris Stepr{\= a}ns}
\thanks{Second author is partially supported by grants from NSERC}
\address[Stepr{\= a}ns]{Department of Mathematics and Statistics\\ York University\\ 4700 Keele St.\@\\ Toronto, ON M3J 1P3, Canada.}
\email{\href{steprans@yorku.ca}{steprans@yorku.ca}}
\urladdr{\url{https://steprans.info.yorku.ca/}}
\date{\today}
\title[Stable ordered-union vs.\@ Selective]{Stable ordered-union versus selective ultrafilters}
\usepackage[only, llbracket, rrbracket]{stmaryrd}
\usepackage{amssymb, amsmath, amsthm, mathrsfs, enumitem, amsfonts, latexsym, bbm}
\usepackage{hyperref}
\def\polhk#1{\setbox0=\hbox{#1}{\ooalign{\hidewidth
    \lower1.5ex\hbox{`}\hidewidth\crcr\unhbox0}}}
\newtheorem{Theorem}{Theorem}[section]
\newtheorem{Claim}[Theorem]{Claim}
\newtheorem{Lemma}[Theorem]{Lemma}
\newtheorem{Cor}[Theorem]{Corollary}

\theoremstyle{definition}
\newtheorem{Def}[Theorem]{Definition}

\theoremstyle{remark}
\newtheorem{remark}[Theorem]{Remark}

\newcommand{\forces}{\Vdash}
\newcommand{\restrict}{\mathord\upharpoonright}
\newcommand{\forallbutfin}{{\forall}^{\infty}}
\newcommand{\existsinf}{{\exists}^{\infty}}

\renewcommand{\[}{\left[}
\renewcommand{\]}{\right]}
\newcommand{\PP}{\mathbb{P}}

\newcommand{\QQ}{\mathbb{Q}}
\newcommand{\lc}{\left|}
\newcommand{\rc}{\right|}

\newcommand\ZFC{\mathrm{ZFC}}
\newcommand\FIN{\mathrm{FIN}}

\newcommand\CH{\mathrm{CH}}

\newcommand{\BS}{{\omega}^{\omega}}

\DeclareMathOperator{\dom}{dom}

\DeclareMathOperator{\succc}{succ}

\DeclareMathOperator{\cf}{cf}

\DeclareMathOperator{\clu}{{cl}_{\cup}}
\newcommand{\Pset}{\mathcal{P}}

\newcommand{\BBB}{\mathcal{B}}

\newcommand{\C}{{\mathscr{C}}}

\newcommand{\DDD}{\mathcal{D}}

\newcommand{\GGG}{{\mathcal{G}}}

\newcommand{\EEE}{{\mathcal{E}}}
\newcommand{\UUU}{{\mathcal{U}}}
\newcommand{\VVV}{{\mathcal{V}}}
\newcommand{\HHH}{{\mathcal{H}}}

\newcommand{\FFF}{{\mathcal{F}}}

\newcommand{\TT}{{\mathbb{T}}}

\newcommand{\V}{{\mathbf{V}}}

\newcommand{\VG}{{{\mathbf{V}}[G]}}

\newcommand{\pr}[2]{\left\langle #1, #2 \right\rangle}
\newcommand{\seq}[4]{\left\langle {#1}_{#2}: #2 #3 #4 \right\rangle}
\newcommand{\trp}[3]{{\left( {#1}^{#2} \right)}^{#3}}

\newcommand{\pc}[2]{{\[#1\]}^{#2}}

\newcommand{\lb}[2]{#1 \; {<}_{\mathtt{b}} \; #2}
\newcommand{\lbb}{{<}_{\mathtt{b}}}
\newcommand{\lv}[2]{{\mathtt{Lev}}_{#2}{\mathord{\left(#1\right)}}}
\begin{document}
\begin{abstract}
 It will be shown to be consistent that there are at least two non-isomorphic selective ultrafilters, but no stable ordered-union ultrafilters.
 This answers a question of Blass from his 1987 paper~\cite{blass-hindman}.
\end{abstract}
\maketitle
\section{Introduction} \label{sec:intro}
Ramsey's theorem~\cite{ramsey} and Hindman's theorem~\cite{MR349574}, together with their common generalization, the Milliken--Taylor theorem of~\cite{MR373906} and~\cite{MR424571}, are among the most widely used facts of infinite combinatorics.
The results of this paper clarify an important aspect of their relationship to one another.
Before stating these results, it may be worthwhile to set the stage by reviewing some aspects of the history of this subject.
The classical theorem of Ramsey is often referred to as a higher dimensional pigeonhole principle and this might lead one to think that there is not much of interest to say about the one dimensional case. But any consideration of van der Waerden's theorem, to say nothing of Szemer\'{e}di's Theorem, would reveal the error of this view point. The interest, as well as the complexity of both of these pigeonhole type results stems from the fact that their conclusions yield sets that are not necessarily large in terms of cardinality, but large in terms of their algebraic structure.
For example, a consequence of van der Waerden's Theorem is that for any partition of $\mathbb N$ into two pieces, one of the pieces will contain arbitrarily large arithmetic progressions, or, in other words, arbitrarily large sets that are closed under addition by a fixed integer.
This provides the motivation for a conjecture of Graham and Rothschild~\cite{MR284352} that for any partition of $\mathbb N$ into finitely many pieces, one of the pieces
contains a set that is closed under all sums of distinct members.
The truth of this conjecture was established by Hindman (Theorem 3.1 in  \cite{MR349574}) and
is now known as Hindman's Theorem.

The history leading to the proof is both interesting and relevant to the results about to be presented.
A natural approach to proving Hindman's Theorem would be to proceed inductively to construct $k_0, k_1, \dotsc, k_n$ and $A_n$ such that all sums of distinct integers from  $k_0, k_1, \dotsc, k_n$ belong to one piece of the given partition and, crucially, $A_n$ is an infinite set from which it is possible to select the next integer $k_{n+1}$. And, indeed, this is how all proofs of Hindman's Theorem proceed, but the technical details in the original proof of \cite{MR349574} are daunting and a potential stumbling block is correctly choosing the set $A_n$.
An early realization was that it would help if the $A_n$ could be selected from an ultrafilter; but that ultrafilter would need to enjoy some very specific properties.

In the earlier paper~\cite{MR307926}, Hindman had made the following observation: The Graham and Rothschild Conjecture holds if and only if there is an ultrafilter on $\mathbb N$ every member of which contains all non-repeating sums from some infinite subset.
In the same paper, he shows that if the Continuum Hypothesis holds, then the Graham and Rothschild Conjecture is equivalent to the existence of an ultrafilter $\UUU$ that is an idempotent (namely $\UUU + \UUU = \UUU$) in the semigroup $(\beta\mathbb N,+)$ where ultrafilters are thought of as finitely additive measures and the $+$ operation is the convolution of measures.
This paper will focus on a strengthening of the property of an ultrafilter that all of its members contain the non-repeating sums of an infinite subset.
But before explaining the strengthening we need to return to the story and mention that Hindman was able to eliminate the use of ultrafilters by fashioning an elaborately technical, albeit elementary, argument in \cite{MR349574} that seemingly did away with the use of the idempotent.
Somewhat later, though, Baumgartner produced a much simpler version of Hindman's technical argument in \cite{MR354394}.
A key idea used by Baumgartner was the notion of a large set, somewhat reminiscent of the construction of Haar measure.

As explained by Bergelson in \cite{MR2757532}, the notion of largeness in this context can be traced back to Poincar\'{e}'s work on celestial mechanics and, when combined with an idempotent ultrafilter, very quickly yields Hindman's Theorem.
This realization of Galvin and Glazer that a much older and more general theorem of Ellis~\cite{MR101283} about idempotents in compact semigroups could vastly simplify the proof of Hindman's Theorem points to the important role of ultrafilters in this area of Ramsey Theory.
So even though Hindman's construction using the Continuum Hypothesis was not, ultimately, necessary, it played an important role in the development of the subject and fostered future research.

For example, in \cite{MR349574} van Douwen is credited with realizing that, assuming the Continuum Hypothesis, it is possible to construct an ultrafilter that satisfies a stronger version of the property Hindman established under the same assumptions.
As documented in \cite{MR2991425}, it was noticed by van Douwen that, assuming the Continuum Hypothesis, certain ultrafilters had a base consisting of all of the finite sums of some infinite set of positive integers.
The difference is worth highlighting: Hindman had asked only that each member of his ultrafilter contain a set closed under finite sums, but van Douwen is asking that this set actually belongs to the ultrafilter.
These ultrafilters identified by van Douwen are now known as {\em strongly summable ultrafilters} and the question of whether the Continuum Hypothesis is needed to construct them is also attributed to van Douwen.
The strongly summable ultrafilters play a significant role in the theory of the semigroup $(\beta \mathbb N,+)$  and the interested reader can learn more about this, and much more, in \cite{MR2893605}.

Closely related to Hindman's Theorem is Theorem~3.3 from  \cite{MR349574}, a result about the union operation on the finite subsets of the positive integers.
It will be worthwhile introducing some notation in order to state this result, since this notation will be used in later sections as well.
\begin{Def} \label{def:omegafinclu}
 As usual in set theory, $\omega$ denotes $\mathbb{N}$.
 For a set $X$ and a cardinal $\kappa$, $\pc{X}{\kappa} = \{A \subseteq X: \lc A \rc = \kappa\}$ and $\pc{X}{< \kappa} = \{A \subseteq X: \lc A \rc < \kappa\}$.
 $\FIN$ denotes $\pc{\omega}{< {\aleph}_{0}} \setminus \{\emptyset\}$.
 
 $A \subseteq \FIN$ is said to be \emph{closed under finite unions} if $s \cup t \in A$, for all $s, t \in A$.
 Let
 \begin{align*}
  \clu\left(A\right) = \left\{\bigcup{X}: X \in \pc{A}{< {\aleph}_{0}} \setminus \{\emptyset\} \right\}.
 \end{align*}
 $\clu\left(A\right) \subseteq \FIN$ and $\clu\left(A\right)$ is closed under finite unions.
 Indeed, $\clu\left(A\right)$ is the closure of $A$ under finite non-empty unions.
\end{Def}
In analogy with the strongly summable ultrafilters, it is possible to formulate the following definition.
\begin{Def} \label{def:union}
 An ultrafilter on $\FIN$ will be called a \emph{union ultrafilter} if it has a base consisting of sets of the form $\clu\left(A\right)$, where $A \subseteq \FIN$ consists of pairwise disjoint sets and is infinite.
\end{Def}
However, it turns out that the connection between strongly summable ultrafilters and union ultrafilters goes beyond analogy.
In \S5 of \cite{blass-hindman} it is shown that the mapping from $\FIN$ to $\mathbb N$ sending $a$ to $\sum_{n\in a}2^n$ sends union ultrafilters to strongly summable ultrafilters and a converse is established in Theorem~1 of \cite{MR906807}.
Indeed Baumgartner's simplified proof of Hindman's Theorem actually states that if $\FIN$ is partitioned into finitely many pieces then there is an infinite $A\subseteq \FIN$ consisting of pairwise disjoint sets such that $\clu\left(A\right)$ is contained in one of the pieces.
It turns out that many of the constructions of union ultrafilters actually produce a stronger property than needed and this is the main property of ultrafilters to be examined in the subsequent sections.
\begin{Def} \label{def:block}
 For $s, t \in \FIN$, write $\lb{s}{t}$ to mean $\max(s) < \min(t)$.
 Note that $\lbb$ is a transitive and irreflexive relation on $\FIN$.
 
 $X \subseteq \FIN$ is called a \emph{block sequence} if $X$ is non-empty and it is linearly ordered by the relation $\lbb$.
 More formally, $X \subseteq \FIN$ is a block sequence if and only if $X \neq \emptyset$ and $\forall s, t \in X\[s = t \vee \lb{s}{t} \vee \lb{t}{s}\]$.
 Note that if $X$ is a block sequence, then $\pr{X}{\lbb}$ is a well-order and the order type of $\pr{X}{\lbb}$ is $\lc X \rc$.
 The notation $X(i)$ will be used to denote the $i$th member of $\pr{X}{\lbb}$, for all $i < \lc X \rc$.
 
 The following notation, which is based on Todorcevic~\cite{introramsey}, will be used.
 For $1 \leq \alpha \leq \omega$ and $A \subseteq \FIN$,
 \begin{align*}
  &{A}^{\[\alpha\]} = \left\{ X \subseteq A: X \ \text{is a block sequence and} \ \lc X \rc = \alpha \right\};\\
  &{A}^{\[< \alpha\]} = \left\{ X \subseteq A: X \ \text{is a block sequence and} \ 1 \leq \lc X \rc < \alpha \right\}.
 \end{align*}
 Thus ${A}^{\[\alpha\]}$ denotes the collection of all block sequences of length $\alpha$ from $A$ and ${A}^{\[< \alpha\]}$ is the collection of all block sequences of length $< \alpha$ from $A$ (by definition block sequences are non-empty).
 For $A \subseteq \FIN$,
 \begin{align*}
  \[A\] = \left\{\bigcup{X}: X \in {A}^{\[< \omega\]}\right\}
 \end{align*}
 Thus $\[A\]$ is the collection of all unions of finite length block sequences from $A$.
 Note $A \subseteq \[A\] \subseteq \FIN$.
 
 Further, if $X \subseteq \FIN$ is a block sequence, then $\clu\left(X\right) = \[X\]$.
\end{Def}
%
\begin{Def} \label{def:orderedunionuf}
 An ultrafilter $\HHH$ on $\FIN$ is called an \emph{ordered-union ultrafilter} if $\HHH$ has a base consisting of sets of the form $\[X\]$, where $X$ is an infinite block sequence.
 In other words, an ultrafilter $\HHH$ on $\FIN$ is an ordered-union ultrafilter if and only if $\forall A \in \HHH \exists X \in {\FIN}^{\[\omega\]}\[\[X\] \in \HHH \ \text{and} \ \[X\] \subseteq A\]$.
\end{Def}
That union ultrafilters and ordered-union ultrafilters are consistently different notions is proved in Theorem~4 of \cite{MR906807} assuming the Continuum Hypothesis.
While the ordered-union ultrafilters are defined to contain witnesses to Hindman's theorem, it is not clear that they also contain witnesses to the higher dimensional analogue of Hindman's theorem, namely the Milliken-Taylor theorem of \cite{MR373906} and \cite{MR424571}.
Obtaining such properties requires introducing a further strengthening.
\begin{Def} \label{def:refines}
 Let $X, Y \in {\FIN}^{\[\omega\]}$.
 $Y$ is said to \emph{refine} $X$ if $\forall i \in \omega\[Y(i) \in \[X\]\]$.
 $Y$ is said to \emph{almost refine} $X$ if $\forallbutfin i \in \omega\[Y(i) \in \[X\]\]$.
\end{Def}
\begin{Def} \label{def:stable}
 An ordered-union ultrafilter $\HHH$ on $\FIN$ is said to be \emph{stable} if for every sequence $\seq{X}{n}{\in}{\omega}$ with the property that ${X}_{n} \in {\FIN}^{\[\omega\]}$ and $\[{X}_{n}\] \in \HHH$, for all $n \in \omega$, there exists $Y \in {\FIN}^{\[\omega\]}$ such that $\forall n \in \omega\[Y \ \text{almost refines} \ {X}_{n}\]$ and $\[Y\] \in \HHH$.
\end{Def}
While stability is deceptively similar to the definition of a P-point, it is shown in \cite{blass-hindman} to be a weaker notion.
However, it was also shown in \cite{blass-hindman} to capture a number of important combinatorial properties.
Indeed, within the realm of ordered-union ultrafilters on $\FIN$, stability is the analogue of selectivity for ultrafilters on $\mathbb{N}$.
\begin{Def} \label{def:canonical}
 An ordered-union ultrafilter $\HHH$ on $\FIN$ has the \emph{canonical partition property} if for every function $f: \FIN \rightarrow \omega$, there exists $A \in \HHH$ such that one of the following statements hold:
 \begin{enumerate}
  \item
  $\forall s, t \in A\[f(s)=f(t)\]$;
  \item
  $\forall s, t \in A\[f(s)=f(t) \leftrightarrow \min(s)=\min(t)\]$;
  \item
  $\forall s, t \in A\[f(s)=f(t) \leftrightarrow \max(s)=\max(t)\]$;
  \item
  $\forall s, t \in A\[ f(s)=f(t) \leftrightarrow \left( \min(s)=\min(t) \wedge \max(s)=\max(t) \right) \]$;
  \item
  $\forall s, t \in A\[f(s)=f(t) \leftrightarrow s=t\]$.
 \end{enumerate}
\end{Def}
Blass~\cite{blass-hindman} proved that ordered-union ultrafilters are stable if and only if they enjoy strong Ramsey theoretic properties.
\begin{Theorem}[Theorem 4.2 of Blass~\cite{blass-hindman}] \label{thm:blassramsey}
 For any ordered-union ultrafilter $\HHH$ on $\FIN$, the following are equivalent:
 \begin{enumerate}
  \item
  $\HHH$ is stable;
  \item
  for each $1 \leq n < \omega$ and $c: {\FIN}^{\[n\]} \rightarrow 2$, there is an $A \in \HHH$ such that $c$ is constant on ${A}^{\[n\]}$; 
  \item
  for each $c: {\FIN}^{\[2\]} \rightarrow 2$, there is an $A \in \HHH$ such that $c$ is constant on ${A}^{\[2\]}$;
  \item \label{blassramsey:fourth}
  $\HHH$ has the canonical partition property.
 \end{enumerate}
\end{Theorem}
The theorem of \cite{blass-hindman} actually contains a number of other equivalences.
One of these will be touched upon below.
It will be illuminating to note that Theorem \ref{thm:blassramsey} is the analogue of a well-known characterization of selective ultrafilters on $\mathbb{N}$ due to Kunen.
Recall the following definition.
\begin{Def} \label{def:selective}
 An ultrafilter $\UUU$ on $\omega$ is called \emph{selective} if for every $f: \omega \rightarrow \omega$, there is a set $A \in \UUU$ such that $f$ is either constant or 1-1 on $A$.
\end{Def}
Observe that Definition \ref{def:selective} is saying that every function on $\omega$ attains a canonical form on some member of a selective ultrafilter.
It is the analogue of Item (\ref{blassramsey:fourth}) of Theorem \ref{thm:blassramsey} for ultrafilters on $\mathbb{N}$.
An unpublished theorem of Kunen says that selectivity is equivalent to the existence of witnesses for Ramsey's theorem, which provides the analogue of Theorem \ref{thm:blassramsey} for ultrafilters on $\mathbb{N}$.
\begin{Theorem}[Kunen; see \cite{BJ}] \label{thm:kunen}
 The following are equivalent for an ultrafilter $\UUU$ on $\omega$:
 \begin{enumerate}
  \item
  $\UUU$ is selective;
  \item \label{kunen:second}
  for each $1 \leq n < \omega$ and $c: {\[\omega\]}^{n} \rightarrow 2$, there is an $A \in \UUU$ such that $c$ is constant on ${\[A\]}^{n}$; 
  \item \label{kunen:third}
  for each $c: {\[\omega\]}^{2} \rightarrow 2$, there is an $A \in \UUU$ such that $c$ is constant on ${\[A\]}^{2}$;
 \end{enumerate}
\end{Theorem}
The homogeneity properties given by Items (\ref{kunen:second}) and (\ref{kunen:third}) of Theorems \ref{thm:blassramsey} and \ref{thm:kunen} can be improved to cover all analytic partitions of ${\FIN}^{\[\omega\]}$ and $\pc{\omega}{\omega}$ respectively.
In the presence of large cardinals, they be strengthened even further to obtain homogeneous sets for any partition in the inner model $\mathbf{L}(\mathbb{R})$.
This further strengthening leads to the notion of a generic ultrafilter over a Ramsey space.
Forcing with $\pc{\omega}{\omega}$ ordered by almost inclusion adds a selective ultrafilter and forcing with ${\FIN}^{\[\omega\]}$ ordered by almost refinement adds a stable ordered-union ultrafilter.
Furthermore, it is a theorem of Todorcevic, appearing in \cite{carlosstevo}, that in the presence of large cardinals, the selective ultrafilters are precisely those ultrafilters on $\omega$ that are generic over $\mathbf{L}(\mathbb{R})$ for $\pc{\omega}{\omega}$ partially ordered by almost inclusion.
Similarly, in the presence of large cardinals, the stable ordered-union ultrafilters are precisely those ultrafilters on $\FIN$ that are generated from a generic filter over $\mathbf{L}(\mathbb{R})$ for ${\FIN}^{\[\omega\]}$ partially ordered by almost refinement.
Blass~\cite{Bl} proved the same result in a version of the Solovay model.

Todorcevic~\cite{introramsey} has investigated the abstract concept of a topological Ramsey space.
In \cite{introramsey}, the space corresponding to Ramsey's theorem is called the \emph{Ellentuck space} and the one corresponding to Hindman's theorem and the Milliken-Taylor theorem is known as the \emph{Milliken-Taylor space}.
The results discussed above establish that, in the presence of large cardinals, the selective ultrafilters are the generic ultrafilters corresponding to the Ellentuck space, while the stable ordered-union ultrafilters are the generic ultrafilters corresponding to the Milliken-Taylor space.
An abstract notion of generic ultrafilter corresponding to an arbitrary topological Ramsey space has been studied by Mijares~\cite{MR2330595}.

The existence of canonical forms for functions has striking implications for which ultrafilters may appear below a selective or a stable ordered-union ultrafilter in the Rudin-Keisler or Tukey orders.
Recall the following definitions.
\begin{Def} \label{def:rk}
 Let $\FFF$ be a filter on $X$ and $\GGG$ a filter on $Y$.
 $\FFF$ is \emph{Rudin-Keisler below} or \emph{RK below} $\GGG$, written as $\FFF \; {\leq}_{RK} \; \GGG$, if there is a function $f: Y \rightarrow X$ such that $\FFF = \{A \subseteq X: {f}^{-1}(A) \in \GGG\}$.
 
 $\FFF$ is \emph{RK-equivalent} to $\GGG$, written as $\FFF \; {\equiv}_{RK} \; \GGG$, if $\FFF \; {\leq}_{RK} \; \GGG$ and $\GGG \; {\leq}_{RK} \; \FFF$.
 
 A subset $\BBB \subseteq \FFF$ is said to be \emph{cofinal in $\FFF$} if for each $A \in \FFF$, there exists $B \in \BBB$ with $B \subseteq A$.
 A map $\varphi: \GGG \rightarrow \FFF$ is a \emph{convergent map} if the image under $\varphi$ of every cofinal subset of $\GGG$ is cofinal in $\FFF$.
 $\FFF$ is said to be \emph{Tukey below} $\GGG$, written as $\FFF \; {\leq}_{T} \; \GGG$, if there is a convergent map $\varphi: \GGG \rightarrow \FFF$.
 If $\FFF \; {\leq}_{T} \; \GGG$ and $\GGG \; {\leq}_{T} \; \FFF$, then $\FFF$ and $\GGG$ are said to be \emph{Tukey equivalent} and this is denoted $\FFF \; {\equiv}_{T} \; \GGG$.
\end{Def}
It is a well-known fact (see e.g.\@ \cite{maryellenorder}) that when $\UUU$ and $\VVV$ are ultrafilters on $\omega$, $\UUU \; {\equiv}_{RK} \; \VVV$ if and only if there is a permutation $e: \omega \rightarrow \omega$ so that $\UUU = \{ e\[A\]: A \in \VVV \}$.
In this case, the equivalence $\UUU \; {\equiv}_{RK} \; \VVV$ is often expressed by saying $\UUU$ and $\VVV$ are \emph{RK-isomorphic}.

An easy consequence of Item (\ref{kunen:third}) of Theorem \ref{thm:kunen} is that selective ultrafilters are RK-minimal among all non-principal ultrafilters on $\omega$.
Hence if $\UUU$ and $\VVV$ are selective ultrafilters on $\omega$, then either $\UUU$ and $\VVV$ are RK-incomparable or they are RK-isomorphic.
Using the existence of a more complicated canonical form for convergent maps (see, for example, Lemma 28 of \cite{pomegaembed}), it is proved by Raghavan and Todorcevic~\cite{tukey} that selective ultrafilters are also Tukey minimal among all non-principal ultrafilters on $\omega$.

Analogously, Item (\ref{blassramsey:fourth}) of Theorem \ref{thm:blassramsey} implies that there are precisely two RK-equivalence classes of selective ultrafilters that are RK-below a stable ordered-union ultrafilter.
\begin{Def} \label{def:maxmin}
 The function ${f}_{\max}: \FIN \rightarrow \omega$ is defined by ${f}_{\max}(s) = \max(s)$ and the function ${f}_{\min}: \FIN \rightarrow \omega$ is defined by ${f}_{\min}(s) = \min(s)$, for all $s \in \FIN$.
\end{Def}
\begin{Def} \label{def:minmax}
 Let $\HHH$ be an ultrafilter on $\FIN$.
 Then \emph{${\HHH}_{\min}$} and \emph{${\HHH}_{\max}$} are the RK-projections of $\HHH$ induced by ${f}_{\min}$ and ${f}_{\max}$ respectively.
 More formally,
 \begin{align*}
  &{\HHH}_{\min} = \left\{ M \subseteq \omega: {f}^{-1}_{\min}(M) \in \HHH \right\}\\
  &{\HHH}_{\max} = \left\{ M \subseteq \omega: {f}^{-1}_{\max}(M) \in \HHH \right\}.
 \end{align*}
\end{Def}
\begin{Theorem}[Blass~\cite{blass-hindman}] \label{thm:minmaxrk}
 Let $\HHH$ be a stable ordered-union ultrafilter on $\FIN$.
 Then ${\HHH}_{\min}$ and ${\HHH}_{\max}$ are selective ultrafilters on $\omega$.
 Furthermore, ${\HHH}_{\min} \; {\not\equiv}_{RK} \; {\HHH}_{\max}$, and hence ${\HHH}_{\min}$ and ${\HHH}_{\max}$ are RK-incomparable.
\end{Theorem}
It is also implicit in \cite{blass-hindman} that if $\HHH$ is a stable ordered-union ultrafilter and $\UUU$ is a selective ultrafilter such that $\UUU \; {\leq}_{RK} \; \HHH$, then either $\UUU \; {\equiv}_{RK} \; {\HHH}_{\min}$ or $\UUU \; {\equiv}_{RK} \; {\HHH}_{\max}$.
A modification of the methods in \cite{tukey} similarly shows that the only selective ultrafilters that are Tukey below a stable ordered-union ultrafilter are the ones that are RK-isomorphic to ${\HHH}_{\min}$ or to ${\HHH}_{\max}$.

Thus the existence of a stable ordered-union ultrafilter guarantees the existence of at least two distinct RK-classes of selective ultrafilters.
Blass' question, which will be answered in this paper, asks if the converse is true.
Blass had observed in \cite{blass-hindman} that under the Continuum Hypothesis, any two RK-non-isomorphic selective ultrafilters are realized at the ${\HHH}_{\min}$ and ${\HHH}_{\max}$ of some stable ordered-union ultrafilter $\HHH$.
\begin{Theorem}(Theorem 2.4 of Blass~\cite{blass-hindman})
 Assume $\CH$, and let $\UUU$ and $\VVV$ be selective ultrafilters such that $\UUU \; {\not\equiv}_{RK} \; \VVV$.
 Then there is a stable ordered-union ultrafilter $\HHH$ such that ${\HHH}_{\max} = \UUU$ and ${\HHH}_{\min} = \VVV$.
\end{Theorem}
In the final sentence of \cite{blass-hindman}, Blass asked whether the existence of at least two RK-non-isomorphic selective ultrafilters guarantees the existence of a stable ordered-union ultrafilter.
The main result of this paper gives a negative answer to this long-standing problem.
It will be shown that there is a model of $\ZFC$ with ${2}^{{\aleph}_{0}}$ distinct RK-classes of selective ultrafilters, but no stable ordered-union ultrafilters.
It will also be shown that for any $\kappa \leq {\aleph}_{2}$, it is possible to construct a model of $\ZFC$ with precisely $\kappa$ distinct RK-classes of selective ultrafilters, but no stable ordered-union ultrafilters.
These appear to be the first known models of $\ZFC$ containing continuum many distinct generic ultrafilters corresponding to some Ramsey space, but no generic ultrafilters at all corresponding to another Ramsey space.

The approach taken will be to iteratively destroy all stable ordered-union ultrafilters  while preserving all selective ultrafilters from a suitably chosen ground model.
Many of the ideas to be used have their roots in Shelah's proof (Theorem~4.1 of \S4 of Chapter~XVIII of \cite{PIF}) of the consistency of a single P-point, up to RK equivalence, where a similar strategy is employed. However, it has to be kept in mind that stable ordered-union ultrafilters are not P-points, so the arguments used by Shelah, which rely heavily on the combinatorics of P-points, cannot be applied directly.
The major advance in this paper is the introduction of a new partial order that allows different arguments to be employed to circumvent these difficulties.
In fact, it may be of interest to restate the main result of this paper in the language of forcing theory.
In this language, the main result says that for any stable ordered-union ultrafilter $\HHH$, there is a proper $\BS$-bounding forcing notion $\PP(\HHH)$ which preserves all selective ultrafilters and \emph{strongly destroys} $\HHH$, which means that $\HHH$ cannot be extended to a stable ordered-union ultrafilter in any forcing extension by any $\BS$-bounding forcing that contains $\PP(\HHH)$ as a complete suborder.

New techniques for proving the preservation of selective ultrafilters are introduced in establishing the main result.
All previously known methods for destroying an ultrafilter $\VVV$ while preserving a selective ultrafilter require that $\VVV$ be a P-point.
The new methods introduced in this paper allow the destruction of the non-P-point $\HHH$ while ${\HHH}_{\min}$ and ${\HHH}_{\max}$ are both preserved.
It is well known (see Corollary 3 of \cite{blassrk}) that a P-point cannot have two RK-non-isomorphic selective ultrafilters RK-below it.
Hence this paper is the first known instance where an ultrafilter $\VVV$ is destroyed while two RK-non-isomorphic ultrafilters that are RK-below $\VVV$ are preserved.
\section{Some preliminaries} \label{sec:prelim}
This section will establish several combinatorial results about two player games on selective and stable ordered-union ultrafilters that will be essential to virtually all of the results in later sections.
The most important of these will be Lemma \ref{lem:selstabgame}.
It also provides a good opportunity to establish notation and to gather together some well-known results as well as some simple observations involving the concepts defined in Section \ref{sec:intro}.
\begin{Def} \label{def:image}
 If $F$ is a function and $X \subseteq \dom(F)$, then $F\[X\]$ denotes the image of $X$ under $F$.
 Formally, $F\[X\] = \left\{F(x): x \in X\right\}$.
\end{Def}
With regard to the notations of Definition \ref{def:image} and Definition \ref{def:block} from Section \ref{sec:intro}, the reader should bear in mind that if $F$ is a function and $X \subseteq \FIN$ is such that $X \subseteq \[X\] \subseteq \dom(F)$, then $F\[X\]$ is the image of $X$ under $F$ while $F\[\[X\]\]$ is the image of $\[X\]$ under $F$.
This is unlikely to cause confusion.
\begin{Def} \label{def:tree}
 Let $X$ be any set.
 $T \subseteq {X}^{< \omega}$ is called a \emph{subtree} if $T$ is downwards closed.
 In other words, $T \subseteq {X}^{< \omega}$ is a subtree if and only if $\forall g \in T \forall l' \leq \dom(g)\[g\restrict l' \in T\]$.
 For a subtree $T \subseteq {X}^{< \omega}$, the \emph{$n$th level of $T$}, denoted \emph{$\lv{T}{n}$}, is $\{f \in T: \dom(f) = n \}$, for all $n \in \omega$.
 For a subtree $T \subseteq {X}^{< \omega}$ and $f \in T$, define
 \begin{align*}
  &T\langle f \rangle = \left\{ g \in T: f \subseteq g \vee g \subseteq f \right\};\\
  &{\succc}_{T}(f) = \left\{ \sigma: {f}^{\frown}{\langle \sigma \rangle} \in T \right\}.
 \end{align*}
 For a subtree $T \subseteq {X}^{< \omega}$, $\[T\]$ denotes the collection of all infinite branches through $T$.
 In other words,
 \begin{align*}
  \[T\] = \left\{ f \in {X}^{\omega}: \forall l \in \omega \[f \restrict l \in T \]\right\}.
 \end{align*}
\end{Def}
Note that if $\HHH$ is an ordered-union ultrafilter on $\FIN$, then for any $A \in \HHH$ and $k \in \omega$, $\{s \in A: \min(s) > k\} \in \HHH$ because $\{s \in A: \min(s) \leq k\}$ does not contain any set of the form $\[X\]$ for $X \in {\FIN}^{\[\omega\]}$.
Although the next two results do not explicitly appear in Blass~\cite{blass-hindman}, their proofs are standard and will be omitted.
\begin{Lemma} \label{lem:stableramsey}
 Suppose $\HHH$ is a stable ordered-union ultrafilter on $\FIN$.
 Then for any $A \in \HHH$, $1 \leq n, k < \omega$, and $c: {A}^{\[n\]} \rightarrow k$, there exists $B \in \HHH$ such that $B \subseteq A$ and $c$ is constant on ${B}^{\[n\]}$.
\end{Lemma}
\begin{Lemma} \label{lem:selectivediag}
 Suppose $\HHH$ is a stable ordered-union ultrafilter on $\FIN$.
 Suppose $\seq{A}{n}{\in}{\omega}$ is a sequence with ${A}_{n} \in \HHH$, for all $n \in \omega$.
 Then there exists $X \in {\FIN}^{\[\omega\]}$ such that $\[X\] \in \HHH$, $\[X\] \subseteq {A}_{0}$, and $\forall i \in \omega\[\[\left\{ X(j): j > i \right\}\] \subseteq {A}_{\max\left(X(i)\right)}\]$.
\end{Lemma}
\begin{Def} \label{def:IsNX}
 Suppose $s \in \FIN$.
 Define $I(s) = \{ k \in \omega: \min(s) \leq k \leq \max(s) \}$.
 Note that $\min(s), \max(s) \in I(s)$.
 Suppose $X \in {\FIN}^{\[\omega\]}$.
 Define $N(X) = {\bigcup}_{i \in \omega}{I(X(i))} \subseteq \omega$.
\end{Def}
The classes ${\C}_{0}(\HHH)$ and ${\C}_{1}(\HHH)$ to be defined below will play an important role in Section \ref{sec:preservesel}, where it will be shown that the main forcing notion used in this paper preserves all selective ultrafilters.
To elaborate, Definition \ref{def:c0c1} introduces two new classes of selective ultrafilters.
These classes will be needed in Lemma \ref{lem:selstabgame}, which in turn will play a crucial role in the proof of Theorem \ref{thm:preservesselective}.
\begin{Def} \label{def:c0c1}
 For a stable ordered-union ultrafilter $\HHH$ on $\FIN$, define ${\C}_{0}(\HHH)$ to be the collection of all $\UUU$ such that:
 \begin{enumerate}[series=c0c1]
  \item
  $\UUU$ is a selective ultrafilter on $\omega$;
  \item
  for every $X \in {\FIN}^{\[\omega\]}$, if $\[X\] \in \HHH$, then there exists $Y \in {\FIN}^{\[\omega\]}$ such that $\[Y\] \in \HHH$, $\[Y\] \subseteq \[X\]$, and $N(Y) \notin \UUU$.
 \end{enumerate}
 Define ${\C}_{1}(\HHH)$ to be the collection of all $\UUU$ such that:
 \begin{enumerate}[resume=c0c1]
  \item
  $\UUU$ is a selective ultrafilter on $\omega$;
  \item
  $\forall \VVV\[\text{if} \ \VVV \ \text{is a selective ultrafilter on} \ \omega \ \text{and} \ \VVV \notin {\C}_{0}(\HHH)\text{,} \ \text{then} \ \UUU \; {\not\equiv}_{RK} \; \VVV\]$.
 \end{enumerate}
 Observe that ${\C}_{1}(\HHH) \subseteq {\C}_{0}(\HHH)$.
\end{Def}
\begin{Lemma} \label{lem:minmaxnotc0}
 If $\HHH$ is a stable ordered-union ultrafilter on $\FIN$, then ${\HHH}_{\min} \notin {\C}_{0}(\HHH)$ and ${\HHH}_{\max} \notin {\C}_{0}(\HHH)$.
\end{Lemma}
\begin{proof}
 Suppose $Y \in {\FIN}^{\[\omega\]}$ and that $\[Y\] \in \HHH$.
 Then ${f}_{\min}\[\[Y\]\] \in {\HHH}_{\min}$ and ${f}_{\max}\[\[Y\]\] \in {\HHH}_{\max}$.
 If $k \in {f}_{\min}\[\[Y\]\]$, then $k = \min(s)$, for some $s \in \[Y\]$, which means $k = \min\left(Y(j)\right)$, for some $j \in \omega$.
 Thus $k \in I\left(Y(j)\right) \subseteq N(Y)$, and so ${f}_{\min}\[\[Y\]\] \subseteq N(Y) \subseteq \omega$.
 Similarly, if $k \in {f}_{\max}\[\[Y\]\]$, then $k = \max(s)$, $s \in \[Y\]$, which means $k = \max\left(Y(j)\right)$, $j \in \omega$.
 Thus $k \in I\left(Y(j)\right) \subseteq N(Y)$, and so ${f}_{\max}\[\[Y\]\] \subseteq N(Y) \subseteq \omega$.
 Therefore, $N(Y) \in {\HHH}_{\min}$ and $N(Y) \in {\HHH}_{\max}$.
 Hence ${\HHH}_{\min}$ and ${\HHH}_{\max}$ fail Clause (2) of Definition \ref{def:c0c1}.
\end{proof}
\begin{Def} \label{def:utrees}
 For an ultrafilter $\HHH$ on $\FIN$, $T$ is called an \emph{$\HHH$-tree} if $T \subseteq {\FIN}^{< \omega}$ is a downwards closed subtree of ${\FIN}^{< \omega}$ such that $\emptyset \in T$ and $\forall \sigma \in T\[{\succc}_{T}(\sigma) \in \HHH\]$.
 
 For ultrafilters $\UUU$ on $\omega$ and $\HHH$ on $\FIN$, $T$ is called a \emph{$\pr{\UUU}{\HHH}$-tree} if $T \subseteq {\left( \omega \cup \FIN \right)}^{< \omega}$ is a downwards closed subtree of ${\left( \omega \cup \FIN \right)}^{< \omega}$ such that:
 \begin{enumerate}
  \item
  $\emptyset \in T$;
  \item
  $\forall \sigma \in T\[\lc \sigma \rc \ \text{is even} \implies {\succc}_{T}(\sigma) \in \UUU\]$;
  \item
  $\forall \sigma \in T\[\lc \sigma \rc \ \text{is odd} \implies {\succc}_{T}(\sigma) \in \HHH\]$.
 \end{enumerate}
\end{Def}
\begin{Lemma} \label{lem:htreebranch}
 Suppose $\HHH$ is a stable ordered-union ultrafilter on $\FIN$ and $T \subseteq {\FIN}^{< \omega}$ is an $\HHH$-tree.
 Then there exists $f \in \[T\]$ such that $\{ f(i): i \in \omega \} \in {\FIN}^{\[\omega\]}$, $\forall i < j < \omega\[\lb{f(i)}{f(j)}\]$, and $\[\{ f(i): i \in \omega \}\] \in \HHH$.
\end{Lemma}
\begin{proof}
 Define
 \begin{align*}
  {S}_{n} = \left\{ \sigma \in T: \forall i < j < \dom(\sigma)\[ \lb{\sigma(i)}{\sigma(j)} \] \ \text{and} \ \forall i \in \dom(\sigma)\[ \max\left(\sigma(i)\right) \leq n \] \right\},
 \end{align*}
 for $n \in \omega$.
 Note that each ${S}_{n}$ is finite and downwards closed.
 Hence ${A}_{n} = \bigcap \left\{ {\succc}_{T}(\sigma): \sigma \in {S}_{n} \right\} \in \HHH$ (here $\bigcap \emptyset$ is taken to be $\FIN$).
 Applying Lemma \ref{lem:selectivediag}, find $X \in {\FIN}^{\[\omega\]}$ such that $\[X\] \in \HHH$, $\[X\] \subseteq {A}_{0}$, and for each $i \in \omega$, $\[ \left\{ X(j): j > i \right\} \] \subseteq {A}_{\max\left( X(i) \right)}$.
 Define $f: \omega \rightarrow \FIN$ by setting $f(i)=X(i)$, for all $i \in \omega$.
 Since $X \in {\FIN}^{\[\omega\]}$, for all $i < j < \omega$, $f(i)=\lb{X(i)}{X(j)}=f(j)$.
 To complete the proof, it suffices to argue that $f \in \[T\]$, and for this, it suffices to verify that $\forall i \in \omega\[f\restrict i \in T\]$.
 This is done by induction on $i$.
 $f\restrict 0 = \emptyset \in T$ by the definition of an $\HHH$-tree.
 Since $\emptyset \in {S}_{0}$, ${A}_{0} \subseteq {\succc}_{T}(\emptyset)$.
 Therefore, $f(0)=X(0) \in X \subseteq \[X\] \subseteq {A}_{0} \subseteq {\succc}_{T}(\emptyset)$, whence $\langle f(0) \rangle \in T$.
 Now assume $i \in \omega$ and that $\sigma = \langle f(0), \dotsc, f(i) \rangle \in T$.
 Let $n = \max(f(i))$.
 Then $\sigma \in {S}_{n}$ and so ${A}_{n} \subseteq {\succc}_{T}(\sigma)$.
 Since $f(i+1)=X(i+1) \in {A}_{\max\left(X(i)\right)} = {A}_{\max\left(f(i)\right)} = {A}_{n} \subseteq {\succc}_{T}(\sigma)$, $\langle f(0), \dotsc, f(i), f(i+1) \rangle = {\sigma}^{\frown}{\langle f(i+1) \rangle} \in T$.
\end{proof}
\begin{Lemma} \label{lem:uhtreebranch}
 Suppose $\UUU$ is a selective ultrafilter on $\omega$ and $\HHH$ is a stable ordered-union ultrafilter on $\FIN$.
 Suppose $T \subseteq {\left( \omega \cup \FIN \right)}^{< \omega}$ is an $\pr{\UUU}{\HHH}$-tree.
 Suppose that $\UUU \in {\C}_{1}(\HHH)$.
 Then there exists $G \in \[T\]$ such that $\left\{ G(2i): i \in \omega \right\} \in \UUU$, $\left\{ G(2i+1): i \in \omega \right\} \in {\FIN}^{\[\omega\]}$, $\forall i < j < \omega\[\lb{G(2i+1)}{G(2j+1)}\]$, and
 \begin{align*}
  \[\left\{ G(2i+1): i \in \omega \right\}\] \in \HHH.
 \end{align*}
\end{Lemma}
\begin{proof}
 Since ${\HHH}_{\min}$ and ${\HHH}_{\max}$ are both selective ultrafilters on $\omega$, ${\HHH}_{\min}, {\HHH}_{\max} \notin {\C}_{0}(\HHH)$, and $\UUU \in {\C}_{1}(\HHH)$, then $\UUU \; {\not\equiv}_{RK} \; {\HHH}_{\min}$ and $\UUU \; {\not\equiv}_{RK} \; {\HHH}_{\max}$.
 The following claim will be established using this.
 \begin{Claim} \label{claim:uhtree1}
  Suppose $\seq{B}{n}{\in}{\omega}$ and $\seq{A}{n}{\in}{\omega}$ are sequences such that $\forall n \in \omega\[{B}_{n} \in \UUU \ \text{and} \ {A}_{n} \in \HHH\]$.
  Then there exist $B \in \UUU$ and $X \in {\FIN}^{\[\omega\]}$ such that:
  \begin{enumerate}
   \item
   $\[X\] \in \HHH$;
   \item
   if $\seq{n}{i}{\in}{\omega}$ is the strictly increasing enumeration of $B$, then for each $i \in \omega$, ${n}_{i} < \min\left(X(i)\right) \leq \max\left(X(i)\right) < {n}_{i+1}$, ${n}_{i+1} \in {B}_{\max\left(X(i)\right)}$, and $\[\left\{X(j): j \geq i \right\}\] \subseteq {A}_{{n}_{i}}$;
   \item
   $B \subseteq {B}_{0}$ and $\[X\] \subseteq {A}_{0}$.
  \end{enumerate}
 \end{Claim}
 \begin{proof}
  Define ${A}_{n}' = {\bigcap}_{m \leq n}{{A}_{m}} \in \HHH$ and ${B}_{n}' = {\bigcap}_{m \leq n}{{B}_{m}} \in \UUU$, for all $n \in \omega$.
  Since $\UUU$ is selective, there is $D \in \UUU$ such that $D \subseteq {B}_{0}'$ and $\forall i \in \omega\[{l}_{i+1} \in {B}_{{l}_{i}}'\]$, where $\seq{l}{i}{\in}{\omega}$ is the strictly increasing enumeration of $D$.
  By Lemma \ref{lem:selectivediag}, there exists $Z \in {\FIN}^{\[\omega\]}$ such that $\[Z\] \in \HHH$, $\[Z\] \subseteq {A}_{0}'$ and for each $i \in \omega$, $\[\left\{Z(j): j > i\right\}\] \subseteq {A}_{\max\left(Z(i)\right)}'$.
  Define $\varphi: \omega \rightarrow \omega$ and $\psi: \omega \rightarrow \omega$ as follows.
  For $m \in \omega$, $\varphi(m) = \min\{{l}_{i}: i \in \omega \ \text{and} \ {l}_{i} > m\}$, and
  \begin{align*}
   \psi(m) = \min\left\{ \min\left(Z(i)\right): i \in \omega \ \text{and} \ \min\left(Z(i)\right) > m \right\}.
  \end{align*}
  As ${\HHH}_{\min} \; {\not\leq}_{RK} \; \UUU$, there exists $K \in \UUU$ such that $\psi\[K\] \notin {\HHH}_{\min}$, and as $\UUU \; {\not\leq}_{RK} \; {\HHH}_{\max}$, there is $N \in {\HHH}_{\max}$ such that $\varphi\[N\] \notin \UUU$.
  Thus $E = D \cap K \cap \left( \omega \setminus \varphi\[N\] \right) \in \UUU$ and $L = \[Z\] \cap \{s \in \FIN: \min(s) \in \omega \setminus \psi\[K\]\} \cap \{s \in \FIN: \max(s) \in N\} \in \HHH$.
  
  Consider $m \in E$ and $s \in L$ with $m < \min(s)$.
  Let $i \in \omega$ be such that $\psi(m) = \min\left(Z(i)\right) > m$.
  As $s \in \[Z\]$, $m < \min(s) = \min\left(Z(j)\right)$, for some $j \in \omega$.
  By the minimality of $\psi(m)$, $\min\left(Z(i)\right) \leq \min\left(Z(j)\right)$, whence $i \leq j$.
  Since $m \in K$, $\min\left(Z(i)\right) = \psi(m) \in \psi\[K\]$, while $\min\left(Z(j)\right) = \min(s) \in \omega \setminus \psi\[K\]$.
  Therefore, $i < j$ and $s \in \[\left\{Z({j}^{\ast}): {j}^{\ast} > i \right\}\] \subseteq {A}_{\max\left(Z(i)\right)}'$.
  Since $m < \min\left(Z(i)\right) \leq \max\left(Z(i)\right)$, ${A}_{\max\left(Z(i)\right)}' \subseteq {A}_{m}$, and so $s \in {A}_{m}$.
  Thus $s \in {A}_{m}$, whenever $m \in E$ and $s \in L$ with $m < \min(s)$.
  
  Next, consider $m \in E$ and $s \in L$ with $\max(s) < m$.
  As $m \in D$, $m = {l}_{j}$, for some $j \in \omega$.
  There exists $i \in \omega$ such that $\varphi(\max(s)) = {l}_{i} > \max(s)$.
  Since $\max(s) < {l}_{j}$, the minimality of $\varphi(\max(s))$ implies that ${l}_{i} \leq {l}_{j}$, whence $i \leq j$.
  As $s \in L$, $\max(s) \in N$, and so ${l}_{i} = \varphi(\max(s)) \in \varphi\[N\]$, while ${l}_{j} = m \in \omega \setminus \varphi\[N\]$.
  Therefore, $i < j$, and so $m = {l}_{j} \in {B}_{{l}_{j-1}}'$.
  Since $i \leq j-1$, $\max(s) < {l}_{i} \leq {l}_{j-1}$, and so ${B}_{{l}_{j-1}}' \subseteq {B}_{\max(s)}$.
  Therefore, $m \in {B}_{\max(s)}$.
  Thus $m \in {B}_{\max(s)}$ whenever $m \in E$, $s \in L$, and $\max(s) < m$.
  
  Define $c: {L}^{\[2\]} \rightarrow 2$ as follows.
  Given $s, t \in L$ with $\lb{s}{t}$, $c(\{s, t\}) = 0$ if and only if $\exists m \in E\[\max(s) < m < \min(t)\]$.
  Find $R \in \HHH$ such that $R \subseteq L$ and $c$ is constant on ${R}^{\[2\]}$.
  This constant value cannot be $1$.
  To see this, argue by contradiction and suppose it is $1$.
  Fix some $s \in R$.
  Find $m \in E$ with $\max(s) < m$.
  Now find $t \in R$ with $m < \min(t)$.
  Then $c(\{s, t\}) = 0$ contradicting the supposition that the constant value is $1$.
  Hence $c$ is constantly $0$ on ${R}^{\[2\]}$.
  Pick ${n}_{0} \in E$ and let ${R}^{\ast} = \{s \in R: \min(s) > {n}_{0}\} \in \HHH$.
  Let $Y \in {\FIN}^{\[\omega\]}$ be such that $\[Y\] \in \HHH$ and $\[Y\] \subseteq {R}^{\ast}$.
  Since $\UUU \in {\C}_{0}(\HHH)$, there exists $X \in {\FIN}^{\[\omega\]}$ such that $\[X\] \in \HHH$, $\[X\] \subseteq \[Y\]$, and $N(X) \notin \UUU$.
  For each $i \in \omega$, define $J(i) = \{m \in \omega: \max\left(X(i)\right) < m < \min\left(X(i+1)\right)\}$, and let $J(X) = {\bigcup}_{i \in \omega}{J(i)}$.
  As $X \in {\FIN}^{\[\omega\]}$, it is clear that $\omega = \min(X(0)) \cup N(X) \cup J(X)$, and since $N(X) \cup \min(X(0)) \notin \UUU$, $J(X) \in \UUU$.
  Note that ${n}_{0} < \min\left(X(0)\right)$ because $X(0) \in {R}^{\ast}$.
  Let $P = E \cap J(X) \in \UUU$.
  If $i \in \omega$, then $\{X(i), X(i+1)\} \in {R}^{\[2\]}$ and so there exists $m \in E$ with $\max\left(X(i)\right) < m < \min\left(X(i+1)\right)$.
  By definition, $m \in J(i) \subseteq J(X)$, whence $m \in P$.
  Hence for every $i \in \omega$, $J(i) \cap P \neq \emptyset$.
  Since $P \subseteq J(X)$, since $\UUU$ is a Q-point, and since $\forall i < {i}^{\ast} < \omega\[J(i) \cap J({i}^{\ast}) = \emptyset\]$, there exists $Q \in \UUU$ so that $Q \subseteq P \subseteq J(X)$ and $\forall i \in \omega\[\lc Q \cap J(i) \rc = 1\]$.
  Let $B = Q \cup \{{n}_{0}\} \in \UUU$.
  Observe that $B \cap \min\left(X(0)\right) = \{{n}_{0}\}$.
  It now follows that the strictly increasing enumeration of $B$ has the form $\seq{n}{i}{\in}{\omega}$, where ${n}_{0}$ is the value chosen earlier and for each $i \in \omega$, ${n}_{i} < \min\left(X(i)\right) \leq \max\left(X(i)\right) < {n}_{i+1}$.
  Note that for any $i \in \omega$, ${n}_{i+1} \in Q \subseteq P \subseteq E$, $X(i) \in \[X\] \subseteq \[Y\] \subseteq {R}^{\ast} \subseteq R \subseteq L$, and $\max\left(X(i)\right) < {n}_{i+1}$, whence ${n}_{i+1} \in {B}_{\max\left(X(i)\right)}$.
  If $s \in \[\left\{X(j): j \geq i\right\}\]$, then $s \in \[X\] \subseteq L$ and $\min(s) \geq \min\left(X(i)\right) > {n}_{i}$.
  Further, ${n}_{i} \in E$ because ${n}_{0} \in E$ and $Q \subseteq E$.
  Hence $s \in {A}_{{n}_{i}}$.
  So $\[\left\{X(j): j \geq i \right\}\] \subseteq {A}_{{n}_{i}}$.
  This verifies (2).
  (1) is satisfied by the choice of $X$.
  Finally, (3) holds because $B \subseteq E \subseteq D \subseteq {B}_{0}' \subseteq {B}_{0}$ and $\[X\] \subseteq L \subseteq \[Z\] \subseteq {A}_{0}' \subseteq {A}_{0}$.
 \end{proof}
 To prove the lemma, define ${S}_{n}$ to be the collection of all $\sigma$ such that:
 \begin{enumerate}
  \item
  $\sigma \in T$, $\dom(\sigma) \leq n+1$;
  \item
  $\forall i \in \dom(\sigma)\[i \ \text{is even} \implies \sigma(i) \leq n \]$;
  \item
  $\forall i \in \dom(\sigma)\[i \ \text{is odd} \implies \max(\sigma(i)) \leq n \]$,
 \end{enumerate}
 for all $n \in \omega$.
 Then ${S}_{n}$ is finite, and so
 \begin{align*}
  {B}_{n} = \bigcap{\left\{ {\succc}_{T}(\sigma): \sigma \in {S}_{n} \ \text{and} \ \lc \sigma \rc \ \text{is even}\right\}} \in \UUU
 \end{align*}
 (with $\bigcap \emptyset$ interpreted as $\omega$) and ${A}_{n} = \bigcap{\left\{ {\succc}_{T}(\sigma): \sigma \in {S}_{n} \ \text{and} \ \lc \sigma \rc \ \text{is odd} \right\}} \in \HHH$ (with $\bigcap \emptyset$ interpreted as $\FIN$).
 Find $B$ and $X$ as in Claim \ref{claim:uhtree1}, and let $\seq{n}{i}{\in}{\omega}$ be the strictly increasing enumeration of $B$.
 Define $G: \omega \rightarrow \omega \cup \FIN$ by setting $G(2i) = {n}_{i}$ and $G(2i+1) = X(i)$, for all $i \in \omega$.
 Observe that by the properties of $B$ and $X$, for any $i \in \omega$, if $i$ is even, then $i \leq G(i)$, while if $i$ is odd, then $i \leq \max\left(G(i)\right)$.
 Since $X \in {\FIN}^{\[\omega\]}$, for all $i < j < \omega$, $G(2i+1)=\lb{X(i)}{X(j)}=G(2j+1)$.
 To complete the proof it suffices to show that $G \in \[T\]$, and for this it suffices to show that $G\restrict l \in T$, for all $l \in \omega$.
 Now $G \restrict 0 = \emptyset \in T$ by the definition of a $\pr{\UUU}{\HHH}$-tree.
 $\emptyset \in {S}_{0}$ and so $B \subseteq {B}_{0} \subseteq {\succc}_{T}(\emptyset)$, whence ${n}_{0} \in {\succc}_{T}(\emptyset)$.
 So ${\emptyset}^{\frown}{\langle {n}_{0} \rangle} = \langle {n}_{0} \rangle = \langle G(0) \rangle \in T$.
 Letting $\sigma = \langle {n}_{0} \rangle$, $\dom(\sigma) = 1 \leq {n}_{0}+1$, and so $\sigma \in {S}_{{n}_{0}}$.
 Therefore, $X(0) \in {A}_{{n}_{0}} \subseteq {\succc}_{T}(\sigma)$, and so ${\sigma}^{\frown}{\langle X(0) \rangle} = \langle {n}_{0}, X(0) \rangle = \langle G(0), G(1) \rangle \in T$.
 Now suppose $i \in \omega$ and that $\sigma = \langle {n}_{0}, X(0), \dotsc, {n}_{i}, X(i) \rangle = \langle G(0), G(1), \dotsc, G(2i), G(2i+1) \rangle \in T$.
 Then $\dom(\sigma) = 2i+2 = (2i+1)+1 \leq \max\left(G(2i+1)\right)+1 = \max\left(X(i)\right)+1$, and so $\sigma \in {S}_{\max\left(X(i)\right)}$.
 Therefore, ${n}_{i+1} \in {B}_{\max\left(X(i)\right)} \subseteq {\succc}_{T}(\sigma)$, and so $\langle {n}_{0}, X(0), \dotsc, {n}_{i}, X(i), {n}_{i+1} \rangle = {\sigma}^{\frown}{\langle {n}_{i+1} \rangle} \in T$.
 Now letting $\sigma = \langle {n}_{0}, X(0), \dotsc, {n}_{i}, X(i), {n}_{i+1} \rangle = \langle G(0), G(1), \dotsc, G(2i), G(2i+1), G(2i+2) \rangle$, $\dom(\sigma) = 2i+3 = (2i+2)+1 \leq G(2i+2)+1 = {n}_{i+1}+1$, and so $\sigma \in {S}_{{n}_{i+1}}$.
 Therefore, $X(i+1) \in {A}_{{n}_{i+1}} \subseteq {\succc}_{T}(\sigma)$, and so ${\sigma}^{\frown}{\langle X(i+1) \rangle} = \langle {n}_{0}, X(0), \dotsc, {n}_{i}, X(i), {n}_{i+1}, X(i+1) \rangle = \langle G(0), G(1), \dotsc, G(2i), G(2i+1), G(2i+2), G(2i+3) \rangle \in T$.
 Thus by induction, $G\restrict l \in T$, for all $l \in \omega$.
\end{proof}
\begin{Def} \label{def:games}
 Define the following two player games.
 \begin{enumerate}
  \item
  Let $\HHH$ be an ordered-union ultrafilter on $\FIN$.
  The \emph{stability game on $\HHH$}, denoted \emph{${\Game}^{\mathtt{Stab}}(\HHH)$}, is a two player perfect information game in which Players I and II alternatively choose sets ${A}_{i}$ and ${s}_{i}$ respectively, where ${A}_{i} \in \HHH$ and ${s}_{i} \in {A}_{i}$.
  Together they construct the sequence
  \begin{align*}
   {A}_{0}, {s}_{0}, {A}_{1}, {s}_{1}, \dotsc,
  \end{align*}
  where each ${A}_{i} \in \HHH$ has been played by Player I and ${s}_{i} \in {A}_{i}$ has been chosen by Player II in response.
  Player II wins if and only if $\forall i < j < \omega\[\lb{{s}_{i}}{{s}_{j}}\]$ and $\[\left\{ {s}_{i}: i < \omega \right\}\] \in \HHH$.
  \item
  Let $\UUU$ be an ultrafilter on $\omega$ and let $\HHH$ be an ordered-union ultrafilter on $\FIN$.
  The \emph{selectivity-stability game on $\pr{\UUU}{\HHH}$}, denoted \emph{${\Game}^{\mathtt{SelStab}}\left(\UUU, \HHH\right)$}, is a two player perfect information game in which Players I and II alternatively choose objects as follows: when $i$ is even, Player I chooses ${B}_{i} \in \UUU$ and Player II responds with ${n}_{i} \in {B}_{i}$; when $i$ is odd, Player I chooses ${A}_{i} \in \HHH$ and Player II responds with ${s}_{i} \in {A}_{i}$.
  Together they construct the sequence
  \begin{align*}
   {B}_{0}, {n}_{0}, {A}_{1}, {s}_{1}, \dotsc,
  \end{align*}
  where each ${B}_{2j} \in \UUU$ has been played by Player I and ${n}_{2j} \in {B}_{2j}$ has been played by Player II, and each ${A}_{2j+1} \in \HHH$ has been played by Player I and ${s}_{2j+1} \in {A}_{2j+1}$ has been played by Player II.
  Player II wins if and only if $\left\{ {n}_{2j}: j \in \omega \right\} \in \UUU$, $\forall i < j < \omega\[\lb{{s}_{2i+1}}{{s}_{2j+1}}\]$, and $\[\left\{ {s}_{2j+1}: j < \omega \right\}\] \in \HHH$.
  \item
  Let $\UUU$ and $\VVV$ be ultrafilters on $\omega$.
  The \emph{selectivity-selectivity game on $\pr{\UUU}{\VVV}$}, denoted \emph{${\Game}^{\mathtt{SelSel}}\left(\UUU, \VVV\right)$}, is a two player perfect information game in which Players I and II alternatively choose objects as follows: when $i$ is even, Player I chooses ${B}_{i} \in \UUU$ and Player II responds with ${n}_{i} \in {B}_{i}$; when $i$ is odd, Player I chooses ${A}_{i} \in \VVV$ and Player II responds with ${n}_{i} \in {A}_{i}$.
  Together they construct the sequence
  \begin{align*}
   {B}_{0}, {n}_{0}, {A}_{1}, {n}_{1}, \dotsc,
  \end{align*}
  where each ${B}_{2j} \in \UUU$ has been played by Player I and ${n}_{2j} \in {B}_{2j}$ has been played by Player II, and each ${A}_{2j+1} \in \VVV$ has been played by Player I and ${n}_{2j+1} \in {A}_{2j+1}$ has been played by Player II.
  Player II wins if and only if $\left\{ {n}_{2j}: j \in \omega \right\} \in \UUU$ and $\left\{ {n}_{2j+1}: j \in \omega \right\} \in \VVV$.
  \item
  Let $\UUU$ be an ultrafilter on $\omega$.
  The \emph{selectivity game on $\UUU$}, denoted \emph{${\Game}^{\mathtt{Sel}}\left(\UUU\right)$}, is a two player perfect information game in which Players I and II alternatively choose ${A}_{i}$ and ${n}_{i}$ respectively, where ${A}_{i} \in \UUU$ and ${n}_{i} \in {A}_{i}$.
  Together they construct the sequence
  \begin{align*}
   {A}_{0}, {n}_{0}, {A}_{1}, {n}_{1}, \dotsc,
  \end{align*}
  where each ${A}_{i} \in \UUU$ has been played by Player I and ${n}_{i} \in {A}_{i}$ has been chosen by Player II in response.
  Player II wins if and only if $\{{n}_{i}: i < \omega\} \in \UUU$.
 \end{enumerate}
\end{Def}
\begin{Lemma} \label{lem:stabgame}
 Suppose $\HHH$ is a stable ordered-union ultrafilter on $\FIN$.
 Then Player I does not have a winning strategy in ${\Game}^{\mathtt{Stab}}(\HHH)$.
\end{Lemma}
\begin{proof}
 Suppose for a contradiction that $\Sigma$ is a winning strategy for Player I in ${\Game}^{\mathtt{Stab}}(\HHH)$.
 Let $T$ be the collection of all $\sigma \in {\FIN}^{< \omega}$ for which there exists a function $\tau$ such that $\dom(\tau) = \dom(\sigma)$ and $\left\langle \pr{\tau(i)}{\sigma(i)}: i \in \dom(\sigma) \right\rangle$ is a partial run of ${\Game}^{\mathtt{Stab}}(\HHH)$ in which Player I has followed $\Sigma$.
 Observe that if $\sigma \in T$, then there is precisely one function $\tau$ which witnesses this.
 It is also easily seen that $T$ is an $\HHH$-tree.
 By Lemma \ref{lem:htreebranch}, find $f \in \[T\]$ such that the conclusions of that lemma are satisfied.
 Then $\forall n \in \omega\[f \restrict n \in T\]$.
 It follows from the argument for the uniqueness of the witness to $f \restrict n \in T$ for each $n \in \omega$ that there is a sequence $\seq{A}{i}{\in}{\omega}$ such that $\left\langle \pr{{A}_{i}}{f(i)}: i \in \omega \right\rangle$ is a run of ${\Game}^{\mathtt{Stab}}(\HHH)$ in which Player I has followed $\Sigma$.
 However Player II wins this run because $\forall i < j < \omega\[\lb{f(i)}{f(j)}\]$ and $\[\{f(i): i \in \omega\}\] \in \HHH$.
 This contradicts the supposition that $\Sigma$ is a winning strategy for Player I and concludes the proof.
\end{proof}
\begin{Lemma} \label{lem:selstabgame}
 Suppose $\UUU$ is a selective ultrafilter on $\omega$ and $\HHH$ is a stable ordered-union ultrafilter on $\FIN$.
 Suppose that $\UUU \in {\C}_{1}(\HHH)$.
 Then Player I does not have a winning strategy in ${\Game}^{\mathtt{SelStab}}\left(\UUU, \HHH\right)$.
\end{Lemma}
\begin{proof}
 This is similar to the proof of Lemma \ref{lem:stabgame}.
 Suppose for a contradiction that $\Sigma$ is a winning strategy for Player I in ${\Game}^{\mathtt{SelStab}}\left(\UUU, \HHH\right)$.
 Let $T$ be the collection of all $\sigma \in {\left( \omega \cup \FIN \right)}^{< \omega}$ for which there exists a function $\tau$ such that $\dom(\tau) = \dom(\sigma)$ and $\left\langle \pr{\tau(i)}{\sigma(i)}: i \in \dom(\sigma) \right\rangle$ is a partial run of ${\Game}^{\mathtt{SelStab}}\left(\UUU, \HHH\right)$ in which Player I has followed $\Sigma$.
 Observe that if $\sigma \in T$, then there is a unique function $\tau$ witnessing this.
 It is also easy to see that $T$ is an $\pr{\UUU}{\HHH}$-tree.
 By Lemma \ref{lem:uhtreebranch}, there exists $G \in \[T\]$ satisfying the conclusions of that lemma.
 Then $\forall n \in \omega\[G \restrict n \in T\]$.
 It follows from the argument for the uniqueness of the witness to $G \restrict n \in T$ for each $n \in \omega$ that there is a sequence $\seq{C}{i}{\in}{\omega} \in {\left( \UUU \cup \HHH \right)}^{\omega}$ such that $\left\langle \pr{{C}_{i}}{G(i)}: i \in \omega \right\rangle$ is a run of ${\Game}^{\mathtt{SelStab}}\left(\UUU, \HHH\right)$ in which Player I has followed $\Sigma$.
 However Player II wins this run of the game because $\{G(2i): i \in \omega\} \in \UUU$, $\forall i < j < \omega\[\lb{G(2i+1)}{G(2j+1)}\]$, and $\[\{G(2i+1): i \in \omega\}\] \in \HHH$.
 This contradicts the hypothesis that $\Sigma$ is a winning strategy for Player I and concludes the proof.
\end{proof}
It is worth noting here that the hypothesis that $\UUU \in {\C}_{1}(\HHH)$ is necessary for Lemma \ref{lem:selstabgame}.
Under $\CH$, weaker hypotheses such as $\UUU \; {\not\equiv}_{RK} \; {\HHH}_{\min}$ and $\UUU \; {\not\equiv}_{RK} \; {\HHH}_{\max}$ are provably not sufficient for the conclusion.
The following two results are well-known from the literature.
\begin{Lemma}[Shelah~\cite{PIF}] \label{lem:selsel}
 Let $\UUU$ and $\VVV$ be selective ultrafilters on $\omega$.
 If $\UUU \; {\not\equiv}_{RK} \; \VVV$, then Player I does not have a winning strategy in ${\Game}^{\mathtt{SelSel}}\left(\UUU, \VVV\right)$.
\end{Lemma}
\begin{Lemma}[Galvin; McKenzie] \label{lem:sel}
 Suppose $\UUU$ is a selective ultrafilter on $\omega$.
 Then Player I does not have a winning strategy in ${\Game}^{\mathtt{Sel}}\left(\UUU\right)$.
\end{Lemma}
\section{The partial order} \label{sec:partial}
This section introduces the main forcing notion used in this paper.
For a fixed stable ordered-union ultrafilter $\HHH$, a partial order $\PP(\HHH)$ is defined.
A key idea in the definition of $\PP(\HHH)$ is the combinatorial notion of a $\pr{k}{s}$-big set introduced in Definition \ref{def:skbig}.
Basic properties of $\PP(\HHH)$ will be established in this section.
In subsequent sections, it will be shown that $\PP(\HHH)$ is proper, $\BS$-bounding, destroys $\HHH$ and preserves all selective ultrafilters.
\begin{Def} \label{def:s-}
 For $s \in \FIN$, define ${s}^{-} = s \setminus \{\max(s)\}$.
 Note that ${s}^{-} \subseteq \max(s)$.
\end{Def}
\begin{Def} \label{def:skbig}
 Let $k \leq l < \omega$ and $s \in \FIN$ with $l = \max(s)$.
 $A \subseteq {2}^{\Pset(l)}$ is \emph{$\pr{k}{s}$-big} if for every $\sigma: \Pset(k) \rightarrow 2$, there exists $\tau \in A$ such that $\forall u \in \Pset(k)\[\sigma(u) = \tau(u \cup {s}^{-})\]$.
\end{Def}
\begin{remark} \label{rem:skbig}
 When $\min(s) < k$, $\pr{k}{s}$-big sets generally do not exist.
 So even though the definition of a $\pr{k}{s}$-big set makes sense when $\min(s) < k$, it will generally not be used in this situation.
 Note also that $\pr{k}{s}$-big sets are non-empty.
\end{remark}
\begin{Lemma} \label{lem:bigmonotone}
 Suppose $k \leq k' \leq l < \omega$ and $s \in \FIN$ with $l = \max(s)$.
 If $A \subseteq {2}^{\Pset(l)}$ is $\pr{k'}{s}$-big, then $A$ is also $\pr{k}{s}$-big.
\end{Lemma}
\begin{proof}
 Let $\sigma: \Pset(k) \rightarrow 2$ be given.
 Define $\sigma': \Pset(k') \rightarrow 2$ by $\sigma'(u) = \sigma(u \cap k)$, for all $u \in \Pset(k')$.
 Note that if $u \in \Pset(k)$, then $\sigma'(u) = \sigma(u \cap k) = \sigma(u)$.
 Since $A$ is $\pr{k'}{s}$-big, there exists $\tau \in A$ such that $\forall u \in \Pset(k')\[\sigma'(u) = \tau(u \cup {s}^{-})\]$.
 Then $\forall u \in \Pset(k)\[\sigma(u) = \sigma'(u) = \tau(u \cup {s}^{-})\]$.
\end{proof}
\begin{Def} \label{def:trees}
 Define $\TT = {\bigcup}_{l \in \omega}{{\prod}_{k \in l}{{2}^{\Pset(k)}}}$.
 Then $\pr{\TT}{\subsetneq}$ is a tree and recall from Definition \ref{def:tree} that for any subtree $T \subseteq \TT$,
 \begin{align*}
  \[T\] = \left\{ F \in \displaystyle\prod_{k \in \omega}{{2}^{\Pset(k)}}: \forall l \in \omega\[F\restrict l \in T\] \right\}.
 \end{align*}
\end{Def}
Note that $\[T\]$ for a subtree $T \subseteq \TT$ and $\[A\]$ for a subset $A \subseteq \FIN$ are entirely different objects.
This should not be a source of confusion.
\begin{Def} \label{def:PH}
 Let $\HHH$ be a stable ordered-union ultrafilter on $\FIN$.
 $p$ is called an \emph{$\HHH$-condition} if $p = {T}_{p} \subseteq \TT$ is a subtree such that the following hold:
 \begin{enumerate}
  \item
  $\emptyset \in {T}_{p}$;
  \item
  $\forall f \in {T}_{p} \forall \dom(f) \leq n < \omega \exists g \in {T}_{p}\[f \subseteq g \wedge n \leq \dom(g)\]$;
  \item
  for each $k \in \omega$, ${H}_{p, k} \in \HHH$, where ${H}_{p, k} = $
  \begin{align*}
   \left\{ s \in \FIN: k \leq \max(s) \wedge \forall f \in \lv{{T}_{p}}{\max(s)}\[{\succc}_{{T}_{p}}(f) \ \text{is} \ \pr{k}{s}\text{-big}\] \right\}.
  \end{align*}
 \end{enumerate}
 Let $\PP(\HHH) = \left\{p: p \ \text{is an} \ \HHH\text{-condition} \right\}$.
 Define $q \leq p$ if and only if ${T}_{q} \subseteq {T}_{p}$, for all $p, q \in \PP(\HHH)$.
\end{Def}
From now until the end of Section \ref{sec:preservesel}, $\HHH$ will be a fixed stable ordered-union ultrafilter on $\FIN$.
\begin{Lemma} \label{lem:infintelymanyk}
 Suppose $p = {T}_{p} \subseteq \TT$ is a subtree satisfying (1) and (2) of Definition \ref{def:PH}.
 Suppose there are infinitely many $k \in \omega$ for which ${H}_{p, k} \in \HHH$.
 Then $p$ is an $\HHH$-condition.
\end{Lemma}
\begin{proof}
 Let $k \in \omega$ be given.
 Choose $k \leq k' < \omega$ for which ${H}_{p, k'} \in \HHH$.
 Then $\forall s \in {H}_{p, k'}\[k \leq \max(s)\]$.
 Consider $s \in {H}_{p, k'}$ and $f \in \lv{{T}_{p}}{\max(s)}$.
 Then $k \leq k' \leq \max(s) < \omega$, $s \in \FIN$, and ${\succc}_{{T}_{p}}(f) \subseteq {2}^{\Pset(\max(s))}$ is $\pr{k'}{s}$-big.
 By Lemma \ref{lem:bigmonotone}, ${\succc}_{{T}_{p}}(f)$ is also $\pr{k}{s}$-big.
 Thus
 \begin{align*}
  \forall s \in {H}_{p, k'} \forall f \in \lv{{T}_{p}}{\max(s)}\[{\succc}_{{T}_{p}}(f) \ \text{is} \ \pr{k}{s}\text{-big}\].
 \end{align*}
 Thus ${H}_{p, k} \supseteq {H}_{p, k'}$ and so ${H}_{p, k} \in \HHH$.
\end{proof}
\begin{Lemma} \label{lem:stepranscondition}
 Suppose $p \in \PP(\HHH)$.
 For each $H \subseteq \FIN$ and $m \leq k < \omega$, define $H\[m, k\]$ to be
 \begin{align*}
  \left\{ {s}^{-}: s \in H \ \text{and} \ \max(s) = k \ \text{and} \ \min(s) \geq m \right\}.
 \end{align*}
 There exists $H \in \HHH$ such that for each $l \in \omega$, for all but finitely many $k \in {f}_{\max}\[H\]$, $l \leq k$ and
 \begin{align*}
  \tag{${\ast}_{l, k}$} \forall f \in \lv{{T}_{p}}{k} &\forall g \in \trp{2}{\Pset(l)}{H\[l, k\]}\\
  &\exists \tau \in {\succc}_{{T}_{p}}(f)\forall x \subseteq l\forall u \in H\[l, k\]\[\tau(x \cup u) = g(u)(x)\].
 \end{align*}
\end{Lemma}
\begin{proof}
 For $n \in \omega$, let ${A}_{n} = {H}_{p, n+1} \in \HHH$.
 Applying Lemma \ref{lem:selectivediag}, find $X \in {\FIN}^{\[\omega\]}$ such that $\[X\] \in \HHH$, $\[X\] \subseteq {A}_{0}$, and $\forall i \in \omega \[ \[\left\{ X(j): j > i \right\}\] \subseteq {A}_{\max\left(X(i)\right)} \]$.
 Let $H = \[X\]$ and fix $l \in \omega$.
 Let ${i}_{0} \in \omega$ be such that $\min\left(X({i}_{0})\right) > l$, and consider any $k \in {f}_{\max}\[H\]$ with $k > \max\left(X({i}_{0})\right)$.
 Note $l < \min\left(X({i}_{0})\right) \leq \max\left(X({i}_{0})\right) < k$.
 Since $k \in {f}_{\max}\[H\]$ and $\max\left(X({i}_{0})\right) < k$, there is a unique $i \in \omega$ with $\max\left(X(i+1)\right) = k$.
 It follows that if $s \in H$ with $\max(s) = k$, and $l \leq \min(s)$, then
 \begin{align*}
  s = \left( s \setminus X(i+1) \right) \cup X(i+1),
 \end{align*}
 $\left( s \setminus X(i+1) \right) \subseteq \max\left(X(i)\right)+1$, and $\forall x \in \left( s \setminus X(i+1) \right) \[l \leq x\]$.
 By the choice of $X$, $X(i+1) \in {A}_{\max\left(X(i)\right)} = {H}_{p, \max\left(X(i)\right)+1}$.
 Consider any $f \in \lv{{T}_{p}}{k}$.
 Then ${\succc}_{{T}_{p}}(f)$ is $\pr{\max\left(X(i)\right)+1}{X(i+1)}$-big.
 Suppose $g: H\[l, k\] \rightarrow {2}^{\Pset(l)}$ is given.
 Define $\sigma: \Pset\left(\max\left(X(i)\right)+1\right) \rightarrow 2$ as follows.
 Given $w \in \Pset\left(\max\left(X(i)\right)+1\right)$, let ${s}_{w} = \left( w \setminus l \right) \cup X(i+1)$
 \begin{align*}
  \sigma(w) = \begin{cases}
               g\left({s}^{-}_{w}\right)(w \cap l) \ &\text{if} \ {s}_{w} \in H \ \text{and} \ \max({s}_{w}) = k \ \text{and} \ l \leq \min({s}_{w}),\\
               0 \ &\text{otherwise}.
              \end{cases}
 \end{align*}
 Find $\tau \in {\succc}_{{T}_{p}}(f)$ such that
 \begin{align*}
  \forall w \in \Pset\left( \max\left(X(i)\right)+1 \right)\[ \sigma(w) = \tau\left(w \cup {\left(X(i+1)\right)}^{-} \right) \].
 \end{align*}
 Fix $x \subseteq l$ and $u \in H\[l, k\]$.
 Then $u = {s}^{-}$, where $s \in H$, $\max(s) = k$, and $l \leq \min(s)$.
 Further, $\left( s \setminus X(i+1) \right) \subseteq \max\left(X(i)\right)+1$, and since $\max\left(X({i}_{0})\right) < \max\left(X(i+1)\right)$, ${i}_{0} \leq i$, whence $x \subseteq l < \min\left(X({i}_{0})\right) \leq \max\left(X({i}_{0})\right) \leq \max\left(X(i)\right) < \max\left(X(i)\right)+1$.
 Therefore, $w = x \cup \left( s \setminus X(i+1) \right) \in \Pset\left( \max\left(X(i)\right)+1 \right)$.
 Further, if $z \in \left( s \setminus X(i+1) \right)$, then $l \leq z$, whence $z \in w \setminus l$.
 Therefore, $w \setminus l = \left( s \setminus X(i+1) \right)$ and $w \cap l = x$.
 Therefore, ${s}_{w} = \left( s \setminus X(i+1) \right) \cup X(i+1) = s$, and $u = {s}^{-} = {s}^{-}_{w} = \left( s \setminus X(i+1) \right) \cup {\left(X(i+1)\right)}^{-}$.
 Therefore, $\tau(x\cup u) = \tau\left(x \cup \left( s \setminus X(i+1) \right) \cup {\left(X(i+1)\right)}^{-}\right) = \tau\left(w \cup {\left(X(i+1)\right)}^{-} \right) = \sigma(w) = g\left(u\right)(x)$, precisely as needed.
\end{proof}
\begin{Lemma} \label{lem:one}
 $\TT \in \PP(\HHH)$.
\end{Lemma}
\begin{proof}
 All of the requirements with the possible exception of (3) of Definition \ref{def:PH} are clear.
 To verify this, given $k \in \omega$, define $H = \{s \in \FIN: \min(s) > k\}$, and note that $H \in \HHH$.
 Clearly, $\forall s \in H\[k < \min(s) \leq \max(s)\]$.
 Consider $s \in H$.
 Let $l = \max(s)$ and consider $f \in \lv{\TT}{l}$.
 Then $A = {\succc}_{\TT}(f) = {2}^{\Pset(l)}$.
 To see that $A$ is $\pr{k}{s}$-big, consider $\sigma: \Pset(k) \rightarrow 2$.
 Define $\tau: \Pset(l) \rightarrow 2$ by $\tau(t) = \sigma(t \cap k)$, for all $t \in \Pset(l)$
 Then $\tau \in A$.
 Note that for any $u \in \Pset(k)$, $\left( u \cup {s}^{-} \right) \cap k = u$.
 So for every $u \in \Pset(k)$, by definition, $\tau\left( u \cup {s}^{-} \right) = \sigma\left(\left( u \cup {s}^{-} \right) \cap k\right) = \sigma(u)$.
 Thus $H \subseteq {H}_{\TT, k}$, and so ${H}_{\TT, k} \in \HHH$.
\end{proof}
\begin{Lemma} \label{def:Tf}
 Suppose $p \in \PP(\HHH)$.
 Then for any $f \in {T}_{p}$, $q = {T}_{p}\langle f \rangle \in \PP(\HHH)$ and $q \leq p$.
\end{Lemma}
\begin{proof}
 Clearly, ${T}_{p}\langle f \rangle \subseteq {T}_{p} \subseteq \TT$ and ${T}_{p}\langle f \rangle$ is a subtree of $\TT$.
 $\emptyset \in {T}_{p}\langle f \rangle$ because $\emptyset \in {T}_{p}$ and $\emptyset \subseteq f$.
 Next, suppose $e \in {T}_{p}\langle f \rangle$ and $\dom(e) \leq n < \omega$.
 Let $h = e \cup f$.
 Then $h \in \{e, f\}$ and $\dom(h) = \max\{\dom(e), \dom(f)\}$.
 As $h \in {T}_{p}$ and $\dom(h) \leq m = \max\{\dom(h), n\} < \omega$, there exists $g \in {T}_{p}$ with $h \subseteq g$ and $m \leq \dom(g)$.
 $f \subseteq h \subseteq g$, so $g \in {T}_{p}\langle f \rangle$.
 And $e \subseteq h \subseteq g$ and $n \leq m \leq \dom(g)$, as required.
 Finally, fix $k \in \omega$.
 Let $H = \{s \in {H}_{p, k}: \min(s) > \dom(f)\} \in \HHH$.
 Then $\forall s \in H\[k \leq \max(s)\]$.
 Suppose $s \in H$ and $e \in \lv{{T}_{p}\langle f \rangle}{\max(s)}$.
 Then $s \in {H}_{p, k}$, $e \in \lv{{T}_{p}}{\max(s)}$, and $\dom(e) = \max(s) \geq \min(s) > \dom(f)$, whence $f \subseteq e$.
 Thus ${\succc}_{{T}_{p}\langle f \rangle}(e) = {\succc}_{{T}_{p}}(e)$ is $\pr{k}{s}$-big.
 Thus $H \subseteq {H}_{q, k}$, and so ${H}_{q, k} \in \HHH$.
 Thus $q \in \PP(\HHH)$ and $q \leq p$.
\end{proof}
\begin{Lemma} \label{lem:amalgam}
 Let $p \in \PP(\HHH)$.
 Let $l \in \omega$, $1 \leq m < \omega$, and ${e}_{1}, \dotsc, {e}_{m} \in \lv{{T}_{p}}{l}$.
 If ${p}_{1}, \dotsc, {p}_{m} \in \PP(\HHH)$ are such that $\forall 1 \leq i \leq m\[{p}_{i} \leq {T}_{p}\langle {e}_{i} \rangle\]$, then $q = {T}_{q} = {\bigcup}_{1 \leq i \leq m}{{T}_{{p}_{i}}} \in \PP(\HHH)$, for each $1 \leq i \leq m$, ${p}_{i} \leq q$, and $q \leq p$.
\end{Lemma}
\begin{proof}
 For each $1 \leq i \leq m$, ${T}_{{p}_{i}} \subseteq {T}_{p}\langle {e}_{i} \rangle \subseteq {T}_{p}$.
 So ${T}_{q} \subseteq {T}_{p} \subseteq \TT$ and clearly ${T}_{q}$ is a subtree of $\TT$.
 $\emptyset \in {T}_{q}$ because $\emptyset \in {T}_{{p}_{1}}$.
 Next, suppose $f \in {T}_{q}$ and $\dom(f) \leq n < \omega$.
 Then $f \in {T}_{{p}_{i}}$ for some $1 \leq i \leq m$, and so there exists $g \in {T}_{{p}_{i}}$ such that $f \subseteq g$ and $n \leq \dom(g)$, whence $g \in {T}_{q}$ as well.
 Next, for each $k \in \omega$, define $H = {H}_{{p}_{1}, k} \cap \dotsb \cap {H}_{{p}_{m}, k} \in \HHH$.
 For every $s \in H$, $s \in \FIN$ and $k \leq \max(s) < \omega$ because $s \in {H}_{{p}_{1}, k}$.
 Consider $f \in \lv{{T}_{q}}{\max(s)}$.
 Then for some $1 \leq i \leq m$, $f \in \lv{{T}_{{p}_{i}}}{\max(s)}$.
 Since ${\succc}_{{T}_{{p}_{i}}}(f)$ is $\pr{k}{s}$-big and ${\succc}_{{T}_{{p}_{i}}}(f) \subseteq {\succc}_{{T}_{q}}(f) \subseteq {2}^{\Pset(\max(s))}$, ${\succc}_{{T}_{q}}(f)$ is also $\pr{k}{s}$-big.
 Thus $H \subseteq {H}_{q, k}$, and so ${H}_{q, k} \in \HHH$.
 Therefore $q \in \PP(\HHH)$.
 Finally, since ${T}_{q} \subseteq {T}_{p}$, $q \leq p$, and for all $1 \leq i \leq m$, since ${T}_{{p}_{i}} \subseteq {T}_{q}$, ${p}_{i} \leq q$.
\end{proof}
\begin{Lemma} \label{lem:fusion}
 Suppose $\seq{p}{i}{\in}{\omega}$ and $X$ satisfy:
 \begin{enumerate}
  \item
  $\forall i \in \omega\[{p}_{i} \in \PP(\HHH)\]$ and $\forall i \in \omega\[{p}_{i+1} \leq {p}_{i}\]$;
  \item
  $X \in {\FIN}^{\[\omega\]}$ with $\[X\] \in \HHH$;
  \item
  for each $i \in \omega$, $X(i+1) \in {H}_{{p}_{i+1}, \max\left(X(i)\right)+1}$;
  \item
  for each $i \leq j < \omega$ and for each $e \in \lv{{T}_{{p}_{j}}}{\max\left(X(i)\right)}$, ${\succc}_{{T}_{{p}_{i}}}(e) \subseteq {\succc}_{{T}_{{p}_{j}}}(e)$.
 \end{enumerate}
 Then $q = {T}_{q} = {\bigcap}_{i \in \omega}{{T}_{{p}_{i}}} \in \PP(\HHH)$.
\end{Lemma}
\begin{proof}
 For ease of notation in this proof, the symbols ${T}_{i}$ will replace ${T}_{{p}_{i}}$, ${l}_{i}$ will denote $\max\left(X(i)\right)$, and ${H}_{i, k}$ will be used for ${H}_{{p}_{i}, k}$, for all $i \in \omega$ and $k \in \omega$.
 Clearly, ${T}_{q} \subseteq {T}_{0} \subseteq \TT$ and ${T}_{q}$ is a subtree of $\TT$.
 Since $\emptyset \in {T}_{i}$ for every $i \in \omega$, $\emptyset \in {T}_{q}$.
 Next, suppose $f \in {T}_{q}$ and that $\dom(f) \leq n < \omega$.
 Let $\FFF = \{e \in {T}_{0}: \dom(e) = n \wedge f \subseteq e\}$.
 Then $\FFF$ is a finite set.
 For each $i \in \omega$, $f \in {T}_{i}$, and so there exists $g \in {T}_{i}$ with $f \subseteq g$ and $n \leq \dom(g)$.
 Then $g \restrict n \in {T}_{i}$, $\dom(g \restrict n) = n$, and $f \subseteq g \restrict n$, whence $g \restrict n \in \FFF \cap {T}_{i}$.
 Thus $\forall i \in \omega \exists {e}_{i} \in \FFF\[{e}_{i} \in {T}_{i}\]$.
 As $\FFF$ is finite, there exists $e \in \FFF$ such that $\existsinf i \in \omega\[e \in {T}_{i}\]$.
 Since $\forall i \leq j < \omega\[{T}_{j} \subseteq {T}_{i}\]$, it follows that $e \in {T}_{q}$.
 Therefore ${T}_{q}$ satisfies (2) of Definition \ref{def:PH}.
 To see that ${T}_{q}$ satisfies (3) of Definition \ref{def:PH}, fix $k \in \omega$.
 Choose ${i}_{0} \in \omega$ such that $k < \min\left(X({i}_{0})\right)$.
 Let $H = \left\{ s \in \[X\]: \min(s) > \max\left(X({i}_{0})\right) \right\} \in \HHH$.
 Suppose $s \in H$.
 Then $s \in \FIN$ and $k < \min\left(X({i}_{0})\right) \leq \max\left(X({i}_{0})\right) < \min(s) \leq \max(s)$.
 Let $l = \max(s)$.
 Then $l = \max\left(X({i}_{1})\right)$, for some ${i}_{1} \in \omega$.
 Since $\max\left(X({i}_{0})\right) < \max\left(X({i}_{1})\right)$, ${i}_{1} = i+1$, where ${i}_{0} \leq i < \omega$.
 Note that $s \setminus X(i+1) \subseteq \max\left(X(i)\right)+1$, $\left( s \setminus X(i+1) \right) \cap k = \emptyset$, $s = s \setminus X(i+1) \cup X(i+1)$, and ${s}^{-} = \left( s \setminus X(i+1) \right) \cup {\left( X(i+1) \right)}^{-}$.
 Further, if $u \subseteq k$, then for any $x \in u$, $x < k < \max\left(X({i}_{0})\right) \leq \max\left(X(i)\right) < \max\left(X(i)\right)+1$, and so $u \subseteq \max\left(X(i)\right)+1$.
 Therefore $u \cup \left( s \setminus X(i+1) \right) \subseteq \max\left(X(i)\right)+1$ and $\left( u \cup \left( s \setminus X(i+1) \right) \right) \cap k = u$.
 By hypothesis, $X(i+1) \in {H}_{i+1, \max\left(X(i)\right)+1}$.
 Now, fix $f \in \lv{{T}_{q}}{l}$.
 It needs to be seen that ${\succc}_{{T}_{q}}(f)$ is $\pr{k}{s}$-big.
 To this end, let $\sigma: \Pset(k) \rightarrow 2$ be fixed.
 Since $f \in \lv{{T}_{i+1}}{\max\left(X(i+1)\right)}$, ${\succc}_{{T}_{i+1}}(f)$ is $\pr{\max\left(X(i)\right)+1}{X(i+1)}$-big.
 Now define ${\sigma}^{\ast}: \Pset\left( \max\left(X(i)\right)+1 \right) \rightarrow 2$ by setting ${\sigma}^{\ast}(w) = \sigma(w \cap k)$, for all $w \subseteq \max\left(X(i)\right)+1$.
 Find $\tau \in {\succc}_{{T}_{i+1}}(f)$ such that for each $w \subseteq \max\left(X(i)\right)+1$, ${\sigma}^{\ast}(w) = \tau\left( w \cup {\left( X(i+1) \right)}^{-} \right)$.
 Consider any $i+1 \leq j < \omega$.
 Then $f \in \lv{{T}_{j}}{\max\left(X(i+1)\right)}$, and the hypothesis is that ${\succc}_{{T}_{i+1}}(f) \subseteq {\succc}_{{T}_{j}}(f)$.
 Thus $\tau \in {\succc}_{{T}_{j}}(f)$, and so ${f}^{\frown}{\langle \tau \rangle} \in {T}_{j}$.
 Hence $\forall i+1 \leq j < \omega\[{f}^{\frown}{\langle \tau \rangle} \in {T}_{j}\]$, whence ${f}^{\frown}{\langle \tau \rangle} \in {T}_{q}$.
 Therefore $\tau \in {\succc}_{{T}_{q}}(f)$.
 Take $u \subseteq k$.
 Then $u \cup \left( s \setminus X(i+1) \right) = w \subseteq \max\left(X(i)\right)+1$ and $w \cap k = u$.
 So $\sigma(u) = \sigma(w \cap k) = {\sigma}^{\ast}(w) = \tau\left( w \cup {\left( X(i+1) \right)}^{-} \right) = \tau\left( u \cup \left( s \setminus X(i+1) \right) \cup {\left( X(i+1) \right)}^{-} \right) = \tau\left( u \cup {s}^{-} \right)$.
 This proves that ${\succc}_{{T}_{q}}(f)$ is $\pr{k}{s}$-big.
 Thus $H \subseteq {H}_{q, k}$, and so ${H}_{q, k} \in \HHH$.
 This concludes the proof that $q \in \PP(\HHH)$.
\end{proof}
\begin{Def} \label{def:GcondH}
 Define the following two player game called the \emph{$\HHH$-condition game} and denoted \emph{${\Game}^{\mathtt{Cond}}\left(\HHH\right)$}.
 Players I and II alternatively choose $\pr{{p}_{i}}{{A}_{i}}$ and ${s}_{i}$ respectively, where
 \begin{enumerate}
  \item
  ${p}_{i} \in \PP(\HHH)$ and ${A}_{i} \in \HHH$;
  \item
  ${s}_{i} \in {A}_{i}$;
  \item
  there exists $\left\langle {p}_{i, e}: e \in \lv{{T}_{{p}_{i}}}{\max({s}_{i})+1} \right\rangle$ such that
  \begin{align*}
   \forall e \in \lv{{T}_{{p}_{i}}}{\max({s}_{i})+1}\[{p}_{i, e} \leq {T}_{{p}_{i}}\langle e \rangle\]
  \end{align*}
  and ${p}_{i+1} = {T}_{{p}_{i+1}} = \bigcup\left\{{T}_{{p}_{i, e}}: e \in \lv{{T}_{{p}_{i}}}{\max({s}_{i})+1}\right\}$.
 \end{enumerate}
 Together they construct the sequence
 \begin{align*}
  \pr{{p}_{0}}{{A}_{0}}, {s}_{0}, \pr{{p}_{1}}{{A}_{1}}, {s}_{1}, \dotsc,
 \end{align*}
 where each $\pr{{p}_{i}}{{A}_{i}}$ has been played by Player I and ${s}_{i}$ has been chosen by Player II in response subject to Conditions (1)--(3).
 Player II wins if and only if $\forall i < j < \omega\[\lb{{s}_{i}}{{s}_{j}}\]$, $\[\left\{{s}_{i}: i < \omega\right\}\] \in \HHH$, and $q = {T}_{q} = {\bigcap}_{i \in \omega}{{T}_{{p}_{i}}} \in \PP(\HHH)$.
\end{Def}
\begin{Lemma} \label{lem:condgame}
 Player I does not have a winning strategy in ${\Game}^{\mathtt{Cond}}\left(\HHH\right)$.
\end{Lemma}
\begin{proof}
 Suppose for a contradiction that $\Sigma$ is a winning strategy for Player I in ${\Game}^{\mathtt{Cond}}\left(\HHH\right)$.
 Define $\Pi$ and $\Phi$ such that:
 \begin{enumerate}
  \item
  $\Pi$ is a strategy for Player I in ${\Game}^{\mathtt{Stab}}\left(\HHH\right)$;
  \item
  for each $n \in \omega$, if $\left\langle \pr{{B}_{i}}{{s}_{i}}: i \leq n \right\rangle$ is a partial run of ${\Game}^{\mathtt{Stab}}\left(\HHH\right)$ in which Player I has followed $\Pi$, then $\Phi\left(\left\langle \pr{{B}_{i}}{{s}_{i}}: i \leq n \right\rangle\right) = \left\langle \pr{{p}_{i}}{{A}_{i}}: i \leq n \right\rangle$ and
  \begin{align*}
   \left\langle \pr{\pr{{p}_{i}}{{A}_{i}}}{{s}_{i}}: i \leq n \right\rangle
  \end{align*}
  is a partial play of ${\Game}^{\mathtt{Cond}}\left(\HHH\right)$ in which Player I has followed $\Sigma$ and it has the property that $\forall i < n\[ {B}_{i+1} = {A}_{i+1} \cap {H}_{{p}_{i+1}, \max({s}_{i})+1} \]$;
  \item
  for each $n \in \omega$, if $\left\langle \pr{{B}_{i}}{{s}_{i}}: i \leq n+1 \right\rangle$ is a partial run of ${\Game}^{\mathtt{Stab}}\left(\HHH\right)$ in which Player I has followed $\Pi$, then
  \begin{align*}
   \Phi\left( \left\langle \pr{{B}_{i}}{{s}_{i}}: i \leq n \right\rangle \right) = \Phi\left( \left\langle \pr{{B}_{i}}{{s}_{i}}: i \leq n+1 \right\rangle \right)\restrict n+1.
  \end{align*}
 \end{enumerate}
 $\Pi$ and $\Phi$ will be defined inductively.
 Let $\Sigma(\emptyset) = \pr{{p}_{0}}{{A}_{0}}$ define $\Pi(\emptyset) = {B}_{0} = {A}_{0}$.
 As ${B}_{0} \in \HHH$, this is a valid move for Player I in ${\Game}^{\mathtt{Stab}}\left(\HHH\right)$.
 Note that every partial run of ${\Game}^{\mathtt{Stab}}\left(\HHH\right)$ of length $1$ in which Player I has followed $\Pi$ will have the form $\pr{{B}_{0}}{{s}_{0}}$, where ${s}_{0} \in {B}_{0} = {A}_{0}$.
 For any such $\pr{{B}_{0}}{{s}_{0}}$, define $\Phi(\pr{{B}_{0}}{{s}_{0}}) = \langle \pr{{p}_{0}}{{A}_{0}} \rangle$, and note that $\pr{\pr{{p}_{0}}{{A}_{0}}}{{s}_{0}}$ is a partial run of ${\Game}^{\mathtt{Cond}}\left(\HHH\right)$ in which Player I has followed $\Sigma$.
 Now suppose $n \in \omega$, $\left\langle \pr{{B}_{i}}{{s}_{i}}: i \leq n \right\rangle$ is a partial run of ${\Game}^{\mathtt{Stab}}\left(\HHH\right)$ in which Player I has followed $\Pi$, $\Phi\left(\left\langle \pr{{B}_{i}}{{s}_{i}}: i \leq n \right\rangle\right) = \left\langle \pr{{p}_{i}}{{A}_{i}}: i \leq n \right\rangle$,
 \begin{align*}
  \left\langle \pr{\pr{{p}_{i}}{{A}_{i}}}{{s}_{i}}: i \leq n \right\rangle
 \end{align*}
 is a partial run of ${\Game}^{\mathtt{Cond}}\left(\HHH\right)$ in which Player I has followed $\Sigma$, and
 \begin{align*}
  \forall i < n\[ {B}_{i+1} = {A}_{i+1} \cap {H}_{{p}_{i+1}, \max({s}_{i})+1} \].
 \end{align*}
 Let $\Sigma\left( \left\langle \pr{\pr{{p}_{i}}{{A}_{i}}}{{s}_{i}}: i \leq n \right\rangle \right) = \pr{{p}_{n+1}}{{A}_{n+1}}$.
 Then ${p}_{n+1} \in \PP(\HHH)$ and ${A}_{n+1} \in \HHH$.
 Hence ${B}_{n+1} = {A}_{n+1} \cap {H}_{{p}_{n+1}, \max({s}_{n})+1} \in \HHH$.
 Note that ${B}_{n+1}$ is therefore a legitimate move for Player I in ${\Game}^{\mathtt{Stab}}\left(\HHH\right)$.
 Define $\Pi\left( \left\langle \pr{{B}_{i}}{{s}_{i}}: i \leq n \right\rangle \right) = {B}_{n+1}$.
 Note that any continuation of $\left\langle \pr{{B}_{i}}{{s}_{i}}: i \leq n \right\rangle$ to length $n+2$ in which Player I follows $\Pi$ must have the form $\left\langle \pr{{B}_{i}}{{s}_{i}}: i \leq n+1 \right\rangle$, where ${s}_{n+1} \in {B}_{n+1}$.
 Given any such $\left\langle \pr{{B}_{i}}{{s}_{i}}: i \leq n+1 \right\rangle$, define $\Phi\left( \left\langle \pr{{B}_{i}}{{s}_{i}}: i \leq n+1 \right\rangle \right) = \left\langle \pr{{p}_{i}}{{A}_{i}}: i \leq n+1 \right\rangle$.
 Note that $\left\langle \pr{\pr{{p}_{i}}{{A}_{i}}}{{s}_{i}}: i \leq n+1 \right\rangle$ is a partial run of ${\Game}^{\mathtt{Cond}}\left(\HHH\right)$ in which Player I has followed $\Sigma$ because of the definition of $\pr{{p}_{n+1}}{{A}_{n+1}}$ and ${s}_{n+1} \in {B}_{n+1} \subseteq {A}_{n+1}$.
 Further, by definition and by the induction hypothesis,
 \begin{align*}
  \forall i \leq n\[ {B}_{i+1} = {A}_{i+1} \cap {H}_{{p}_{i+1}, \max({s}_{i})+1} \].
 \end{align*}
 Thus the inductive definition satisfies (1)--(3).
 This concludes the definition of $\Pi$ and $\Phi$.
 
 Since $\Pi$ is not a winning strategy for Player I in ${\Game}^{\mathtt{Stab}}\left(\HHH\right)$, there is a play $\left\langle \pr{{B}_{i}}{{s}_{i}}: i < \omega \right\rangle$ of ${\Game}^{\mathtt{Stab}}\left(\HHH\right)$ in which Player I follows $\Pi$ and loses.
 There exists $\left\langle \pr{{p}_{i}}{{A}_{i}}: i < \omega \right\rangle$ such that for each $n \in \omega$, $\Phi\left( \left\langle \pr{{B}_{i}}{{s}_{i}}: i \leq n \right\rangle \right) = \left\langle \pr{{p}_{i}}{{A}_{i}}: i \leq n \right\rangle$.
 Therefore, $\left\langle \pr{\pr{{p}_{i}}{{A}_{i}}}{{s}_{i}}: i < \omega \right\rangle$ is a play of ${\Game}^{\mathtt{Cond}}\left(\HHH\right)$ in which Player I has followed $\Sigma$, and $\forall i < \omega\[ {B}_{i+1} = {A}_{i+1} \cap {H}_{{p}_{i+1}, \max({s}_{i})+1} \]$.
 Since Player II wins the play $\left\langle \pr{{B}_{i}}{{s}_{i}}: i < \omega \right\rangle$, $\forall i < j < \omega\[\lb{{s}_{i}}{{s}_{j}}\]$ and $\[\left\{ {s}_{i}: i < \omega \right\}\] \in \HHH$.
 Lemma \ref{lem:fusion} will be used to verify that $q = {T}_{q} = {\bigcap}_{i \in \omega}{{T}_{{p}_{i}}} \in \PP(\HHH)$.
 To this end, let $X = \{{s}_{i}: i < \omega\}$, and note that $X(i) = {s}_{i}$, for all $i \in \omega$.
 Note also that by (3) of Definition \ref{def:GcondH} and by Lemma \ref{lem:amalgam}, ${p}_{i+1} \leq {p}_{i}$, for all $i \in \omega$.
 For each $i \in \omega$, ${s}_{i+1} \in {B}_{i+1}$, and so $X(i+1) \in {H}_{{p}_{i+1}, \max\left(X(i)\right)+1}$.
 Next, fix some $i < \omega$.
 It will be proved by induction on $j$ that for each $i \leq j < \omega$ and for each $e \in \lv{{T}_{{p}_{j}}}{\max\left(X(i)\right)}$, ${\succc}_{{T}_{{p}_{i}}}(e) \subseteq {\succc}_{{T}_{{p}_{j}}}(e)$.
 This is clear when $i=j$.
 Assume this is true for some $i \leq j$.
 Fix $e \in \lv{{T}_{{p}_{j+1}}}{\max\left(X(i)\right)}$ and consider $\sigma \in {\succc}_{{T}_{{p}_{i}}}(e)$.
 Since ${T}_{{p}_{j+1}} \subseteq {T}_{{p}_{j}}$, $e \in \lv{{T}_{{p}_{j}}}{\max\left(X(i)\right)}$.
 So by the induction hypothesis, $\sigma \in {\succc}_{{T}_{{p}_{j}}}(e)$.
 Therefore ${e}^{\frown}{\langle \sigma \rangle} \in {T}_{{p}_{j}}$ and $\dom\left( {e}^{\frown}{\langle \sigma \rangle} \right) = \max\left({s}_{i}\right)+1 \leq \max\left({s}_{j}\right)+1$.
 Choose ${e}^{\ast}$ such that ${e}^{\frown}{\langle \sigma \rangle} \subseteq {e}^{\ast}$ and ${e}^{\ast} \in \lv{{T}_{{p}_{j}}}{\max\left({s}_{j}\right)+1}$.
 By (3) of Definition \ref{def:GcondH}, there exists ${p}_{j, {e}^{\ast}} \leq {T}_{{p}_{j}}\langle {e}^{\ast} \rangle$ such that ${T}_{{p}_{j, {e}^{\ast}}} \subseteq {T}_{{p}_{j+1}}$.
 Since ${e}^{\frown}{\langle \sigma \rangle} \subseteq {e}^{\ast}$, ${e}^{\frown}{\langle \sigma \rangle} \in {T}_{{p}_{j, {e}^{\ast}}}$.
 Therefore ${e}^{\frown}{\langle \sigma \rangle} \in {T}_{{p}_{j+1}}$, whence $\sigma \in {\succc}_{{T}_{{p}_{j+1}}}(e)$, as required.
 This concludes the induction.
 Thus the hypotheses of Lemma \ref{lem:fusion} are all satisfied, and so $q = {T}_{q} = {\bigcap}_{i \in \omega}{{T}_{{p}_{i}}} \in \PP(\HHH)$.
 However, this means that Player II wins the play $\left\langle \pr{\pr{{p}_{i}}{{A}_{i}}}{{s}_{i}}: i < \omega \right\rangle$ of ${\Game}^{\mathtt{Cond}}\left(\HHH\right)$ even though Player I has followed $\Sigma$ during this play, contradicting the hypothesis that $\Sigma$ is a winning strategy for Player I in ${\Game}^{\mathtt{Cond}}\left(\HHH\right)$.
\end{proof}
\begin{Lemma} \label{lem:winning1}
 Suppose $\left\langle \pr{\pr{{p}_{i}}{{A}_{i}}}{{s}_{i}}: i \in \omega \right\rangle$ is a run of ${\Game}^{\mathtt{Cond}}\left(\HHH\right)$ which is won by Player II.
 If $q = {T}_{q} = {\bigcap}_{i \in \omega}{{T}_{{p}_{i}}}$, then for each $i \in \omega$,
 \begin{align*}
  \forall f \in {T}_{{p}_{i}}\[\dom(f) \leq \max({s}_{i})+1 \implies f \in {T}_{q}\].
 \end{align*}
\end{Lemma}
\begin{proof}
 Recall that by Clause (3) of Definition \ref{def:GcondH} and by Lemma \ref{lem:amalgam}, ${p}_{i+1} \leq {p}_{i}$, for all $i \in \omega$.
 Fix $i \in \omega$ and $f \in {T}_{{p}_{i}}$ with $\dom(f) \leq \max({s}_{i})+1$.
 It will be proved by induction on $j$ that $\forall i \leq j < \omega\[f \in {T}_{{p}_{j}}\]$.
 When $j=i$, there is nothing to prove.
 Suppose the statement holds for some $i \leq j < \omega$.
 So $f \in {T}_{{p}_{j}}$ and $\dom(f) \leq \max({s}_{i}) + 1 \leq \max({s}_{j})+1$.
 Choose $e \in {T}_{{p}_{j}}$ such that $f \subseteq e$ and $\dom(e) = \max({s}_{j})+1$.
 By Clause (3) of Definition \ref{def:GcondH}, there exists ${p}_{j, e} \leq {T}_{{p}_{j}} \langle e \rangle$ so that ${T}_{{p}_{j, e}} \subseteq {T}_{{p}_{j+1}}$.
 As $f \subseteq e$, $f \in {T}_{{p}_{j, e}}$, and so $f \in {T}_{{p}_{j+1}}$.
 This concludes the induction.
 Thus, $\forall i \leq j < \omega\[f \in {T}_{{p}_{j}}\]$, whence $f \in {\bigcap}_{l \in \omega}{{T}_{{p}_{l}}} = {T}_{q}$.
\end{proof}
\begin{Lemma} \label{lem:winning2}
 Suppose $\left\langle \pr{\pr{{p}_{i}}{{A}_{i}}}{{s}_{i}}: i \in \omega \right\rangle$ is a run of ${\Game}^{\mathtt{Cond}}\left(\HHH\right)$ which is won by Player II.
 Let $i \in \omega$ and $q = {T}_{q} = {\bigcap}_{j \in \omega}{{T}_{{p}_{j}}}$.
 If ${s}_{i+1} \in {H}_{{p}_{i+1}, \max({s}_{i})+1}$, then ${s}_{i+1} \in {H}_{q, \max({s}_{i})+1}$.
\end{Lemma}
\begin{proof}
 Assume ${s}_{i+1} \in {H}_{{p}_{i+1}, \max({s}_{i})+1}$.
 Then ${s}_{i+1} \in \FIN$ and $\max({s}_{i})+1 \leq \max({s}_{i+1})$.
 Consider $e \in \lv{{T}_{q}}{\max({s}_{i+1})}$.
 Then $e \in \lv{{T}_{{p}_{i+1}}}{\max({s}_{i+1})}$ and so ${\succc}_{{T}_{{p}_{i+1}}}(e)$ is $\pr{\max({s}_{i})+1}{{s}_{i+1}}$-big.
 If $\sigma \in {\succc}_{{T}_{{p}_{i+1}}}(e)$, then ${e}^{\frown}{\langle \sigma \rangle} \in {T}_{{p}_{i+1}}$ and $\dom\left({e}^{\frown}{\langle \sigma \rangle}\right) = \max({s}_{i+1})+1$.
 Therefore, by Lemma \ref{lem:winning1}, ${e}^{\frown}{\langle \sigma \rangle} \in {T}_{q}$, whence $\sigma \in {\succc}_{{T}_{q}}(e)$.
 Thus ${\succc}_{{T}_{{p}_{i+1}}}(e) \subseteq {\succc}_{{T}_{q}}(e) \subseteq {2}^{\Pset\left( \max({s}_{i+1}) \right)}$, and so ${\succc}_{{T}_{q}}(e)$ is $\pr{\max({s}_{i})+1}{{s}_{i+1}}$-big.
 This shows ${s}_{i+1} \in {H}_{q, \max({s}_{i})+1}$.
\end{proof}
\section{Properness and reading of names} \label{sec:propernames}
This section establishes that $\PP(\HHH)$ is proper, $\BS$-bounding, and satisfies a version of the continuous reading of names property (see \cite{creaturebook}).
\begin{Lemma} \label{lem:continuous}
 Let $\theta$ be a sufficiently large regular cardinal.
 Assume $M \prec H(\theta)$ is countable and that $M$ contains all relevant parameters.
 Suppose $f: \omega \rightarrow M$ is such that $\forall n \in \omega\[f(n) \in {\V}^{\PP(\HHH)} \ \text{and} \ {\forces}_{\PP(\HHH)}\;{f(n) \in \V}\]$.
 Let $p \in \PP(\HHH) \cap M$.
 Then there exist $q, X$, and $F$ such that:
 \begin{enumerate}
  \item
  $q \leq p$, $X \in {\FIN}^{\[\omega\]}$, $F$ is a function, $\dom(F) =$
  \begin{align*}
   \left\{ \pr{i}{e}: i \in \omega \wedge e \in \lv{{T}_{q}}{\max\left(X(i)\right)+1} \right\};
  \end{align*}
  \item
  $\[X\] \in \HHH$, $\forall i \in \omega\[X(i+1) \in {H}_{q, \max\left(X(i)\right)+1}\]$;
  \item
  for each $i \leq j < \omega$ and each $e \in \lv{{T}_{q}}{\max\left(X(j)\right)+1}$,
  \begin{align*}
   {T}_{q}\langle e \rangle \; {\forces}_{\PP(\HHH)} \; {f(i) = F\left(\pr{i}{e\restrict \max\left(X(i)\right)+1}\right)};
  \end{align*}
  \item
  for any $\pr{i}{e} \in \dom(F)$, $F(\pr{i}{e}) \in M$.
 \end{enumerate}
\end{Lemma}
\begin{proof}
 Define a strategy $\Sigma$ for Player I in ${\Game}^{\mathtt{Cond}}\left(\HHH\right)$ as follows.
 ${p}_{0} = p \in \PP(\HHH)$ and ${A}_{0} = \FIN \in \HHH$.
 Note that ${p}_{0} \in M$.
 Define $\Sigma(\emptyset) = \pr{{p}_{0}}{{A}_{0}}$.
 Now suppose that $i \in \omega$ and that $\left\langle \pr{\pr{{p}_{j}}{{A}_{j}}}{{s}_{j}}: j \leq i \right\rangle$ is a partial run of ${\Game}^{\mathtt{Cond}}\left(\HHH\right)$ in which Player I has followed $\Sigma$ and that ${p}_{i} \in M$.
 Let $l = \max({s}_{i})+1$, for ease of notation.
 For any $e \in \lv{{T}_{{p}_{i}}}{l}$, ${T}_{{p}_{i}}\langle e \rangle \in \PP(\HHH)$ and by hypothesis, ${\forces}_{\PP(\HHH)}\;{f(i) \in \V}$.
 Since ${p}_{i} \in M$ and $f(i) \in M$, there exist sequences $\left\langle {p}_{i, e}: e \in \lv{{T}_{{p}_{i}}}{l} \right\rangle \in M$ and $\left\langle {x}_{i, e}: e \in \lv{{T}_{{p}_{i}}}{l} \right\rangle \in M$ such that
 \begin{align*}
  \forall e \in \lv{{T}_{{p}_{i}}}{l}\[{p}_{i, e} \leq {T}_{{p}_{i}}\langle e \rangle \ \text{and} \ {p}_{i, e} \; {\forces}_{\PP(\HHH)} \; {f(i) = {x}_{i, e}}\].
 \end{align*}
 Define ${p}_{i+1} = {T}_{{p}_{i+1}} = \bigcup{\left\{ {T}_{{p}_{i, e}}: e \in \lv{{T}_{{p}_{i}}}{l} \right\}}$.
 Note that ${p}_{i+1} \in M$ and that ${p}_{i+1} \in \PP(\HHH)$ by Lemma \ref{lem:amalgam}.
 Hence ${A}_{i+1} = {H}_{{p}_{i+1}, l} \in \HHH$.
 Define
 \begin{align*}
  \Sigma\left(\left\langle \pr{\pr{{p}_{j}}{{A}_{j}}}{{s}_{j}}: j \leq i \right\rangle\right) = \pr{{p}_{i+1}}{{A}_{i+1}}.
 \end{align*}
 The completes the definition of $\Sigma$.
 Note that by definition, if $\left\langle \pr{\pr{{p}_{i}}{{A}_{i}}}{{s}_{i}}: i < \omega \right\rangle$ is a run of ${\Game}^{\mathtt{Cond}}\left(\HHH\right)$ in which Player I has followed $\Sigma$, then ${p}_{i} \in M$, for all $i < \omega$.
 
 Since $\Sigma$ is not a winning strategy for Player I, there is a run $\left\langle \pr{\pr{{p}_{i}}{{A}_{i}}}{{s}_{i}}: i < \omega \right\rangle$ of ${\Game}^{\mathtt{Cond}}\left(\HHH\right)$ in which Player I has followed $\Sigma$ and lost.
 Then $X = \{{s}_{i}: i < \omega\} \in {\FIN}^{\[\omega\]}$ and $\[X\] \in \HHH$.
 Furthermore, $X(i) = {s}_{i}$, for all $i \in \omega$.
 Again, for ease of notation, denote $\max({s}_{i})+1$ by ${l}_{i}$.
 Let $q = {T}_{q} = {\bigcap}_{i \in \omega}{{T}_{{p}_{i}}}$.
 Then $q \in \PP(\HHH)$ and $q \leq {p}_{0} = p$.
 Moreover, for any $i \in \omega$, ${s}_{i+1} \in {A}_{i+1} = {H}_{{p}_{i+1}, {l}_{i}}$, and so by Lemma \ref{lem:winning2}, ${s}_{i+1} \in {H}_{q, {l}_{i}}$, as required for (2).
 Next, by the definition of $\Sigma$, for each $i \in \omega$, ${p}_{i} \in M$ and there exist sequences $\left\langle {p}_{i, e}: e \in \lv{{T}_{{p}_{i}}}{{l}_{i}} \right\rangle \in M$ and $\left\langle {x}_{i, e}: e \in \lv{{T}_{{p}_{i}}}{{l}_{i}} \right\rangle \in M$ as described in the previous paragraph.
 Given $i \in \omega$ and $e \in \lv{{T}_{q}}{{l}_{i}}$, then $e \in \lv{{T}_{{p}_{i}}}{{l}_{i}}$, and since $e \in M$, so ${x}_{i, e} \in M$.
 Define $F(\pr{i}{e}) = {x}_{i, e}$.
 To see that (3) is satisfied, fix $i \leq j < \omega$ and ${e}^{\ast} \in \lv{{T}_{q}}{{l}_{j}}$, and let $e = {e}^{\ast}\restrict {l}_{i} \in \lv{{T}_{q}}{{l}_{i}}$.
 Thus $e \in \lv{{T}_{{p}_{i}}}{{l}_{i}}$.
 Since ${T}_{q} \subseteq {T}_{{p}_{i+1}} = \bigcup{\left\{ {T}_{{p}_{i, e'}}: e' \in \lv{{T}_{{p}_{i}}}{{l}_{i}} \right\}}$, it follows that ${T}_{q}\langle {e}^{\ast} \rangle \leq {p}_{i, e}$.
 Therefore, ${T}_{q}\langle {e}^{\ast} \rangle \; {\forces}_{\PP(\HHH)} \; {f(i) = {x}_{i, e} = F(\pr{i}{e})}$, as required.
\end{proof}
\begin{Cor} \label{cor:proper+bounding}
 $\PP(\HHH)$ is proper and ${\omega}^{\omega}$-bounding.
\end{Cor}
\begin{proof}
 To see that $\PP(\HHH)$ is proper, let $\theta$ be a sufficiently large regular cardinal, and fix a countable $M \prec H(\theta)$ containing the relevant parameters.
 Let $\seq{\mathring{\alpha}}{i}{\in}{\omega}$ be an enumeration of all $\PP(\HHH)$ names $\mathring{\alpha} \in M$ such that ${\forces}_{\PP(\HHH)} \; {\text{``}\mathring{\alpha} \ \text{is an ordinal''}}$.
 Let $f: \omega \rightarrow M$ be defined by $f(i) = {\mathring{\alpha}}_{i}$, for all $i \in \omega$.
 Fix $p \in \PP(\HHH) \cap M$.
 Applying Lemma \ref{lem:continuous}, fix $q \leq p$, $X$, and $F$ satisfying (1)--(4) of that lemma.
 Write ${l}_{i}$ for $\max\left(X(i)\right)+1$, for all $i \in \omega$.
 It needs to be seen that $q$ is $(M, \PP(\HHH))$-generic.
 For this, it suffices to see that whenever $\mathring{\alpha} \in M$ is a $\PP(\HHH)$ name such that ${\forces}_{\PP(\HHH)}\;{\text{``}\mathring{\alpha} \ \text{is an ordinal''}}$, then $q \; {\forces}_{\PP(\HHH)} \; {\mathring{\alpha} \in M}$.
 Indeed, if such an $\mathring{\alpha}$ is given, then $\mathring{\alpha} = {\mathring{\alpha}}_{i}$, for some $i \in \omega$.
 Suppose $r \leq q$.
 Choose $e \in \lv{{T}_{r}}{{l}_{i}}$.
 Then $e \in \lv{{T}_{q}}{{l}_{i}}$, ${T}_{r}\langle e \rangle \leq r$, and ${T}_{r}\langle e \rangle \leq {T}_{q}\langle e \rangle$.
 Therefore, ${T}_{r}\langle e \rangle \; {\forces}_{\PP(\HHH)} \; {\mathring{\alpha} = {\mathring{\alpha}}_{i} = f(i) = F(\pr{i}{e}) \in M}$, as required.
 
 To see that $\PP(\HHH)$ is ${\omega}^{\omega}$-bounding, fix a $\PP(\HHH)$-name $\mathring{f}$ such that
 \begin{align*}
  {\forces}_{\PP(\HHH)}\;{\mathring{f}: \omega \rightarrow \omega}
 \end{align*}
 as well as $p \in \PP(\HHH)$.
 Let $\theta$ be sufficiently large and regular, and fix a countable $M \prec H(\theta)$ with $\mathring{f}, p \in M$ and containing all the other relevant parameters.
 Define a function $f: \omega \rightarrow M$ as follows.
 For $i \in \omega$, $f(i) \in M$ is a $\PP(\HHH)$-name such that ${\forces}_{\PP(\HHH)}\;{f(i) \in \omega}$, and ${\forces}_{\PP(\HHH)}\;{\mathring{f}(i) = f(i)}$.
 Using Lemma \ref{lem:continuous}, find $q \leq p$, $X$, and $F$ satisfying (1)--(4) of that lemma.
 Again, write ${l}_{i}$ for $\max\left(X(i)\right)+1$, for all $i \in \omega$.
 Define $g: \omega \rightarrow \omega$ as follows.
 For any $i \in \omega$, $\lv{{T}_{q}}{{l}_{i}}$ is a finite set.
 So there exists $g(i) \in \omega$ such that for any $e \in \lv{{T}_{q}}{{l}_{i}}$, if $F(\pr{i}{e}) \in \omega$, then $F(\pr{i}{e}) \leq g(i)$.
 Now, it is easy to verify that $q \; {\forces}_{\PP(\HHH)} \; {\forall i \in \omega\[\mathring{f}(i) \leq g(i)\]}$.
 Indeed if not, then there are $r \leq q$ and $i \in \omega$ so that $r \; {\forces}_{\PP(\HHH)} \; {\mathring{f}(i) > g(i)}$.
 Find $e \in \lv{{T}_{r}}{{l}_{i}}$.
 Then $e \in \lv{{T}_{q}}{{l}_{i}}$, ${T}_{r}\langle e \rangle \leq r$, and ${T}_{r}\langle e \rangle \leq {T}_{q}\langle e \rangle$.
 So ${T}_{r}\langle e \rangle \; {\forces}_{\PP(\HHH)} \; {\mathring{f}(i) = f(i) = F(\pr{i}{e}) \leq g(i)}$, a contradiction.
\end{proof}
\begin{Lemma} \label{lem:continuous1}
 Suppose $\mathring{x}$ is a $\PP(\HHH)$-name such that ${\forces}_{\PP(\HHH)}\;{\mathring{x}: \omega \rightarrow 2}$.
 Let $p \in \PP(\HHH)$.
 Then there exist $q, X$, and $\left\langle {\eta}_{i, e}: i \in \omega \wedge e \in \lv{{T}_{q}}{\max\left(X(i)\right)+1} \right\rangle$ such that:
 \begin{enumerate}
  \item
  $q \leq p$, $X \in {\FIN}^{\[\omega\]}$, and
  \begin{align*}
   \forall i \in \omega \forall e \in \lv{{T}_{q}}{\max\left(X(i)\right)+1}\[{\eta}_{i, e}: \max\left(X(i)\right)+1 \rightarrow 2\];
  \end{align*}
  \item
  $\[X\] \in \HHH$, $\forall i \in \omega\[X(i+1) \in {H}_{q, \max\left(X(i)\right)+1}\]$;
  \item
  for each $i \in \omega$ and $e \in \lv{{T}_{q}}{\max\left(X(i)\right)+1}$,
  \begin{align*}
   {T}_{q}\langle e \rangle \; {\forces}_{\PP(\HHH)} \; {\mathring{x}\restrict\left(\max\left(X(i)\right)+1\right) = {\eta}_{i, e}};
  \end{align*}
  \item
  for any $i \leq j < \omega$, $e \in \lv{{T}_{q}}{\max\left(X(i)\right)+1}$, and ${e}^{\ast} \in \lv{{T}_{q}}{\max\left(X(j)\right)+1}$, if $e \subseteq {e}^{\ast}$, then ${\eta}_{i, e} \subseteq {\eta}_{j, {e}^{\ast}}$.
 \end{enumerate}
\end{Lemma}
\begin{proof}
 This is very similar to the proof of Lemma \ref{lem:continuous}, expect that $M$ plays no role here.
 Details are provided for completeness.
 Define a strategy $\Sigma$ for Player I in ${\Game}^{\mathtt{Cond}}\left(\HHH\right)$ as follows.
 ${p}_{0} = p \in \PP(\HHH)$ and ${A}_{0} = \FIN \in \HHH$.
 Define $\Sigma(\emptyset) = \pr{{p}_{0}}{{A}_{0}}$.
 Now suppose that $i \in \omega$ and that $\left\langle \pr{\pr{{p}_{j}}{{A}_{j}}}{{s}_{j}}: j \leq i \right\rangle$ is a partial run of ${\Game}^{\mathtt{Cond}}\left(\HHH\right)$ in which Player I has followed $\Sigma$.
 Let $l = \max({s}_{i})+1$, for ease of notation.
 There exist sequences $\left\langle {p}_{i, e}: e \in \lv{{T}_{{p}_{i}}}{l} \right\rangle$ and $\left\langle {\eta}_{i, e}: e \in \lv{{T}_{{p}_{i}}}{l} \right\rangle$ such that
 \begin{align*}
  \forall e \in \lv{{T}_{{p}_{i}}}{l}\[{p}_{i, e} \leq {T}_{{p}_{i}}\langle e \rangle \ \text{and} \ {p}_{i, e} \; {\forces}_{\PP(\HHH)} \; {\mathring{x} \restrict l = {\eta}_{i, e}}\].
 \end{align*}
 Since ${\forces}_{\PP(\HHH)} \; {\mathring{x}: \omega \rightarrow 2}$, it follows that ${\eta}_{i, e}: l \rightarrow 2$.
 Define ${p}_{i+1} = {T}_{{p}_{i+1}} = \bigcup{\left\{ {T}_{{p}_{i, e}}: e \in \lv{{T}_{{p}_{i}}}{l} \right\}}$.
 Note that ${p}_{i+1} \in \PP(\HHH)$ by Lemma \ref{lem:amalgam}.
 Hence ${A}_{i+1} = {H}_{{p}_{i+1}, l} \in \HHH$.
 Define $\Sigma\left(\left\langle \pr{\pr{{p}_{j}}{{A}_{j}}}{{s}_{j}}: j \leq i \right\rangle\right) = \pr{{p}_{i+1}}{{A}_{i+1}}$.
 This completes the definition of $\Sigma$.
 
 Since $\Sigma$ is not a winning strategy for Player I, there is a run $\left\langle \pr{\pr{{p}_{i}}{{A}_{i}}}{{s}_{i}}: i < \omega \right\rangle$ of ${\Game}^{\mathtt{Cond}}\left(\HHH\right)$ in which Player I has followed $\Sigma$ and lost.
 Then $X = \{{s}_{i}: i < \omega\} \in {\FIN}^{\[\omega\]}$ and $\[X\] \in \HHH$.
 Furthermore, $X(i) = {s}_{i}$, for all $i \in \omega$.
 Again, for ease of notation, denote $\max({s}_{i})+1$ by ${l}_{i}$.
 Let $q = {T}_{q} = {\bigcap}_{i \in \omega}{{T}_{{p}_{i}}}$.
 Then $q \in \PP(\HHH)$ and $q \leq {p}_{0} = p$.
 Moreover, for any $i \in \omega$, ${s}_{i+1} \in {A}_{i+1} = {H}_{{p}_{i+1}, {l}_{i}}$, and so by Lemma \ref{lem:winning2}, ${s}_{i+1} \in {H}_{q, {l}_{i}}$, as required for (2).
 Next, by the definition of $\Sigma$, for each $i \in \omega$, there exist sequences $\left\langle {p}_{i, e}: e \in \lv{{T}_{{p}_{i}}}{{l}_{i}} \right\rangle$ and $\left\langle {\eta}_{i, e}: e \in \lv{{T}_{{p}_{i}}}{{l}_{i}} \right\rangle$ as described in the previous paragraph.
 To see that (3) is satisfied, fix $i \in \omega$ and $e \in \lv{{T}_{q}}{{l}_{i}}$.
 Then $e \in \lv{{T}_{{p}_{i}}}{{l}_{i}}$.
 Since ${T}_{q} \subseteq {T}_{{p}_{i+1}} = \bigcup{\left\{ {T}_{{p}_{i, e'}}: e' \in \lv{{T}_{{p}_{i}}}{{l}_{i}} \right\}}$, it follows that ${T}_{q}\langle e \rangle \leq {p}_{i, e}$.
 Therefore, ${T}_{q}\langle e \rangle \; {\forces}_{\PP(\HHH)} \; {\mathring{x} \restrict {l}_{i} = {\eta}_{i, e}}$, as required.
 Finally for (4), if $i\leq j < \omega$, $e \in \lv{{T}_{q}}{{l}_{i}}$, ${e}^{\ast} \in \lv{{T}_{q}}{{l}_{j}}$, and $e \subseteq {e}^{\ast}$, then ${T}_{q}\langle {e}^{\ast} \rangle \leq {T}_{q}\langle e \rangle$, whence ${T}_{q}\langle {e}^{\ast} \rangle \; {\forces}_{\PP(\HHH)} \; {\mathring{x} \restrict {l}_{i} = {\eta}_{i, e}}$, whence ${\eta}_{i, e} = {\eta}_{j, {e}^{\ast}} \restrict {l}_{i}$.
\end{proof}
\section{Destroying witnesses to Hindman's theorem} \label{sec:nohindman}
This section continues the analysis of $\PP(\HHH)$.
It will be shown that there is no stable ordered-union ultrafilter extending $\HHH$ after forcing with $\PP(\HHH)$.
Furthermore, this is true after forcing with any $\BS$-bounding partial order which contains $\PP(\HHH)$ as a complete suborder.
\begin{Lemma} \label{lem:noextension}
 Suppose $\GGG \subseteq \Pset(\FIN)$ and that $\PP$ is a forcing notion.
 Suppose $\mathring{c}$ is a $\PP$-name such that ${\forces}_{\PP} \; {\mathring{c}: \FIN \rightarrow 2}$.
 Assume that for each $t \in 2$ and each $\PP$-name $\mathring{A}$, if ${\forces}_{\PP} \; {\mathring{A} \subseteq \FIN}$, then for every $p \in \PP$, there are $q \leq p$ and $B \in \GGG$ such that either:
 \begin{enumerate}
  \item
  there exists $k \in \omega$ such that $q \; {\forces}_{\PP} \; {\forall s \in \mathring{A} \cap B\[\min(s) < k\]}$, or
  \item
  $\forall s \in B \forall r \leq q \exists u \exists r' \leq r\[r' \; {\forces}_{\PP} \; {u \in \mathring{A}} \ \text{and} \ r' \; {\forces}_{\PP} \; {\mathring{c}(s \cup u) \neq t} \]$.
 \end{enumerate}
 Then ${\forces}_{\PP} \; {\text{``there is no stable ordered-union ultrafilter} \ \HHH \ \text{on} \ \FIN \ \text{with} \ \GGG \subseteq \HHH\text{''.}}$
\end{Lemma}
\begin{proof}
 For each $t \in 2$ and each $\PP$-name $\mathring{A}$ such that ${\forces}_{\PP} \; {\mathring{A} \subseteq \FIN}$, let
 \begin{align*}
  {\DDD}_{t, \mathring{A}} = \left\{q \in \PP: \text{there is} \ B \in \GGG \ \text{so that either (1) or (2) holds with} \ t, \mathring{A}, q, B \right\}.
 \end{align*}
 The hypothesis of the lemma is that ${\DDD}_{t, \mathring{A}}$ is dense.
 Let $G$ be $(\V, \PP)$-generic.
 Assume for a contradiction that in $\VG$, there exists a stable ordered-union ultrafilter $\HHH$ on $\FIN$ such that $\GGG \subseteq \HHH$.
 As $\mathring{c}\[G\]: \FIN \rightarrow 2$, there exists $H \in \HHH$ such that $\mathring{c}\[G\]$ is constantly $t$ on $H$, for some $t \in 2$.
 Fix $Y \in {\FIN}^{\[\omega\]}$ with $\[Y\] \in \HHH$ and $\[Y\] \subseteq H$.
 Let $\mathring{A}$ be a $\PP$-name in $\V$ such that ${\forces}_{\PP} \; {\mathring{A} \subseteq \FIN}$ and $\mathring{A}\[G\] = \[Y\]$.
 Fix $q \in {\DDD}_{t, \mathring{A}} \cap G$ and $B \in \GGG$ so that either (1) or (2) holds for $t, \mathring{A}$, $q$, and $B$.
 Suppose first that (1) holds.
 Then since $q \in G$, there exists $k \in \omega$ such that $\forall s \in \mathring{A}\[G\] \cap B\[\min(s) < k\]$.
 However, $\mathring{A}\[G\] \cap B \in \HHH$, and so $\{s \in \mathring{A}\[G\] \cap B: \min(s) > k\} \in \HHH$.
 In particular, there is some $s \in \mathring{A}\[G\] \cap B$ with $\min(s) > k$, a contradiction.
 
 Next suppose that (2) holds.
 Fix $s \in B \cap \[Y\]$.
 In $\V$, define
 \begin{align*}
  {\EEE}_{q, s, \mathring{A}, \mathring{c}, t} = \left\{ r' \leq q: \exists u\[r' \; {\forces}_{\PP} \; {u \in \mathring{A}} \ \text{and} \ r' \; {\forces}_{\PP} \; {\mathring{c}(s \cup u) \neq t} \] \right\}.
 \end{align*}
 By (2), ${\EEE}_{q, s, \mathring{A}, \mathring{c}, t}$ is dense below $q$.
 Since $q \in G$, there is some $r' \in G \cap {\EEE}_{q, s, \mathring{A}, \mathring{c}, t}$.
 Thus in $\VG$, there exists $u \in \[Y\]$ such that $\mathring{c}\[G\](s \cup u) \neq t$.
 However, since $s \in \[Y\]$ and $u \in \[Y\]$, so $s \cup u \in \[Y\] \subseteq H$, contradicting the fact that $\mathring{c}\[G\]$ is constantly $t$ on $H$.
 This contradiction concludes the proof.
\end{proof}
\begin{Lemma} \label{lem:bigrestriction}
 Suppose $k \leq k' \leq l < \omega$ and $s \in \FIN$ with $l = \max(s)$.
 Suppose $A \subseteq {2}^{\Pset(l)}$ is $\pr{k'}{s}$-big.
 Suppose $t \in 2$.
 Then $B = \{\tau \in A: \forall u \in \Pset(k) \forall v \in \FIN\[k \leq \min(v) \leq \max(v) < k' \implies \tau(u \cup v \cup {s}^{-}) = 1-t\]\} \subseteq A \subseteq {2}^{\Pset(l)}$ is $\pr{k}{s}$-big.
\end{Lemma}
\begin{proof}
 Let $\sigma: \Pset(k) \rightarrow 2$ be given.
 Define $\sigma': \Pset(k') \rightarrow 2$ by stipulating that for any $w \in \Pset(k')$,
 \begin{align*}
  \sigma'(w) = \begin{cases}
   \sigma(w) &\ \text{if} \ w \subseteq k\\
   1-t       &\ \text{if} \ w \not\subseteq k.
  \end{cases}
 \end{align*}
 As $A$ is $\pr{k'}{s}$-big, there exists $\tau \in A$ such that $\forall w \in \Pset(k')\[\sigma'(w) = \tau(w \cup {s}^{-})\]$.
 To see $\tau \in B$, suppose $u \in \Pset(k)$ and $v \in \FIN$ with $k \leq \min(v) \leq \max(v) < k'$.
 Note that $v, u \subseteq k'$, whence $w = u \cup v \in \Pset(k')$.
 Further, $\min(v) \in w \setminus k$, showing that $w \not\subseteq k$.
 Therefore, $1-t=\sigma'(w) = \tau(u \cup v \cup {s}^{-})$, showing that $\tau \in B$.
 Finally, if $u \in \Pset(k)$, then $u \in \Pset(k')$ and $\sigma(u)=\sigma'(u)=\tau(u \cup {s}^{-})$.
 As $\sigma$ was arbitrary, $B$ is $\pr{k}{s}$-big.
\end{proof}
\begin{Lemma} \label{lem:genericreal}
 Suppose $G$ is $(\V, \PP(\HHH))$-generic.
 In $\VG$, there exists a function $F \in {\prod}_{k \in \omega}{{2}^{\Pset(k)}}$ such that $\{F\} = \bigcap{\left\{ \[{T}_{p}\]: p \in G \right\}}$.
\end{Lemma}
\begin{proof}
 This is a standard density argument.
\end{proof}
\begin{Def} \label{def:CG}
 Suppose $G$ is $(\V, \PP(\HHH))$-generic.
 In $\VG$, let ${F}_{G} \in {\prod}_{k \in \omega}{{2}^{\Pset(k)}}$ denote the unique function such that $\{{F}_{G}\} = \bigcap\left\{\[{T}_{p}\]: p \in G \right\}$.
 Define ${c}_{G}: \FIN \rightarrow 2$ as follows.
 For $s \in \FIN$, ${F}_{G}(\max(s)): \Pset(\max(s)) \rightarrow 2$.
 Recalling that ${s}^{-} \in \Pset(\max(s))$, define ${c}_{G}(s) = {F}_{G}(\max(s))({s}^{-}) \in 2$.
 In $\V$, let ${\mathring{F}}_{G}$ and ${\mathring{c}}_{G}$ be $\PP(\HHH)$-names that are forced by every condition to denote ${F}_{G}$ and ${c}_{G}$ respectively.
\end{Def}
\begin{Lemma} \label{lem:extend1-t}
 Suppose $k \leq k' \leq l < \omega$ and $s \in \FIN$ with $l = \max(s)$.
 Suppose $p \in \PP(\HHH)$ and $t \in 2$.
 Assume that $s \in {H}_{p, k'}$.
 Then there exists $q \leq p$ such that:
 \begin{enumerate}
  \item
  $\forall i < l \forall e \in \lv{{T}_{q}}{i}\[{\succc}_{{T}_{q}}(e) = {\succc}_{{T}_{p}}(e)\]$;
  \item
  $\forall i > l \forall e \in \lv{{T}_{q}}{i}\[{\succc}_{{T}_{q}}(e) = {\succc}_{{T}_{p}}(e)\]$;
  \item
  $s \in {H}_{q, k}$;
  \item
  for each $u \in \Pset(k)$ and each $v \in \FIN$, if $k \leq \min(v) \leq \max(v) < k'$, then $q \; {\forces}_{\PP(\HHH)} \; {{\mathring{c}}_{G}(u \cup v \cup s) = 1-t}$.
 \end{enumerate}
\end{Lemma}
\begin{proof}
 Since $s \in {H}_{p, k'}$, for each $f \in \lv{{T}_{p}}{l}$, ${A}_{f} = {\succc}_{{T}_{p}}(f) \subseteq {2}^{\Pset(l)}$ is $\pr{k'}{s}$-big.
 Thus, by applying Lemma \ref{lem:bigrestriction}, it is seen that for each $f \in \lv{{T}_{p}}{l}$, ${B}_{f} = \{\tau \in {A}_{f}: \forall u \in \Pset(k) \forall v \in \FIN\[k \leq \min(v) \leq \max(v) < k' \implies \tau(u \cup v \cup {s}^{-}) = 1-t\]\}$ is $\pr{k}{s}$-big.
 In particular, each ${B}_{f}$ is non-empty.
 Since $\lv{{T}_{p}}{l}$ is non-empty,
 \begin{align*}
  B = \left\{{f}^{\frown}{\langle \tau \rangle}: f \in \lv{{T}_{p}}{l} \wedge \tau \in {B}_{f} \right\}
 \end{align*}
 is a non-empty subset of $\lv{{T}_{p}}{l+1}$.
 Therefore by defining
 \begin{align*}
  q = \bigcup\left\{ {T}_{p}\langle {f}^{\frown}{\langle \tau \rangle} \rangle: f \in \lv{{T}_{p}}{l} \wedge \tau \in {B}_{f} \right\},
 \end{align*}
 Lemma \ref{lem:amalgam} insures that $q \in \PP(\HHH)$ and $q \leq p$.
 
 To see that (1) holds, suppose $i < l$ and that $e \in \lv{{T}_{q}}{i}$.
 Then $e \in \lv{{T}_{p}}{i}$, and if $\sigma \in {\succc}_{{T}_{p}}(e)$, then ${e}^{\frown}{\langle \sigma \rangle} \in {T}_{p}$ and $\dom\left({e}^{\frown}{\langle \sigma \rangle}\right) \leq l$.
 Choose $f \in \lv{{T}_{p}}{l}$ so that ${e}^{\frown}{\langle \sigma \rangle} \subseteq f$ and choose $\tau \in {B}_{f}$.
 Thus ${e}^{\frown}{\langle \sigma \rangle} \in {T}_{p}\langle {f}^{\frown}{\langle \tau \rangle} \rangle \subseteq {T}_{q}$, whence $\sigma \in {\succc}_{{T}_{q}}(e)$.
 So ${\succc}_{{T}_{p}}(e) \subseteq {\succc}_{{T}_{q}}(e) \subseteq {\succc}_{{T}_{p}}(e)$, as needed.
 
 Similarly for (2), suppose $i > l$ and that $e \in \lv{{T}_{q}}{i}$.
 Then $e \in {T}_{p}$ and ${f}^{\frown}{\langle \tau \rangle} \subseteq e$, for some $f \in \lv{{T}_{p}}{l}$ and $\tau \in {B}_{f}$.
 If $\sigma \in {\succc}_{{T}_{p}}(e)$, then ${e}^{\frown}{\langle \sigma \rangle} \in {T}_{p}\langle {f}^{\frown}{\langle \tau \rangle} \rangle \subseteq {T}_{q}$, whence $\sigma \in {\succc}_{{T}_{q}}(e)$.
 So ${\succc}_{{T}_{p}}(e) \subseteq {\succc}_{{T}_{q}}(e) \subseteq {\succc}_{{T}_{p}}(e)$, as needed.
 
 For (3), consider $f \in \lv{{T}_{q}}{l}$.
 It needs to be seen that ${\succc}_{{T}_{q}}(f) \subseteq {2}^{\Pset(l)}$ is $\pr{k}{s}$-big.
 Indeed $f \in \lv{{T}_{p}}{l}$ and for any $\tau \in {B}_{f}$, ${f}^{\frown}{\langle \tau \rangle} \in {T}_{q}$, whence $\tau \in {\succc}_{{T}_{q}}(f)$.
 So ${B}_{f} \subseteq {\succc}_{{T}_{q}}(f)$, and so ${\succc}_{{T}_{q}}(f)$ is $\pr{k}{s}$-big.
 
 Finally for (4), fix $u \in \Pset(k)$ and $v \in \FIN$ such that $k \leq \min(v) \leq \max(v) < k'$.
 Let $w=u\cup v \cup s$.
 Note that $w \in \FIN$, $l = \max(w)$, and ${w}^{-} = u \cup v \cup {s}^{-}$.
 Now let $G$ be $(\V, \PP(\HHH))$-generic with $q \in G$.
 Since ${F}_{G} \in \[{T}_{q}\]$, ${F}_{G} \restrict l+1 \in {T}_{q}$, and so ${F}_{G} \restrict l+1 = {f}^{\frown}{\langle \tau \rangle}$, for some $f \in \lv{{T}_{p}}{l}$ and $\tau \in {B}_{f}$.
 Therefore, ${c}_{G}(u \cup v \cup s) = {c}_{G}(w) = {F}_{G}(\max(w))({w}^{-}) = {F}_{G}(l)({w}^{-}) = \tau({w}^{-}) = \tau(u \cup v \cup {s}^{-}) = 1-t$, by the definition of ${B}_{f}$.
\end{proof}
\begin{Lemma} \label{lem:1-tfusion}
 Suppose $p \in \PP(\HHH)$ and $\psi \in \BS$ is such that for all $k \in \omega$, $k \leq \psi(k)$.
 Let $t \in 2$.
 Then there exist $q \leq p$ and $X \in {\FIN}^{\[\omega\]}$ such that $\[X\] \in \HHH$ and for each $i \in \omega$, for each $u \in \Pset\left(\max\left(X(i)\right)+1\right)$ and each $v \in \FIN$, if $\max\left(X(i)\right)+1 \leq \min(v) \leq \max(v) < \psi\left(\max\left(X(i)\right)+1\right)$, then $q \; {\forces}_{\PP(\HHH)} \; {{\mathring{c}}_{G}\left(u \cup v \cup X(i+1)\right)=1-t}$.
\end{Lemma}
\begin{proof}
 For each $n \in \omega$, let ${A}_{n} = {H}_{p, \psi(n+1)} \in \HHH$.
 By Lemma \ref{lem:selectivediag}, there exists $X \in {\FIN}^{\[\omega\]}$ such that $\[X\] \in \HHH$ and for each $i \in \omega$, $X(i+1) \in {A}_{\max\left(X(i)\right)}$.
 For ease of notation, let ${s}_{i} = X(i)$, ${l}_{i} = \max\left(X(i)\right)$, ${k}_{i} = \max\left(X(i)\right)+1$, and ${k}_{i}' = \psi\left(\max\left(X(i)\right)+1\right)$.
 Thus ${s}_{i+1} \in {H}_{p, {k}_{i}'}$, and ${k}_{i} \leq {k}_{i}' \leq {l}_{i+1} < \omega$.
 
 Lemma \ref{lem:fusion} will be used to obtain $q$.
 To this end, construct a sequence $\seq{p}{i}{\in}{\omega}$ satisfying the following:
 \begin{enumerate}
  \item
  $\forall i \in \omega\[{p}_{i} \in \PP(\HHH)\]$, $\forall i \in \omega\[{p}_{i+1} \leq {p}_{i}\]$, ${p}_{0} = p$;
  \item
  for each $i \in \omega$, ${s}_{i+1} \in {H}_{{p}_{i+1}, {k}_{i}}$;
  \item
  for each $i \leq j < \omega$ and for each $e \in \lv{{T}_{{p}_{j}}}{{l}_{i}}$, ${\succc}_{{T}_{{p}_{i}}}(e) \subseteq {\succc}_{{T}_{{p}_{j}}}(e)$;
  \item
  for each $j \leq {j}^{\ast} < \omega$, ${s}_{{j}^{\ast}+1} \in {H}_{{p}_{j}, {k}_{{j}^{\ast}}'}$;
  \item
  for each $i \in \omega$, for each $u \in \Pset\left({k}_{i}\right)$ and each $v \in \FIN$, if ${k}_{i} \leq \min(v) \leq \max(v) < {k}_{i}'$, then ${p}_{i+1} \; {\forces}_{\PP(\HHH)} \; {{\mathring{c}}_{G}\left(u \cup v \cup {s}_{i+1}\right)=1-t}$.
 \end{enumerate}
 Suppose for a moment that such a sequence has been constructed.
 Then by Lemma \ref{lem:fusion}, $q = {T}_{q} = {\bigcap}_{i \in \omega}{{T}_{{p}_{i}}} \in \PP(\HHH)$, $q \leq {p}_{0} = p$, and for each $i \in \omega$, since $q \leq {p}_{i+1}$, $q$ is as desired because of (5).
 
 The sequence $\seq{p}{i}{\in}{\omega}$ is constructed by induction.
 Here are some details.
 Define ${p}_{0} = p$ and notice that (4) is satisfied because for each ${j}^{\ast} < \omega$, ${s}_{{j}^{\ast}+1} \in {H}_{p, {k}_{{j}^{\ast}}'}$.
 Fix $j \in \omega$ and suppose that $\seq{p}{i}{\leq}{j}$ satisfying (1)--(5) is given.
 Applying (4) with $j = {j}^{\ast}$ yields ${s}_{j+1} \in {H}_{{p}_{j}, {k}_{j}'}$.
 Hence by Lemma \ref{lem:extend1-t}, there exists ${p}_{j+1} \leq {p}_{j}$ satisfying (1)--(4) of Lemma \ref{lem:extend1-t}.
 It is clear that (1), (2), and (5) are satisfied.
 Now for each $i \leq j$ and for each $e \in \lv{{T}_{{p}_{j+1}}}{{l}_{i}}$, $e \in \lv{{T}_{{p}_{j}}}{{l}_{i}}$, and since ${l}_{i} < {l}_{j+1}$, so ${\succc}_{{T}_{{p}_{i}}}(e) \subseteq {\succc}_{{T}_{{p}_{j}}}(e) = {\succc}_{{T}_{{p}_{j+1}}}(e)$ by the induction hypothesis and by (1) of Lemma \ref{lem:extend1-t}.
 This verifies (3).
 Next, suppose $j+1 \leq {j}^{\ast} < \omega$.
 It needs to be seen that ${s}_{{j}^{\ast}+1} \in {H}_{{p}_{j+1}, {k}_{{j}^{\ast}}'}$.
 By the induction hypothesis, ${s}_{{j}^{\ast}+1} \in {H}_{{p}_{j}, {k}_{{j}^{\ast}}'}$.
 Suppose that $f \in \lv{{T}_{{p}_{j+1}}}{{l}_{{j}^{\ast}+1}}$.
 It needs to be seen that ${\succc}_{{T}_{{p}_{j+1}}}(f)$ is $\pr{{k}_{{j}^{\ast}}'}{{s}_{{j}^{\ast}+1}}$-big.
 As $f \in \lv{{T}_{{p}_{j}}}{{l}_{{j}^{\ast}+1}}$, it is known that ${\succc}_{{T}_{{p}_{j}}}(f)$ is $\pr{{k}_{{j}^{\ast}}'}{{s}_{{j}^{\ast}+1}}$-big.
 By (2) of Lemma \ref{lem:extend1-t}, since ${l}_{{j}^{\ast}+1} > {l}_{j+1}$, ${\succc}_{{T}_{{p}_{j+1}}}(f) = {\succc}_{{T}_{{p}_{j}}}(f)$.
 Hence ${\succc}_{{T}_{{p}_{j+1}}}(f)$ is $\pr{{k}_{{j}^{\ast}}'}{{s}_{{j}^{\ast}+1}}$-big, as required for (4).
 This concludes the induction and the proof.
\end{proof}
\begin{Theorem} \label{thm:noresseruction}
 Suppose $\QQ$ is an $\BS$-bounding forcing.
 If $\PP(\HHH)$ completely embeds into $\QQ$, then
 \begin{align*}
  {\forces}_{\QQ}\;{``\text{there is no stable ordered-union ultrafilter on} \ \FIN \ \text{extending} \ \HHH''}.
 \end{align*}
\end{Theorem}
\begin{proof}
 Let $\pi: \PP(\HHH) \rightarrow \QQ$ be a complete embedding, and let ${\pi}^{\ast}$ denote the associated map from $\PP(\HHH)$-names to $\QQ$-names.
 Let $\mathring{c}$ denote ${\pi}^{\ast}\left({\mathring{c}}_{G}\right)$.
 Then ${\forces}_{\QQ}\;{\mathring{c}: \FIN \rightarrow 2}$.
 Lemma \ref{lem:noextension} will be used to obtain the desired conclusion.
 To this end, suppose $t \in 2$ and that $\mathring{A}$ is a $\QQ$-name such that ${\forces}_{\QQ}\;{\mathring{A} \subseteq \FIN}$.
 Let ${p}^{\ast} \in \QQ$ be given.
 Assume that there are no ${q}^{\ast} \leq {p}^{\ast}$ and $B \in \HHH$ such that (1) of Lemma \ref{lem:noextension} holds.
 
 Let $G$ be any $(\V, \QQ)$-generic filter with ${p}^{\ast} \in G$.
 Then $\mathring{A}\[G\] \subseteq \FIN$.
 In $\VG$, it must be the case that for every $k \in \omega$, there exists $s \in \mathring{A}\[G\]$ with $k \leq \min(s)$, for otherwise there would be a ${q}^{\ast} \in G$ such that ${q}^{\ast} \leq {p}^{\ast}$ and ${q}^{\ast}$ witnesses (1) of Lemma \ref{lem:noextension} with $B = \FIN \in \HHH$.
 Therefore there is a function $\varphi: \omega \rightarrow \omega$ such that for every $k \in \omega$, there exists $v \in \mathring{A}\[G\]$ with $k \leq \min(v) \leq \max(v) < \varphi(k)$.
 As this holds for every $(\V, \QQ)$-generic $G$ with ${p}^{\ast} \in G$, there is a $\QQ$-name $\mathring{\varphi}$ in $\V$ such that ${\forces}_{\QQ} \; {\mathring{\varphi}: \omega \rightarrow \omega}$ and ${p}^{\ast} \; {\forces}_{\QQ} \; {\forall k \in \omega \exists v \in \mathring{A}\[k \leq \min(v) \leq \max(v) < \mathring{\varphi}(k)\]}$.
 
 Since $\QQ$ is $\BS$-bounding, there exist ${p}^{\ast}_{1} \leq {p}^{\ast}$ and $\psi: \omega \rightarrow \omega$ in $\V$ such that ${p}^{\ast}_{1} \; {\forces}_{\QQ} \; {\forall k \in \omega\[\mathring{\varphi}(k) < \psi(k)\]}$.
 Let $p \in \PP(\HHH)$ be a reduction of ${p}^{\ast}_{1}$ with respect to the complete embedding $\pi$.
 Applying Lemma \ref{lem:1-tfusion} in $\V$, find $q \leq p$ and $X \in {\FIN}^{\[\omega\]}$ satisfying the conclusions of Lemma \ref{lem:1-tfusion}.
 Let $B = \left\{ s \in \[X\]: \min(s) > \max\left(X(0)\right) \right\} \in \HHH$.
 By the choice of $p$, $\pi(q)$ is compatible with ${p}^{\ast}_{1}$ in $\QQ$.
 Choose any ${q}^{\ast} \leq \pi(q), {p}^{\ast}_{1}$.
 To see that ${q}^{\ast} \leq {p}^{\ast}$ and $B \in \HHH$ satisfy (2) of Lemma \ref{lem:noextension}, fix some $s \in B$ and ${r}^{\ast} \leq {q}^{\ast}$.
 Then $\max(s) = \max\left(X(i+1)\right)$, for some $i \in \omega$.
 Let $u = s \cap \left(\max\left(X(i)\right)+1\right) \in \Pset\left(\max\left(X(i)\right)+1\right)$ and note $s = u \cup X(i+1)$.
 Since ${r}^{\ast} \leq {p}^{\ast}_{1} \leq {p}^{\ast}$, there exist ${r}^{\ast}_{1} \leq {r}^{\ast}$ and $v$ such that $\max\left(X(i)\right)+1 \leq \min(v)$, ${r}^{\ast}_{1} \; {\forces}_{\QQ} \; {v \in \mathring{A} \subseteq \FIN}$, and ${r}^{\ast}_{1} \; {\forces}_{\QQ} \; {\max(v) < \mathring{\varphi}\left(\max\left(X(i)\right)+1\right) < \psi\left(\max\left(X(i)\right)+1\right)}$.
 By the choice of $q$, $q \; {\forces}_{\PP(\HHH)} \; {{\mathring{c}}_{G}(s \cup v) = {\mathring{c}}_{G}(u \cup v \cup X(i+1)) = 1-t}$, and so
 \begin{align*}
  \pi(q) \; {\forces}_{\QQ} \; {\mathring{c}(s \cup v) = {\pi}^{\ast}\left({\mathring{c}}_{G}\right)(s \cup v) = 1-t}.
 \end{align*}
 Therefore, ${r}^{\ast}_{1} \; {\forces}_{\QQ} \; {\mathring{c}(s \cup v) = 1-t}$, as needed.
\end{proof}
\section{Preservation of all selective ultrafilters} \label{sec:preservesel}
It will be proved that $\PP(\HHH)$ preserves all selective ultrafilters from the ground model.
This is arguably the most intricate section of the paper.
Given a selective ultrafilter $\UUU$, the proof breaks down into three cases depending on whether $\UUU \; {\equiv}_{RK} \; {\HHH}_{\max}$, or $\UUU \; {\equiv}_{RK} \; \VVV$, for some $\VVV \notin {\C}_{0}(\HHH)$ and $\VVV \; {\not\equiv}_{RK} \; {\HHH}_{\max}$, or $\UUU \in {\C}_{1}(\HHH)$.
A bit of thought shows (and it will be shown) that these cases are exhaustive.
\begin{Def} \label{def:preservesU}
 Suppose $\UUU$ is an ultrafilter on $\omega$ and that $\PP$ is a forcing notion.
 $\PP$ is said to \emph{preserve $\UUU$} if ${\forces}_{\PP} \; {``\{A \subseteq \omega: \exists B \in \UUU\[B \subseteq A\]\} \ \text{is an ultrafilter on} \ \omega.''}$
 Unraveling the definitions, $\PP$ preserves $\UUU$ if and only if for every $p \in \PP$ and every $\PP$-name $\mathring{A}$ such that ${\forces}_{\PP} \; {\mathring{A} \subseteq \omega}$, there exist $q \leq p$ and $B \in \UUU$ such that $q \; {\forces}_{\PP} \; {B \subseteq \mathring{A}}$ or $q \; {\forces}_{\PP} \; {B \subseteq \omega \setminus \mathring{A}}$.
\end{Def}
\begin{Lemma} \label{lem:hmax1}
 Suppose $\mathring{x}$ is a $\PP(\HHH)$-name such that ${\forces}_{\PP(\HHH)}\;{\mathring{x}: \omega \rightarrow 2}$.
 For any $p \in \PP(\HHH)$, there exists $q \leq p$ such that there are $t \in 2$, $A \in \HHH$, $M \in {\HHH}_{\max}$, $\seq{A}{l}{\in}{M}$, and $\left\langle \left\langle {e}_{k}, {S}^{0}_{k}, {S}^{1}_{k} \right\rangle: k \in {f}_{\max}\[A\] \right\rangle$ such that:
 \begin{enumerate}
  \item
  ${e}_{k} \in \lv{{T}_{q}}{k}$, ${S}^{0}_{k}, {S}^{1}_{k} \subseteq {\succc}_{{T}_{q}}({e}_{k})$;
  \item
  for any $l \in M$, ${A}_{l} \in \HHH$, ${A}_{l} \subseteq A$, and for any $s \in {A}_{l}$ with $\max(s) = k$, ${S}^{t}_{k}$ is $\pr{l}{s}$-big;
  \item
  for any $l \in M$, for any $k \in {f}_{\max}\[{A}_{l}\]$, for any $\sigma \in {S}^{t}_{k}$,
  \begin{align*}
   {T}_{q}\langle {{e}_{k}}^{\frown}{\langle \sigma \rangle} \rangle \; {\forces}_{\PP(\HHH)} \; {\mathring{x}(k) = t}.
  \end{align*}
 \end{enumerate}
\end{Lemma}
\begin{proof}
 Apply Lemma \ref{lem:continuous1} to find $q$, $X$, and $\left\langle {\eta}_{i, e}: i \in \omega \wedge e \in \lv{{T}_{q}}{\max\left(X(i)\right)+1} \right\rangle$ satisfying (1)--(4) of that lemma.
 For any $k \in {f}_{\max}\[\[X\]\]$, there exists a unique ${j}_{k} \in \omega$ with $k = \max\left(X({j}_{k})\right)$.
 Choose ${e}_{k} \in \lv{{T}_{q}}{k}$.
 For $t \in 2$, let ${S}^{t}_{k} = \left\{ \sigma \in {\succc}_{{T}_{q}}({e}_{k}): {\eta}_{{j}_{k}, {{e}_{k}}^{\frown}{\langle \sigma \rangle}}(k) = t \right\}$.
 Observe ${S}^{0}_{k} \cup {S}^{1}_{k} = {\succc}_{{T}_{q}}({e}_{k})$.
 
 Fix $l \in \omega$.
 Define ${c}_{l}: \[X\] \rightarrow 3$ as follows.
 Given $s \in \[X\]$, let $k = \max(s) \in {f}_{\max}\[\[X\]\]$.
 If ${S}^{0}_{k}$ is $\pr{l}{s}$-big, then set ${c}_{l}(s) = 0$.
 If ${S}^{0}_{k}$ is not $\pr{l}{s}$-big, but ${S}^{1}_{k}$ is $\pr{l}{s}$-big, then set ${c}_{l}(s) = 1$.
 If neither ${S}^{0}_{k}$ nor ${S}^{1}_{k}$ is $\pr{l}{s}$-big, then set ${c}_{l}(s)=2$.
 Find ${A}_{l} \in \HHH$ such that ${A}_{l} \subseteq \[X\]$ and ${c}_{l}$ is constant on ${A}_{l}$.
 \begin{Claim} \label{claim:hmax1-1}
  ${c}_{l}$ is not constantly $2$ on ${A}_{l}$.
 \end{Claim}
 \begin{proof}
  Suppose for a contradiction that it is constantly $2$.
  Applying Lemma \ref{lem:stepranscondition} to $q$, fix $H \in \HHH$ satisfying the conclusion of that lemma.
  Fix ${k}_{0} \in \omega$ such that $l < {k}_{0}$ and for all $k \in {f}_{\max}\[H\]$ with $k \geq {k}_{0}$, (${\ast}_{l, k}$) of Lemma \ref{lem:stepranscondition} holds.
  Write ${l}^{\ast} = {\left({2}^{{2}^{l}}\right)}^{2}$. Choose $Y \in {\FIN}^{\[{l}^{\ast}+2\]}$ such that ${k}_{0} \leq \min\left(Y(0)\right)$ and $\[Y\] \subseteq {A}_{l} \cap H$.
  Let $k = \max\left(Y({l}^{\ast}+1)\right) \in {f}_{\max}\[{A}_{l}\] \cap {f}_{\max}\[H\] \subseteq {f}_{\max}\[\[X\]\]$.
  Note $l < {k}_{0} \leq k$.
  For each $i < {l}^{\ast}+1$, define ${s}_{i} = Y(i) \cup Y({l}^{\ast}+1) \in {A}_{l} \cap H$.
  Observe that $l < {k}_{0} \leq \min({s}_{i})$ and $\max({s}_{i}) = k$, and so ${s}^{-}_{i} \in H\[l, k\]$.
  Since ${c}_{l}({s}_{i}) = 2$, there exists $\pr{{\sigma}_{i, 0}}{{\sigma}_{i, 1}} \in {2}^{\Pset(l)} \times {2}^{\Pset(l)}$ such that $\forall t \in 2 \forall \tau \in {S}^{t}_{k}\exists u \in \Pset(l)\[ {\sigma}_{i, t}(u) \neq \tau(u \cup {s}^{-}_{i}) \]$.
  There must exist $i < i' < {l}^{\ast}+1$ with $\pr{{\sigma}_{i, 0}}{{\sigma}_{i, 1}} = \pr{{\sigma}_{i', 0}}{{\sigma}_{i', 1}} = \pr{{\sigma}_{0}}{{\sigma}_{1}}$.
  Let $g: H\[l, k\] \rightarrow {2}^{\Pset(l)}$ be any function so that $g({s}^{-}_{i}) = {\sigma}_{0}$ and $g({s}^{-}_{i'}) = {\sigma}_{1}$.
  Applying (${\ast}_{l, k}$) to ${e}_{k} \in \lv{{T}_{q}}{k}$ and $g$, find $\tau \in {\succc}_{{T}_{q}}({e}_{k})$ so that $\forall u \in \Pset(l)\[\tau(u \cup {s}^{-}_{i}) = g({s}^{-}_{i})(u) = {\sigma}_{0}(u) \wedge \tau(u \cup {s}^{-}_{i'}) = g({s}^{-}_{i'})(u) = {\sigma}_{1}(u)\]$.
  Either $\tau \in {S}^{0}_{k}$ or $\tau \in {S}^{1}_{k}$.
  If $\tau \in {S}^{0}_{k}$, then this contradicts the choice of ${\sigma}_{i, 0}$ because $\exists u \in \Pset(l)\[{\sigma}_{0}(u) \neq \tau(u \cup {s}^{-}_{i})\]$.
  If $\tau \in {S}^{1}_{k}$, then this contradicts the choice of ${\sigma}_{i', 1}$ because $\exists u \in \Pset(l)\[{\sigma}_{1}(u) \neq \tau(u \cup {s}^{-}_{i'})\]$.
  The contradiction proves the claim.
 \end{proof}
 Therefore, there exists ${t}_{l} \in 2$ such that ${c}_{l}$ is constantly ${t}_{l}$ on ${A}_{l}$.
 As this was for every $l \in \omega$, there exists $t \in 2$ such that $M = \{l \in \omega: {t}_{l} = t\} \in {\HHH}_{\max}$.
 Define $A = \[X\] \in \HHH$.
 For each $l \in M$, ${A}_{l} \in \HHH$, ${A}_{l} \subseteq A$, and for any $s \in {A}_{l}$ with $k = \max(s)$, ${S}^{t}_{k}$ is $\pr{l}{s}$-big because ${c}_{l}(s) = t$.
 For any $l \in M$, any $k \in {f}_{\max}\[{A}_{l}\] \subseteq {f}_{\max}\[\[X\]\]$, and any $\sigma \in {S}^{t}_{k}$, ${T}_{q}\langle {{e}_{k}}^{\frown}{\langle \sigma \rangle} \rangle \; {\forces}_{\PP(\HHH)} \; {\mathring{x}(k) = {\eta}_{{j}_{k}, {{e}_{k}}^{\frown}{\langle \sigma \rangle}}(k) = t}$.
 Thus (1)--(3) are satisfied.
\end{proof}
The next lemma is true for any partial order.
It has a simple proof, which is left to the reader.
\begin{Lemma} \label{lem:densesplit}
 Suppose $D \subseteq \PP(\HHH)$ is dense and $D = {D}_{0} \cup {D}_{1}$.
 Then for any ${p}^{\ast} \in \PP(\HHH)$, there exist $p \leq {p}^{\ast}$ and $t \in 2$ such that ${D}_{t}$ is dense below $p$.
\end{Lemma}
\begin{Lemma} \label{lem:hmax2}
 Suppose $\mathring{x}$ is a $\PP(\HHH)$-name such that ${\forces}_{\PP(\HHH)} \; {\mathring{x}: \omega \rightarrow 2}$.
 Then for any ${p}^{\ast} \in \PP(\HHH)$, there exists $q \leq {p}^{\ast}$ such that there are $t \in 2$ and $X \in {\FIN}^{\[\omega\]}$ such that $\[X\] \in \HHH$ and for every $i \in \omega$ and every $e \in \lv{{T}_{q}}{\max\left(X(i)\right)+1}$, there exists ${e}^{\ast} \in \lv{{T}_{q}}{\max\left(X(i+1)\right)}$ so that $e \subseteq {e}^{\ast}$, there exists $S \subseteq {\succc}_{{T}_{q}}({e}^{\ast})$ such that $S$ is $\pr{\max\left(X(i)\right)+1}{X(i+1)}$-big, and for each $\sigma \in S$,
 \begin{align*}
  {T}_{q}\langle {{e}^{\ast}}^{\frown}{\langle \sigma \rangle} \rangle \; {\forces}_{\PP(\HHH)} \; {\mathring{x}(\max\left(X(i+1)\right)) = t}.
 \end{align*}
\end{Lemma}
\begin{proof}
 For $t \in 2$, let ${D}_{t}$ be the collection of all $q \in \PP(\HHH)$ for which there exist $A \in \HHH$, $M \in {\HHH}_{\max}$, $\seq{A}{l}{\in}{M}$, and $\left\langle \left\langle {e}_{k}, {S}^{0}_{k}, {S}^{1}_{k} \right\rangle: k \in {f}_{\max}\[A\] \right\rangle$ such that (1)--(3) of Lemma \ref{lem:hmax1} are satisfied w.r.t.\@ $t$.
 By Lemma \ref{lem:hmax1}, ${D}_{0} \cup {D}_{1}$ is dense.
 Hence by Lemma \ref{lem:densesplit}, there exist $p \leq {p}^{\ast}$ and $t \in 2$ such that ${D}_{t}$ is dense below $p$.
 
 Now define a strategy $\Sigma$ for Player I in ${\Game}^{\mathtt{Cond}}\left(\HHH\right)$ as follows.
 ${p}_{0} = p \in \PP(\HHH)$ and ${A}_{0} = \FIN \in \HHH$.
 Define $\Sigma(\emptyset) = \pr{{p}_{0}}{{A}_{0}}$.
 Suppose that $i \in \omega$ and that $\left\langle \pr{\pr{{p}_{j}}{{A}_{j}}}{{s}_{j}}: j \leq i \right\rangle$ is a partial run of ${\Game}^{\mathtt{Cond}}\left(\HHH\right)$ in which Player I has followed $\Sigma$.
 Let $k = \max({s}_{i})+1$, for ease of notation.
 For any $e \in \lv{{T}_{{p}_{i}}}{k}$, ${T}_{{p}_{i}}\langle e \rangle \leq {p}_{i} \leq {p}_{0} = p$.
 Since ${D}_{t}$ is dense below $p$, there exists ${p}_{i, e} \leq {T}_{{p}_{i}}\langle e \rangle$ such that there are ${A}_{e} \in \HHH$, ${M}_{e} \in {\HHH}_{\max}$, $\left\langle {A}_{e, l}: l \in {M}_{e} \right\rangle$, and $\left\langle \left\langle {h}_{e, l}, {S}^{0}_{e, l}, {S}^{1}_{e, l} \right\rangle: l \in {f}_{\max}\[{A}_{e}\] \right\rangle$ satisfying (1)--(3) of Lemma \ref{lem:hmax1}.
 Choose $k' \in \bigcap\{{M}_{e}: e \in \lv{{T}_{{p}_{i}}}{k}\} \in {\HHH}_{\max}$ with $k \leq k'$.
 Let ${A}^{\ast} = \bigcap\{{A}_{e, k'}: e \in \lv{{T}_{{p}_{i}}}{k}\} \in \HHH$ and let ${A}_{i+1} = \{s \in {A}^{\ast}: \min(s) > k'\} \in \HHH$.
 Define ${p}_{i+1} = {T}_{{p}_{i+1}} = \bigcup{\left\{ {T}_{{p}_{i, e}}: e \in \lv{{T}_{{p}_{i}}}{k} \right\}}$.
 Note that ${p}_{i+1} \in \PP(\HHH)$ by Lemma \ref{lem:amalgam}.
 Define $\Sigma\left(\left\langle \pr{\pr{{p}_{j}}{{A}_{j}}}{{s}_{j}}: j \leq i \right\rangle\right) = \pr{{p}_{i+1}}{{A}_{i+1}}$.
 Observe that if $s \in {A}_{i+1}$, then for any $e \in \lv{{T}_{{p}_{i}}}{k}$, $s \in {A}_{e, k'} \subseteq {A}_{e}$.
 Letting $l = \max(s) \in {f}_{\max}\[{A}_{e, k'}\] \subseteq {f}_{\max}\[{A}_{e}\]$, ${h}_{e, l} \in \lv{{T}_{{p}_{i, e}}}{l}$, and since $\dom(e) = k \leq k' < \min(s) \leq \max(s) = l = \dom({h}_{e, l})$, $e \subseteq {h}_{e, l}$.
 Further, ${S}^{t}_{l} \subseteq {\succc}_{{T}_{{p}_{i, e}}}({h}_{e, l})$ and ${S}^{t}_{l}$ is $\pr{k'}{s}$-big.
 For any $\sigma \in {S}^{t}_{l}$, ${T}_{{p}_{i, e}}\langle {{h}_{e, l}}^{\frown}{\langle \sigma \rangle} \rangle \; {\forces}_{\PP(\HHH)} \; {\mathring{x}(l) = t}$.
 Note that ${h}_{e, l} \in \lv{{T}_{{p}_{i+1}}}{l}$ and that ${\succc}_{{T}_{{p}_{i+1}}}({h}_{e, l}) = {\succc}_{{T}_{{p}_{i, e}}}({h}_{e, l})$.
 Therefore, ${S}^{t}_{l} \subseteq {\succc}_{{T}_{{p}_{i+1}}}({h}_{e, l})$ and ${S}^{t}_{l}$ is $\pr{k'}{s}$-big.
 By Lemma \ref{lem:bigmonotone}, ${S}^{t}_{l}$ is $\pr{k}{s}$-big.
 Also since for every $\sigma \in {S}^{t}_{l}$, ${T}_{{p}_{i+1}}\langle {{h}_{e, l}}^{\frown}{\langle \sigma \rangle} \rangle \leq {T}_{{p}_{i, e}}\langle {{h}_{e, l}}^{\frown}{\langle \sigma \rangle} \rangle$, ${T}_{{p}_{i+1}}\langle {{h}_{e, l}}^{\frown}{\langle \sigma \rangle} \rangle \; {\forces}_{\PP(\HHH)} \; {\mathring{x}(l) = t}$.
 The completes the definition of $\Sigma$.
 
 Since $\Sigma$ is not a winning strategy for Player I, there is a run $\left\langle \pr{\pr{{p}_{i}}{{A}_{i}}}{{s}_{i}}: i < \omega \right\rangle$ of ${\Game}^{\mathtt{Cond}}\left(\HHH\right)$ in which Player I has followed $\Sigma$ and lost.
 Then $X = \{{s}_{i}: i < \omega\} \in {\FIN}^{\[\omega\]}$ and $\[X\] \in \HHH$.
 Furthermore, $X(i) = {s}_{i}$, for all $i \in \omega$.
 Again, for ease of notation, denote $\max({s}_{i})+1$ by ${k}_{i}$.
 Let $q = {T}_{q} = {\bigcap}_{i \in \omega}{{T}_{{p}_{i}}}$.
 Then $q \in \PP(\HHH)$ and $q \leq {p}_{0} = p$.
 Consider $i \in \omega$ and $e \in \lv{{T}_{q}}{{k}_{i}}$.
 Then $e \in \lv{{T}_{{p}_{i}}}{{k}_{i}}$.
 Since ${s}_{i+1} \in {A}_{i+1}$, then as noted above, letting $l = \max({s}_{i+1}) = \max\left(X(i+1)\right)$, there exists ${e}^{\ast} \in \lv{{T}_{{p}_{i+1}}}{l}$ such that $e \subseteq {e}^{\ast}$, there exists $S \subseteq {\succc}_{{T}_{{p}_{i+1}}}({e}^{\ast})$ such that $S$ is $\pr{{k}_{i}}{{s}_{i+1}}$-big, and for every $\sigma \in S$, ${T}_{{p}_{i+1}}\langle {{e}^{\ast}}^{\frown}{\langle \sigma \rangle} \rangle \; {\forces}_{\PP(\HHH)} \; {\mathring{x}(l) = t}$.
 As ${e}^{\ast} \in {T}_{{p}_{i+1}}$ and $\dom({e}^{\ast}) = l \leq \max({s}_{i+1})+1$, ${e}^{\ast} \in \lv{{T}_{q}}{\max\left(X(i+1)\right)}$ by Lemma \ref{lem:winning1}.
 Similarly, if $\sigma \in {\succc}_{{T}_{{p}_{i+1}}}({e}^{\ast})$, then ${{e}^{\ast}}^{\frown}{\langle \sigma \rangle} \in {T}_{{p}_{i+1}}$ and $\dom\left({{e}^{\ast}}^{\frown}{\langle \sigma \rangle}\right) = l+1 = \max({s}_{i+1})+1$, ${{e}^{\ast}}^{\frown}{\langle \sigma \rangle} \in {T}_{q}$, whence $\sigma \in {\succc}_{{T}_{q}}({e}^{\ast})$.
 Therefore, $S \subseteq {\succc}_{{T}_{{p}_{i+1}}}({e}^{\ast}) \subseteq {\succc}_{{T}_{q}}({e}^{\ast})$ and $S$ is $\pr{\max\left(X(i)\right)+1}{X(i+1)}$-big.
 Finally, for any $\sigma \in S$, ${T}_{q}\langle {{e}^{\ast}}^{\frown}{\langle \sigma \rangle} \rangle \leq {T}_{{p}_{i+1}}\langle {{e}^{\ast}}^{\frown}{\langle \sigma \rangle} \rangle$, and so ${T}_{q}\langle {{e}^{\ast}}^{\frown}{\langle \sigma \rangle} \rangle \; {\forces}_{\PP(\HHH)} \; {\mathring{x}\left(\max\left(X(i+1)\right)\right) = t}$.
 This is as required.
\end{proof}
\begin{Lemma} \label{lem:bigdecision}
 Suppose $\mathring{x}$ is a $\PP(\HHH)$-name such that ${\forces}_{\PP(\HHH)} \; {\mathring{x}: \omega \rightarrow 2}$.
 Suppose there are $p, t, X,$ and $N$ such that:
 \begin{enumerate}[series=bigdecision]
  \item
  $p \in \PP(\HHH)$, $t \in 2$, $X \in {\FIN}^{\[\omega\]}$ with $\[X\] \in \HHH$, and $N \in {\[\omega\]}^{\omega}$;
  \item
  letting $\seq{n}{i}{\in}{\omega}$ be the strictly increasing enumeration of $N$, for each $i \in \omega$, for each $e \in \lv{{T}_{p}}{\max\left(X(i)\right)+1}$, there exists ${e}^{\ast} \in \lv{{T}_{p}}{\max\left(X(i+1)\right)}$ so that $e \subseteq {e}^{\ast}$, there exists $S \subseteq {\succc}_{{T}_{p}}({e}^{\ast})$ such that
  \begin{align*}
   S \ \text{is} \ \pr{\max\left(X(i)\right)+1}{X(i+1)}\text{-big,}
  \end{align*}
  and for each $\sigma \in S$, ${T}_{p}\langle {{e}^{\ast}}^{\frown}{\langle \sigma \rangle} \rangle \; {\forces}_{\PP(\HHH)} \; {\mathring{x}({n}_{i+1}) = t}$.
 \end{enumerate}
 Then there exists $q \leq p$ such that $\forall 1 \leq i < \omega \[q \; {\forces}_{\PP(\HHH)} \; {\mathring{x}({n}_{i}) = t}\]$.
\end{Lemma}
\begin{proof}
 Lemma \ref{lem:fusion} will be used to get $q$.
 To this end, a sequence $\seq{p}{i}{\in}{\omega}$ having the following properties will be constructed:
 \begin{enumerate}[resume=bigdecision]
  \item
  for each $i \in \omega$, ${p}_{i} \in \PP(\HHH)$ and ${p}_{i+1} \leq {p}_{i}$;
  \item
  for each $i \in \omega$, $X(i+1) \in {H}_{{p}_{i+1}, \max\left(X(i)\right)+1}$;
  \item
  for each $j \in \omega$, for each $e \in {T}_{{p}_{j}}$, if $\dom(e) \geq \max\left(X(j)\right)+1$, then ${T}_{{p}_{j}}\langle e \rangle = {T}_{{p}_{0}}\langle e \rangle$;
  \item
  for each $i \leq j < \omega$ and for each $e \in \lv{{T}_{{p}_{j}}}{\max\left(X(i)\right)}$, ${\succc}_{{T}_{{p}_{i}}}(e) \subseteq {\succc}_{{T}_{{p}_{j}}}(e)$;
  \item
  ${p}_{0} = p$ and for each $1 \leq i < \omega$, ${p}_{i} \; {\forces}_{\PP(\HHH)} \; {\mathring{x}({n}_{i}) = t}$.
 \end{enumerate}
 Assume for a moment that such a sequence has been constructed.
 By Lemma \ref{lem:fusion}, $q = {T}_{q} = {\bigcap}_{i \in \omega}{{T}_{{p}_{i}}} \in \PP(\HHH)$.
 And $q \leq {T}_{{p}_{0}} = {p}_{0} = p$.
 For any $1 \leq i < \omega$, $q \leq {p}_{i}$, whence by (7), $q \; {\forces}_{\PP(\HHH)} \; {\mathring{x}({n}_{i}) = t}$, as required.
 
 To construct $\seq{p}{i}{\in}{\omega}$, proceed by induction.
 Define ${p}_{0} = p$ and note that (3)--(7) are satisfied.
 Suppose ${p}_{j}$ satisfying (3)--(7) is given.
 For each
 \begin{align*}
  e \in \lv{{T}_{{p}_{j}}}{\max\left(X(j)\right)+1},
 \end{align*}
 since ${T}_{{p}_{j}}\langle e \rangle = {T}_{{p}_{0}}\langle e \rangle$, then by (2), there exist ${e}^{\ast} \in \lv{{T}_{{p}_{j}}}{\max\left(X(j+1)\right)}$ and $S \subseteq {\succc}_{{T}_{{p}_{j}}}({e}^{\ast})$ having the properties listed in (2).
 Let $\{{e}_{l}: 1 \leq l \leq L\}$ be a 1-1 enumeration of $\lv{{T}_{{p}_{j}}}{\max\left(X(j)\right)+1}$, where $L = \lc \lv{{T}_{{p}_{j}}}{\max\left(X(j)\right)+1} \rc$.
 Choose $\{{e}^{\ast}_{l}: 1 \leq l \leq L\}$ and $\{{S}_{l}: 1 \leq l \leq L\}$ such that each ${e}^{\ast}_{l}$ and ${S}_{l}$ satisfy the conditions of (2) with respect to ${e}_{l}$.
 The bigness of ${S}_{l}$ implies that it is non-empty, and so $U = \{{{e}^{\ast}_{l}}^{\frown}{\langle \sigma \rangle}: 1 \leq l \leq L \ \text{and} \ \sigma \in {S}_{l}\} \subseteq \lv{{T}_{{p}_{j}}}{\max\left(X(j+1)\right)+1}$ is non-empty.
 Define ${p}_{j+1} = {T}_{{p}_{j+1}} = \bigcup\left\{{T}_{{p}_{j}}\langle h \rangle: h \in U \right\}$.
 By Lemma \ref{lem:amalgam}, ${p}_{j+1} \in \PP(\HHH)$, and ${p}_{j+1} \leq {p}_{j}$.
 Note that if $h \in U$, then $h \in {T}_{{p}_{j}}$ and $\dom(h) = \max\left(X(j+1)\right)+1 \geq \max\left(X(j)\right)+1$, and so ${T}_{{p}_{j}}\langle h \rangle = {T}_{{p}_{0}}\langle h \rangle$.
 Observe also that if $f \in \lv{{T}_{{p}_{j+1}}}{\max\left(X(j+1)\right)}$, then $f = {e}^{\ast}_{l}$, for some $1 \leq l \leq L$.
 To verify (4), first note that $\max\left(X(j)\right)+1 \leq \max\left(X(j+1)\right)$.
 Consider any $f \in \lv{{T}_{{p}_{j+1}}}{\max\left(X(j+1)\right)}$.
 Then $f = {e}^{\ast}_{l}$, for some $1 \leq l \leq L$, and so for any $\sigma \in {S}_{l}$, ${f}^{\frown}{\langle \sigma \rangle} \in {T}_{{p}_{j+1}}$, whence ${S}_{l} \subseteq {\succc}_{{T}_{{p}_{j+1}}}(f) \subseteq {2}^{\Pset\left(\max\left(X(j+1)\right)\right)}$.
 As ${S}_{l}$ is $\pr{\max\left(X(j)\right)+1}{X(j+1)}$-big, so is ${\succc}_{{T}_{{p}_{j+1}}}(f)$.
 Therefore, $X(j+1) \in {H}_{{p}_{j+1}, \max\left(X(j)\right)+1}$.
 Next, consider any $e \in {T}_{{p}_{j+1}}$ with $\dom(e) \geq \max\left(X(j+1)\right)+1$.
 Then $h \subseteq e$ for some $h \in U$.
 Hence ${T}_{{p}_{0}}\langle e \rangle \subseteq {T}_{{p}_{0}}\langle h \rangle = {T}_{{p}_{j}}\langle h \rangle \subseteq {T}_{{p}_{j+1}}$, whence ${T}_{{p}_{0}}\langle e \rangle \subseteq {T}_{{p}_{j+1}}\langle e \rangle \subseteq {T}_{{p}_{0}}\langle e \rangle$.
 So ${T}_{{p}_{0}}\langle e \rangle = {T}_{{p}_{j+1}}\langle e \rangle$, as required for (5).
 For (6), fix $i \leq j+1$ and $e \in \lv{{T}_{{p}_{j+1}}}{\max\left(X(i)\right)}$.
 If $i = j+1$, then there is nothing to prove.
 So assume $i \leq j$.
 Suppose $\sigma \in {\succc}_{{T}_{{p}_{i}}}(e)$.
 Since $e \in \lv{{T}_{{p}_{j}}}{\max\left(X(i)\right)}$, the induction hypothesis says $\sigma \in {\succc}_{{T}_{{p}_{j}}}(e)$, whence ${e}^{\frown}{\langle \sigma \rangle} \in {T}_{{p}_{j}}$.
 As $\max\left(X(i)\right)+1 \leq \max\left(X(j)\right)+1$, there exists $1 \leq l \leq L$ with ${e}^{\frown}{\langle \sigma \rangle} \subseteq {e}_{l} \subseteq {e}^{\ast}_{l}$, and as ${S}_{l} \neq \emptyset$, there is some $\tau \in {S}_{l}$ with ${e}^{\frown}{\langle \sigma \rangle} \subseteq {{e}^{\ast}_{l}}^{\frown}{\langle \tau \rangle} \in U$.
 Thus for some $h \in U$, ${e}^{\frown}{\langle \sigma \rangle} \in {T}_{{p}_{j}}\langle h \rangle$, whence ${e}^{\frown}{\langle \sigma \rangle} \in {T}_{{p}_{j+1}}$, and so $\sigma \in {\succc}_{{T}_{{p}_{j+1}}}(e)$.
 Therefore, ${\succc}_{{T}_{{p}_{i}}}(e) \subseteq {\succc}_{{T}_{{p}_{j+1}}}(e)$, as needed for (6).
 Finally for (7), suppose $r \leq {p}_{j+1}$.
 Choose $h \in \lv{{T}_{r}}{\max\left(X(j+1)\right)+1}$.
 Then $h \in U$ and $h = {{e}^{\ast}_{l}}^{\frown}{\langle \sigma \rangle}$ for some $1 \leq l \leq L$ and $\sigma \in {S}_{l}$.
 Since ${T}_{r}\langle {{e}^{\ast}_{l}}^{\frown}{\langle \sigma \rangle} \rangle \leq {T}_{p}\langle {{e}^{\ast}_{l}}^{\frown}{\langle \sigma \rangle} \rangle$, ${T}_{r}\langle {{e}^{\ast}_{l}}^{\frown}{\langle \sigma \rangle} \rangle \; {\forces}_{\PP(\HHH)} \; {\mathring{x}({n}_{j+1})=t}$.
 Therefore, $\forall r \leq {p}_{j+1} \exists r' \leq r\[r' \; {\forces}_{\PP(\HHH)} \; {\mathring{x}({n}_{j+1})=t}\]$, whence ${p}_{j+1} \; {\forces}_{\PP(\HHH)} \; {\mathring{x}({n}_{j+1})=t}$.
\end{proof}
\begin{Lemma} \label{lem:preserving}
 Suppose $\PP$ is a forcing notion and $\VVV$ is an ultrafilter on $\omega$.
 The following hold:
 \begin{enumerate}
  \item
  Suppose for every $\PP$-name $\mathring{x}$ such that ${\forces}_{\PP} \; {\mathring{x}: \omega \rightarrow 2}$ and for every $p \in \PP$, there exist $q \leq p$ and $A \in \VVV$ such that $\forall l \in A\exists {t}_{l} \in 2\[q \; {\forces}_{\PP} \; {\mathring{x}(l) = {t}_{l}}\]$, then $\PP$ preserves $\VVV$.
  \item
  If $\PP$ preserves $\VVV$ and $\UUU$ is an ultrafilter on $\omega$ such that $\UUU \; {\leq}_{RK} \; \VVV$, then $\PP$ preserves $\UUU$.
 \end{enumerate}
\end{Lemma}
\begin{proof}
 For (1): let $\mathring{A}$ be a $\PP$-name such that ${\forces}_{\PP} \; {\mathring{A} \subseteq \omega}$.
 Let $\mathring{x}$ be a $\PP$-name such that ${\forces}_{\PP} \; {\text{``}\mathring{x}: \omega \rightarrow 2 \ \text{and} \ \mathring{A} = \{l \in \omega: \mathring{x}(l) = 1\}\text{''}}$.
 Suppose $p \in \PP$.
 Find $q \leq p$ and $A \in \VVV$ such that $\forall l \in A\exists {t}_{l} \in 2\[q \; {\forces}_{\PP} \; {\mathring{x}(l) = {t}_{l}}\]$.
 As $\VVV$ is an ultrafilter, there are $t \in 2$ and $B \in \VVV$ such that $B \subseteq A$ and $\forall l \in B\[{t}_{l} = t\]$.
 If $t = 0$, then $q \; {\forces}_{\PP} \; {B \subseteq \omega \setminus \mathring{A}}$, while if $t = 1$, then $q \; {\forces}_{\PP} \; {B \subseteq \mathring{A}}$, which shows $\PP$ preserves $\VVV$.
 
 For (2): suppose $f: \omega \rightarrow \omega$ witnesses $\UUU \; {\leq}_{RK} \; \VVV$ in the ground model.
 Let $G$ be $(\V, \PP)$-generic.
 In $\VG$, suppose that $A \subseteq \omega$.
 Since $\VVV$ is preserved, there exists $B \in \VVV$ such that either $B \subseteq {f}^{-1}(A)$ or $B \subseteq \omega \setminus {f}^{-1}(A) = {f}^{-1}(\omega \setminus A)$.
 Since $B \in \VVV$ and $f$ is an RK-map in $\V$, $f\[B\] \in \UUU$.
 Now either $f\[B\] \subseteq A$ or $f\[B\] \subseteq \omega \setminus A$, showing that $\UUU$ is preserved.
\end{proof}
\begin{Cor} \label{cor:H_max}
 Suppose $\UUU$ is an ultrafilter such that $\UUU \; {\equiv}_{RK} \; {\HHH}_{\max}$.
 Then $\PP(\HHH)$ preserves $\UUU$.
\end{Cor}
\begin{proof}
 Suppose $\mathring{x}$ is a $\PP(\HHH)$-name such that ${\forces}_{\PP(\HHH)} \; {\mathring{x}: \omega \rightarrow 2}$ and that ${p}^{\ast} \in \PP(\HHH)$.
 Applying Lemma \ref{lem:hmax2}, find $p \leq {p}^{\ast}$, $t \in 2$, and $X \in {\FIN}^{\[\omega\]}$ satisfying the conclusions of that lemma.
 As $\[X\] \in \HHH$, so $N = {f}_{\max}\[\[X\]\] \in {\HHH}_{\max}$.
 Let $\seq{n}{i}{\in}{\omega}$ be the strictly increasing enumeration of $N$.
 Then $\forall i \in \omega\[{n}_{i} = \max\left(X(i)\right)\]$.
 Thus the hypotheses of Lemma \ref{lem:bigdecision} are satisfied.
 So there exists $q \leq p \leq {p}^{\ast}$ such that $\forall 1 \leq i < \omega\[q \; {\forces}_{\PP(\HHH)} \; {\mathring{x}({n}_{i}) = t}\]$.
 Let $A = \{{n}_{i}: 1 \leq i < \omega\}$.
 Then $A \in {\HHH}_{\max}$ and $\forall l \in A\[q \; {\forces}_{\PP(\HHH)} \; {\mathring{x}(l) = t}\]$.
 Now (1) of Lemma \ref{lem:preserving} implies that $\PP(\HHH)$ preserves ${\HHH}_{\max}$.
 Hence by (2) of Lemma \ref{lem:preserving} $\PP(\HHH)$ also preserves every ultrafilter $\UUU$ such that $\UUU \; {\equiv}_{RK} \; {\HHH}_{\max}$.
\end{proof}
\begin{Lemma} \label{lem:refiningX}
 Suppose $\UUU$ is a selective ultrafilter on $\omega$ with $\UUU \notin {\C}_{0}(\HHH)$.
 Then for any $L \in \HHH$ and $E \in \UUU$, there exist $Y \in {\FIN}^{\[\omega\]}$ and $D \subseteq E$ such that $\[Y\] \in \HHH$, $\[Y\] \subseteq L$, $D \in \UUU$, and letting $\seq{n}{i}{\in}{\omega}$ be the strictly increasing enumeration of $D$, $\forall i \in \omega\[{n}_{i} \in I\left(Y(i)\right)\]$.
\end{Lemma}
\begin{proof}
 As $\UUU$ is selective, but $\UUU \notin {\C}_{0}(\HHH)$, there exists $X \in {\FIN}^{\[\omega\]}$ such that $\[X\] \in \HHH$ and for every $Z \in {\FIN}^{\[\omega\]}$, if $\[Z\] \in \HHH$, $\[Z\] \subseteq \[X\]$, then $N(Z) \in \UUU$.
 Let $K = L \cap \[X\] \in \HHH$.
 Define $c: K \rightarrow 2$ by setting $c(s) = 0$ if and only if $I(s) \cap E \neq \emptyset$.
 Fix $Z \in {\FIN}^{\[\omega\]}$ such that $\[Z\] \subseteq K$, $\[Z\] \in \HHH$, and $c$ is constant on $\[Z\]$.
 This constant value cannot be $1$.
 To see this, suppose for a contradiction that $c$ is constantly $1$ on $\[Z\]$.
 As $E$ is an infinite set, choose $m \in E$ with $\min\left(Z(0)\right) \leq m$.
 Choose $0 < i < \omega$ such that $m \leq \max\left(Z(i)\right)$.
 Let $s = Z(0) \cup Z(i) \in \[Z\]$.
 Then $\min(s) = \min\left(Z(0)\right) \leq m \leq \max\left(Z(i)\right) = \max(s)$, whence $m \in I(s) \cap E$, contradicting the supposition that $c$ is constantly $1$ on $\[Z\]$.
 Thus $c$ is constantly $0$ on $\[Z\]$.
 By the choice of $X$, $N(Z) \in \UUU$.
 Let $C = E \cap N(Z) \in \UUU$.
 For any $i \in \omega$, since $c\left(Z(i)\right) = 0$, there is $m \in I\left(Z(i)\right) \cap E$.
 By the definition of $N(Z)$, $m \in I\left(Z(i)\right) \cap E \cap N(Z) = I\left(Z(i)\right) \cap C$, and so $\forall i \in \omega\[I\left(Z(i)\right) \cap C \neq \emptyset\]$.
 Since $C \subseteq N(Z)$, $\UUU$ is a Q-point and $\forall i < {i}^{\ast} < \omega\[I\left(Z(i)\right) \cap I\left(Z({i}^{\ast})\right) = \emptyset\]$, there exists $D \in \UUU$ such that $D \subseteq C \subseteq N(Z)$ and $\forall i \in \omega\[\lc I\left(Z(i)\right) \cap D \rc = 1\]$.
 Let $\seq{n}{i}{\in}{\omega}$ be the strictly increasing enumeration of $D$.
 As $Z \in {\FIN}^{\[\omega\]}$ and as $\lc I\left(Z(i)\right) \cap D \rc = 1$, it follows that ${n}_{i} \in I\left(Z(i)\right)$, for all $i \in \omega$.
 Hence $Z$ and $D$ are as required.
\end{proof}
\begin{Cor} \label{cor:hitsD}
  Suppose $\UUU$ is a selective ultrafilter on $\omega$ with $\UUU \notin {\C}_{0}(\HHH)$.
 Then for any $L \in \HHH$ and $E \in \UUU$, $\{s \in L: I(s) \cap E \neq \emptyset\} \in \HHH$.
\end{Cor}
\begin{proof}
 Fix $Y \in {\FIN}^{\[\omega\]}$ and $D \subseteq E$ as in Lemma \ref{lem:refiningX}.
 Let $\seq{n}{i}{\in}{\omega}$ be the strictly increasing enumeration of $D$.
 If $s \in \[Y\]$, then $s \in L$ and $Y(i) \subseteq s$, for some $i \in \omega$.
 So ${n}_{i} \in I\left(Y(i)\right) \subseteq I(s)$.
 As $D \subseteq E$, ${n}_{i} \in E \cap I(s)$.
 Thus $\[Y\] \subseteq \{s \in L: I(s) \cap E \neq \emptyset\} \subseteq L \subseteq \FIN$.
 As $\[Y\] \in \HHH$, $\{s \in L: I(s) \cap E \neq \emptyset\} \in \HHH$.
\end{proof}
\begin{Lemma}\label{lem:c01}
 Suppose $\UUU$ is a selective ultrafilter on $\omega$ such that $\UUU \notin {\C}_{0}(\HHH)$ and $\UUU \; {\not\equiv}_{RK} \; {\HHH}_{\max}$.
 Suppose $\mathring{x}$ is a $\PP(\HHH)$-name such that ${\forces}_{\PP(\HHH)} \; {\mathring{x}: \omega \rightarrow 2}$.
 For any $p \in \PP(\HHH)$, there exists $q \leq p$ such that there are $t \in 2$, $A \in \HHH$, $Y \in {\FIN}^{\[\omega\]}$, $D \in \UUU$, $\seq{A}{l}{\in}{\omega}$, and $\seq{e}{k}{\in}{{f}_{\max}\[A\]}$ such that:
 \begin{enumerate}
  \item
  ${e}_{k} \in \lv{{T}_{q}}{k}$;
  \item
  for any $l \in \omega$, ${A}_{l} \in \HHH$, ${A}_{l} \subseteq A$, and for any $s \in {A}_{l}$ with $\max(s) = k$, ${\succc}_{{T}_{q}}({e}_{k})$ is $\pr{l}{s}$-big;
  \item
  for any $s \in A$, $I(s) \cap D \neq \emptyset$, and letting $k = \max(s)$, for all $n \in I(s) \cap D$, ${T}_{q}\langle {e}_{k} \rangle \; {\forces}_{\PP(\HHH)} \; {\mathring{x}(n) = t}$;
  \item
  $A = \[Y\]$ and for all $Z \in {\FIN}^{\[\omega\]}$, if $\[Z\] \in \HHH$ and $\[Z\] \subseteq \[Y\]$, then $N(Z) \in \UUU$.
 \end{enumerate}
\end{Lemma}
\begin{proof}
 Since $\UUU \notin {\C}_{0}(\HHH)$, but $\UUU$ is a selective ultrafilter, there exists ${Y}^{\ast} \in {\FIN}^{\[\omega\]}$ such that $\[{Y}^{\ast}\] \in \HHH$ and for all $Z \in {\FIN}^{\[\omega\]}$, if $\[Z\] \in \HHH$ and $\[Z\] \subseteq \[{Y}^{\ast}\]$, then $N(Z) \in \UUU$.
 Apply Lemma \ref{lem:continuous1} to find $q, X$, and $\left\langle {\eta}_{i, e}: i \in \omega \wedge e \in \lv{{T}_{q}}{\max\left(X(i)\right)+1} \right\rangle$ satisfying (1)--(4) of that lemma.
 For any $k \in {f}_{\max}\[\[X\]\]$, there exists a unique ${j}_{k} \in \omega$ with $k = \max\left(X({j}_{k})\right)$.
 Choose an infinite branch $F$ through ${T}_{q}$ -- that is, choose $F \in \[{T}_{q}\]$.
 For each $k \in {f}_{\max}\[\[X\]\]$, define ${e}_{k} = F \restrict k \in \lv{{T}_{q}}{k}$.
 Define $\psi: \omega \rightarrow \omega$ by $\psi(n) = \min\left\{ k \in {f}_{\max}\[\[X\]\]: k \geq n \right\}$, for all $n \in \omega$.
 For each $n \in \omega$, ${\eta}_{{j}_{\psi(n)}, F\restrict \psi(n)+1}$ is defined at $n$ and is a member of $2$.
 Hence there are $t \in 2$ and $B \in \UUU$ such that $\forall n \in B\[{\eta}_{{j}_{\psi(n)}, F\restrict \psi(n)+1}(n)=t\]$.
 Next, since ${\HHH}_{\max} \; {\not\leq}_{RK} \; \UUU$, there exists $C \in \UUU$ with $\psi\[C\] \notin {\HHH}_{\max}$.
 Thus $E = B \cap C \in \UUU$ and $L = \[{Y}^{\ast}\] \cap \[X\] \cap \{s \in \FIN: \max(s) \in \omega \setminus \psi\[C\]\} \in \HHH$.
 Apply Lemma \ref{lem:refiningX} to find $Y \in {\FIN}^{\[\omega\]}$ and $D \subseteq E$ satisfying the conclusions of that lemma.
 Define $A = \[Y\] \in \HHH$.
 As $q \in \PP(\HHH)$, for each $l \in \omega$, ${H}_{q, l} \in \HHH$, and so ${A}_{l} = A \cap {H}_{q, l} \in \HHH$.
 For each $k \in {f}_{\max}\[A\] \subseteq {f}_{\max}\[\[X\]\]$, ${e}_{k} \in \lv{{T}_{q}}{k}$ by definition.
 For any $l \in \omega$ and for any $s \in {A}_{l}$ with $\max(s)=k$, since $s \in {H}_{q, l}$ and ${e}_{k} \in \lv{{T}_{q}}{\max(s)}$, ${\succc}_{{T}_{q}}({e}_{k})$ is $\pr{l}{s}$-big.
 Hence (1) and (2) hold.
 For (3), let $\seq{n}{i}{\in}{\omega}$ be the strictly increasing enumeration of $D$.
 Then $\forall i \in \omega\[{n}_{i} \in I\left(Y(i)\right)\]$.
 For any $s \in A$, $Y(i) \subseteq s$, for some $i \in \omega$, whence ${n}_{i} \in I\left(Y(i)\right) \subseteq I(s)$, whence $I(s) \cap D \neq \emptyset$.
 Now consider any $s \in A$, let $k = \max(s)$, and fix $n \in I(s) \cap D$.
 Then $n \in C$ and $\psi(n) \in \psi\[C\]$.
 On the other hand, $k \in \omega \setminus \psi\[C\]$, whence $k \neq \psi(n)$.
 As $n \in I(s)$, $n \leq \max(s) = k$.
 As $s \in \[X\]$ and $k \in {f}_{\max}\[\[X\]\]$, the minimality of $\psi(n)$ implies that $\psi(n) \leq k$.
 Therefore $\psi(n) < k$ and $\psi(n)+1 \leq k$, whence $F\restrict \psi(n)+1 \subseteq F \restrict k = {e}_{k}$.
 Therefore ${T}_{q}\langle {e}_{k} \rangle \leq {T}_{q}\langle F\restrict \psi(n)+1 \rangle$.
 Since ${T}_{q}\langle F\restrict \psi(n)+1 \rangle \; {\forces}_{\PP(\HHH)} \; {\mathring{x}(n) = {\eta}_{{j}_{\psi(n)}, F\restrict \psi(n)+1}(n)}$ and since $n \in B$, ${T}_{q}\langle {e}_{k} \rangle \; {\forces}_{\PP(\HHH)} \; {\mathring{x}(n) = t}$, as needed for (3).
 For (4), if $Z \in {\FIN}^{\[\omega\]}$, $\[Z\] \in \HHH$, and $\[Z\] \subseteq \[Y\] \subseteq L \subseteq \[{Y}^{\ast}\]$, then $N(Z) \in \UUU$.
\end{proof}
\begin{Lemma} \label{lem:c0preserve}
 Suppose $\UUU$ is a selective ultrafilter on $\omega$ such that $\UUU \notin {\C}_{0}(\HHH)$ and $\UUU \; {\not\equiv}_{RK} \; {\HHH}_{\max}$.
 Suppose $\mathring{x}$ is a $\PP(\HHH)$-name such that ${\forces}_{\PP(\HHH)} \; {\mathring{x}: \omega \rightarrow 2}$.
 Then for every ${p}^{\ast} \in \PP(\HHH)$, there are $q, t, X,$ and $N$ such that:
 \begin{enumerate}
  \item
  $q \leq {p}^{\ast}$, $t \in 2$, $X \in {\FIN}^{\[\omega\]}$ with $\[X\] \in \HHH$, and $N \in \UUU$;
  \item
  letting $\seq{n}{i}{\in}{\omega}$ be the strictly increasing enumeration of $N$, for each $i \in \omega$, for each $e \in \lv{{T}_{q}}{\max\left(X(i)\right)+1}$, there exists ${e}^{\ast} \in \lv{{T}_{q}}{\max\left(X(i+1)\right)}$ so that $e \subseteq {e}^{\ast}$, there exists $S \subseteq {\succc}_{{T}_{q}}({e}^{\ast})$ such that
  \begin{align*}
   S \ \text{is} \ \pr{\max\left(X(i)\right)+1}{X(i+1)}\text{-big,}
  \end{align*}
  and for each $\sigma \in S$, ${T}_{q}\langle {{e}^{\ast}}^{\frown}{\langle \sigma \rangle} \rangle \; {\forces}_{\PP(\HHH)} \; {\mathring{x}({n}_{i+1}) = t}$.
 \end{enumerate}
\end{Lemma}
\begin{proof}
 For each $t \in 2$, let ${\DDD}_{t}$ be the collection of all $q \in \PP(\HHH)$ such that there are $A \in \HHH$, $Y \in {\FIN}^{\[\omega\]}$, $D \in \UUU$, $\seq{A}{l}{\in}{\omega}$, and $\seq{e}{k}{\in}{{f}_{\max}\[A\]}$ satisfying (1)--(4) of Lemma \ref{lem:c01} for $t$.
 By Lemma \ref{lem:c01}, $\DDD = {\DDD}_{0} \cup {\DDD}_{1}$ is dense in $\PP(\HHH)$.
 Hence by Lemma \ref{lem:densesplit}, there are $p \leq {p}^{\ast}$ and $t \in 2$ such that ${\DDD}_{t}$ is dense below $p$.
 Since $\UUU \notin {\C}_{0}(\HHH)$, but $\UUU$ is a selective ultrafilter, there exists ${Y}^{\ast} \in {\FIN}^{\[\omega\]}$ such that $\[{Y}^{\ast}\] \in \HHH$ and for all $Z \in {\FIN}^{\[\omega\]}$, if $\[Z\] \in \HHH$ and $\[Z\] \subseteq \[{Y}^{\ast}\]$, then $N(Z) \in \UUU$.
 Let $\theta$ be a sufficiently large regular cardinal and let $M \prec H(\theta)$ be countable with $M$ containing all the relevant parameters.
 In particular, $\HHH, \PP(\HHH), \UUU, \mathring{x}, {\DDD}_{t}, p, {Y}^{\ast} \in M$.
 As $\UUU$ is a P-point, find $D \in \UUU$ such that $D \; {\subseteq}^{\ast} \; C$, for all $C \in M \cap \UUU$.
 
 Now define a strategy $\Sigma$ for Player I in ${\Game}^{\mathtt{Cond}}\left(\HHH\right)$ as follows.
 The definition will insure that whenever $i \in \omega$ and $\left\langle \pr{\pr{{p}_{j}}{{A}_{j}}}{{s}_{j}}: j \leq i \right\rangle$ is a partial run of ${\Game}^{\mathtt{Cond}}\left(\HHH\right)$ in which Player I has followed $\Sigma$, then ${p}_{i} \in M$.
 ${p}_{0} = p \in \PP(\HHH) \cap M$.
 ${A}_{0} = \left\{s \in \[{Y}^{\ast}\]: I(s) \cap D \neq \emptyset \right\}$.
 By Corollary \ref{cor:hitsD}, ${A}_{0} \in \HHH$.
 Define $\Sigma(\emptyset) = \pr{{p}_{0}}{{A}_{0}}$.
 Suppose that $i \in \omega$ and that $\left\langle \pr{\pr{{p}_{j}}{{A}_{j}}}{{s}_{j}}: j \leq i \right\rangle$ is a partial run of ${\Game}^{\mathtt{Cond}}\left(\HHH\right)$ in which Player I has followed $\Sigma$.
 Assume that ${p}_{i} \in M$.
 Let $k = \max({s}_{i})+1$, for ease of notation.
 For any $e \in \lv{{T}_{{p}_{i}}}{k}$, ${T}_{{p}_{i}}\langle e \rangle \leq {p}_{i} \leq {p}_{0} = p$.
 Also ${T}_{{p}_{i}}\langle e \rangle \in M$.
 Since ${\DDD}_{t}$ is dense below $p$ and ${\DDD}_{t} \in M$, there exists ${p}_{i, e} \leq {T}_{{p}_{i}}\langle e \rangle$ such that ${p}_{i, e} \in M$ and there are ${A}_{e} \in \HHH \cap M$, ${Y}_{e} \in {\FIN}^{\[\omega\]} \cap M$, ${D}_{e} \in \UUU \cap M$, $\left\langle {A}_{e, k'}: k' \in \omega \right\rangle \in M$, and $\left\langle {h}_{e, l}: l \in {f}_{\max}\[{A}_{e}\] \right\rangle \in M$ satisfying (1)--(4) of Lemma \ref{lem:c01} for $t$.
 Now $C = \bigcap\left\{{D}_{e}: e \in \lv{{T}_{{p}_{i}}}{k}\right\} \in \UUU \cap M$ because it is the intersection of finitely many members of $\UUU \cap M$, and so there is ${l}^{\ast} \in \omega$ such that ${l}^{\ast} \geq k$ and $D \setminus {l}^{\ast} \subseteq C$.
 Let ${A}^{\ast} = \[{Y}^{\ast}\] \cap \bigcap\{{A}_{e, k}: e \in \lv{{T}_{{p}_{i}}}{k}\} \in \HHH$ and let $L = \{s \in {A}^{\ast}: \min(s) > {l}^{\ast}\} \in \HHH$.
 By Corollary \ref{cor:hitsD}, ${A}_{i+1} = \{s \in L: I(s) \cap D \neq \emptyset\} \in \HHH$.
 Define ${p}_{i+1} = {T}_{{p}_{i+1}} = \bigcup{\left\{ {T}_{{p}_{i, e}}: e \in \lv{{T}_{{p}_{i}}}{k} \right\}}$.
 Note that ${p}_{i+1} \in \PP(\HHH)$ by Lemma \ref{lem:amalgam}, and that ${p}_{i+1} \in M$ because it is the union of finitely many members of $M$.
 Define $\Sigma\left(\left\langle \pr{\pr{{p}_{j}}{{A}_{j}}}{{s}_{j}}: j \leq i \right\rangle\right) = \pr{{p}_{i+1}}{{A}_{i+1}}$.
 As ${p}_{i+1} \in M$, the inductive condition on $\Sigma$ is insured at $i+1$.
 Consider any $s \in {A}_{i+1}$.
 Observe that $s \in L$ and that $I(s) \cap D \neq \emptyset$.
 Hence $s \in {A}^{\ast}$ and $\min(s) > {l}^{\ast}$, which, in particular, means that $s \in \[{Y}^{\ast}\]$.
 Observe that for any $e \in \lv{{T}_{{p}_{i}}}{k}$, $s \in {A}_{e, k} \subseteq {A}_{e}$.
 Letting $l = \max(s) \in {f}_{\max}\[{A}_{e, k}\] \subseteq {f}_{\max}\[{A}_{e}\]$, ${h}_{e, l} \in \lv{{T}_{{p}_{i, e}}}{l}$, and since $\dom(e) = k \leq {l}^{\ast} < \min(s) \leq \max(s) = l = \dom({h}_{e, l})$, $e \subseteq {h}_{e, l}$.
 Further, ${\succc}_{{T}_{{p}_{i, e}}}({h}_{e, l})$ is $\pr{k}{s}$-big.
 If $n \in I(s) \cap D$, then $n \in D$ and ${l}^{\ast} < \min(s) \leq n$, whence $n \in C$.
 Thus $n \in {D}_{e} \cap I(s)$, and so ${T}_{{p}_{i, e}}\langle {h}_{e, l} \rangle \; {\forces}_{\PP(\HHH)} \; {\mathring{x}(n)=t}$.
 In particular, since for every $\sigma \in {\succc}_{{T}_{{p}_{i, e}}}({h}_{e, l})$, ${T}_{{p}_{i, e}}\langle {{h}_{e, l}}^{\frown}{\langle \sigma \rangle} \rangle \leq {T}_{{p}_{i, e}}\langle {h}_{e, l} \rangle$, ${T}_{{p}_{i, e}}\langle {{h}_{e, l}}^{\frown}{\langle \sigma \rangle} \rangle \; {\forces}_{\PP(\HHH)} \; {\mathring{x}(n)=t}$.
 Note that ${h}_{e, l} \in \lv{{T}_{{p}_{i+1}}}{l}$ and that ${\succc}_{{T}_{{p}_{i+1}}}({h}_{e, l}) = {\succc}_{{T}_{{p}_{i, e}}}({h}_{e, l})$.
 Therefore, ${\succc}_{{T}_{{p}_{i+1}}}({h}_{e, l})$ is $\pr{k}{s}$-big.
 Also since for every $\sigma \in {\succc}_{{T}_{{p}_{i+1}}}({h}_{e, l})$, ${T}_{{p}_{i+1}}\langle {{h}_{e, l}}^{\frown}{\langle \sigma \rangle} \rangle \leq {T}_{{p}_{i, e}}\langle {{h}_{e, l}}^{\frown}{\langle \sigma \rangle} \rangle$,
 ${T}_{{p}_{i+1}}\langle {{h}_{e, l}}^{\frown}{\langle \sigma \rangle} \rangle \; {\forces}_{\PP(\HHH)} \; {\mathring{x}(n)=t}$, for every $n \in I(s) \cap D$.
 This completes the definition of $\Sigma$.
 
 Since $\Sigma$ is not a winning strategy for Player I, there is a run $\left\langle \pr{\pr{{p}_{i}}{{A}_{i}}}{{s}_{i}}: i < \omega \right\rangle$ of ${\Game}^{\mathtt{Cond}}\left(\HHH\right)$ in which Player I has followed $\Sigma$ and lost.
 Then $X = \{{s}_{i}: i < \omega\} \in {\FIN}^{\[\omega\]}$ and $\[X\] \in \HHH$.
 Furthermore, $X(i) = {s}_{i}$, for all $i \in \omega$.
 For each $i \in \omega$, $X(i) = {s}_{i} \in {A}_{i} \subseteq \[{Y}^{\ast}\]$.
 Since $X, {Y}^{\ast} \in {\FIN}^{\[\omega\]}$, it follows that $\[X\] \subseteq \[{Y}^{\ast}\]$, and so $N(X) \in \UUU$ by the choice of ${Y}^{\ast}$.
 Furthermore, for each $i \in \omega$, $I\left(X(i)\right) \cap D \neq \emptyset$.
 Let $R = N(X) \cap D \in \UUU$.
 Note that for any $i \in \omega$, if $n \in I\left(X(i)\right) \cap D$, then $n \in N(X) \cap D = R$, and so $n \in I\left(X(i)\right) \cap R$, meaning that $I\left(X(i)\right) \cap R \neq \emptyset$.
 Since $R \subseteq N(X)$, $\UUU$ is a Q-point, and $\forall i < j < \omega\[I\left(X(i)\right) \cap I\left(X(j)\right) = \emptyset\]$, there exists $N \in \UUU$ such that $N \subseteq R$, and $\forall i \in \omega\[\lc I\left(X(i)\right) \cap N \rc = 1\]$.
 Let $\seq{n}{i}{\in}{\omega}$ be the strictly increasing enumeration of $N$.
 As $N \subseteq R \subseteq N(X)$, it is clear that $\{{n}_{i}\} = I\left(X(i)\right) \cap N$, for every $i \in \omega$.
 In particular, ${n}_{i} \in I\left(X(i)\right) \cap D$, for all $i \in \omega$.
 Again, for ease of notation, denote $\max({s}_{i})+1$ by ${k}_{i}$.
 Let $q = {T}_{q} = {\bigcap}_{i \in \omega}{{T}_{{p}_{i}}}$.
 Then $q \in \PP(\HHH)$ and $q \leq {p}_{0} = p \leq {p}^{\ast}$.
 Consider $i \in \omega$ and $e \in \lv{{T}_{q}}{{k}_{i}}$.
 Then $e \in \lv{{T}_{{p}_{i}}}{{k}_{i}}$.
 Since ${s}_{i+1} \in {A}_{i+1}$, then as noted above, letting $l = \max({s}_{i+1}) = \max\left(X(i+1)\right)$, there exists ${e}^{\ast} \in \lv{{T}_{{p}_{i+1}}}{l}$ such that $e \subseteq {e}^{\ast}$, there exists ${\succc}_{{T}_{{p}_{i+1}}}({e}^{\ast}) = S \subseteq {\succc}_{{T}_{{p}_{i+1}}}({e}^{\ast})$ such that $S$ is $\pr{{k}_{i}}{{s}_{i+1}}$-big, and for every $\sigma \in S$, ${T}_{{p}_{i+1}}\langle {{e}^{\ast}}^{\frown}{\langle \sigma \rangle} \rangle \; {\forces}_{\PP(\HHH)} \; {\mathring{x}({n}_{i+1}) = t}$ because ${n}_{i+1} \in I\left(X(i+1)\right) \cap D = I({s}_{i+1}) \cap D$.
 As ${e}^{\ast} \in {T}_{{p}_{i+1}}$ and $\dom({e}^{\ast}) = l \leq \max({s}_{i+1})+1$, ${e}^{\ast} \in \lv{{T}_{q}}{\max\left(X(i+1)\right)}$ by Lemma \ref{lem:winning1}.
 Similarly, if $\sigma \in {\succc}_{{T}_{{p}_{i+1}}}({e}^{\ast})$, then ${{e}^{\ast}}^{\frown}{\langle \sigma \rangle} \in {T}_{{p}_{i+1}}$ and $\dom\left({{e}^{\ast}}^{\frown}{\langle \sigma \rangle}\right) = l+1 = \max({s}_{i+1})+1$, ${{e}^{\ast}}^{\frown}{\langle \sigma \rangle} \in {T}_{q}$, whence $\sigma \in {\succc}_{{T}_{q}}({e}^{\ast})$.
 Therefore, $S \subseteq {\succc}_{{T}_{{p}_{i+1}}}({e}^{\ast}) \subseteq {\succc}_{{T}_{q}}({e}^{\ast})$ and $S$ is $\pr{\max\left(X(i)\right)+1}{X(i+1)}$-big.
 Finally, for any $\sigma \in S$, ${T}_{q}\langle {{e}^{\ast}}^{\frown}{\langle \sigma \rangle} \rangle \leq {T}_{{p}_{i+1}}\langle {{e}^{\ast}}^{\frown}{\langle \sigma \rangle} \rangle$, and so ${T}_{q}\langle {{e}^{\ast}}^{\frown}{\langle \sigma \rangle} \rangle \; {\forces}_{\PP(\HHH)} \; {\mathring{x}({n}_{i+1}) = t}$.
 This is as required.
\end{proof}
\begin{Cor} \label{cor:notc0-preserve}
 Suppose $\VVV$ is a selective ultrafilter on $\omega$ with $\VVV \notin {\C}_{1}(\HHH)$.
 Then $\PP(\HHH)$ preserves $\VVV$.
\end{Cor}
\begin{proof}
 It will first be argued that if $\UUU$ is a selective ultrafilter on $\omega$ with $\UUU \notin {\C}_{0}(\HHH)$, then $\PP(\HHH)$ preserves $\UUU$.
 Let such a $\UUU$ be given.
 If $\UUU \; {\equiv}_{RK} \; {\HHH}_{\max}$, then $\PP(\HHH)$ preserves $\UUU$ by Corollary \ref{cor:H_max}.
 So assume that $\UUU \; {\not\equiv}_{RK} \; {\HHH}_{\max}$.
 Suppose $\mathring{x}$ is a $\PP(\HHH)$-name such that ${\forces}_{\PP(\HHH)} \; {\mathring{x}: \omega \rightarrow 2}$ and that $p \in \PP(\HHH)$.
 By Lemma \ref{lem:c0preserve}, there are $q, t, X$, and $N$ such that $q \leq p$, $t \in 2$, $N \in \UUU$, and the hypotheses of Lemma \ref{lem:bigdecision} are satisfied.
 Let $\seq{n}{i}{\in}{\omega}$ be the strictly increasing enumeration of $N$.
 By Lemma \ref{lem:bigdecision}, there exists $q' \leq q \leq p$ such that $\forall 1 \leq i < \omega\[q' \; {\forces}_{\PP(\HHH)} \; {\mathring{x}({n}_{i}) = t} \]$.
 Since $N \in \UUU$, $A = \{{n}_{i}: 1 \leq i < \omega\} \in \UUU$, and so (1) of Lemma \ref{lem:preserving} says that $\PP(\HHH)$ preserves $\UUU$.
 Thus it has been proved that $\PP(\HHH)$ preserves $\UUU$ whenever $\UUU$ is a selective ultrafilter on $\omega$ with $\UUU \notin {\C}_{0}(\HHH)$.
 
 Now suppose $\VVV$ is a selective ultrafilter on $\omega$ with $\VVV \notin {\C}_{1}(\HHH)$.
 Then there exists $\UUU$ such that $\UUU$ is a selective ultrafilter on $\omega$, $\UUU \notin {\C}_{0}(\HHH)$, and $\VVV \; {\equiv}_{RK} \; \UUU$.
 By the previous paragraph, $\PP(\HHH)$ preserves $\UUU$.
 Since $\VVV \; {\leq}_{RK} \; \UUU$ and $\VVV$ is an ultrafilter on $\omega$, $\PP(\HHH)$ preserves $\VVV$ by (2) of Lemma \ref{lem:preserving}.
\end{proof}
\begin{Lemma} \label{lem:c1Hpreserved}
 Suppose $\UUU$ is a selective ultrafilter on $\omega$.
 Suppose $\UUU \in {\C}_{1}(\HHH)$.
 Then $\PP(\HHH)$ preserves $\UUU$.
\end{Lemma}
\begin{proof}
 Let $\mathring{x}$ be a $\PP(\HHH)$-name such that ${\forces}_{\PP(\HHH)} \; {\mathring{x}: \omega \rightarrow 2}$ and let $p \in \PP(\HHH)$.
 
 For each $q \leq p$ let ${E}_{q} = \{n \in \omega: \exists r \leq q \[r \; {\forces}_{\PP(\HHH)} \; {\mathring{x}(n) = 1}\]\}$.
 Suppose first that there exists $q \leq p$ such that ${E}_{q} \notin \UUU$.
 Then $A = \omega \setminus {E}_{q} \in \UUU$, and for all $n \in A$, $q \; {\forces}_{\PP(\HHH)} \; {\mathring{x}(n) = 0}$.
 This is as required by (1) of Lemma \ref{lem:preserving}.
 
 It will henceforth be assumed that for all $q \leq p$, ${E}_{q} \in \UUU$.
 Define $\Sigma$ and $\Phi$ such that:
 \begin{enumerate}
  \item
  $\Sigma$ is a strategy for Player I in ${\Game}^{\mathtt{SelStab}}\left(\UUU, \HHH\right)$;
  \item
  for each $n \in \omega$, if $\left\langle \pr{{C}_{i}}{{o}_{i}}: i \leq 2n \right\rangle$ is a partial run of ${\Game}^{\mathtt{SelStab}}\left(\UUU, \HHH\right)$ in which Player I has followed $\Sigma$, then $\Phi\left(\left\langle \pr{{C}_{i}}{{o}_{i}}: i \leq 2n \right\rangle\right) = \seq{p}{i}{\leq}{n}$ and
  \begin{enumerate}
   \item[(a)]
   $\forall i \leq n\[{p}_{i} \in \PP(\HHH)\]$, $\forall i < n\[{p}_{i+1} \leq {p}_{i}\]$, ${p}_{0} \leq p$;
   \item[(b)]
   for each $i < j \leq n$ and for each $e \in \lv{{T}_{{p}_{j}}}{\max\left({o}_{2i+1}\right)}$, ${\succc}_{{T}_{{p}_{i}}}(e) \subseteq {\succc}_{{T}_{{p}_{j}}}(e)$;
   \item[(c)]
   for each $i \leq n$, ${p}_{i} \; {\forces}_{\PP(\HHH)} \; {\mathring{x}({o}_{2i}) = 1}$;
  \end{enumerate}
  \item
  for each $n \in \omega$, if $\left\langle \pr{{C}_{i}}{{o}_{i}}: i \leq 2n+3 \right\rangle$ is a partial run of ${\Game}^{\mathtt{SelStab}}\left(\UUU, \HHH\right)$ in which Player I has followed $\Sigma$ and if
  \begin{align*}
   \Phi\left(\left\langle \pr{{C}_{i}}{{o}_{i}}: i \leq 2n+2 \right\rangle\right) = \seq{p}{i}{\leq}{n+1},
  \end{align*}
  then ${o}_{2n+3} \in {H}_{{p}_{n+1}, \max\left({o}_{2n+1}\right)+1}$;
  \item
  for each $n \in \omega$, if $\left\langle \pr{{C}_{i}}{{o}_{i}}: i \leq 2n+2 \right\rangle$ is a partial run of ${\Game}^{\mathtt{SelStab}}\left(\UUU, \HHH\right)$ in which Player I has followed $\Sigma$, then
  \begin{align*}
   \Phi\left(\left\langle \pr{{C}_{i}}{{o}_{i}}: i \leq 2n \right\rangle\right) = \Phi\left(\left\langle \pr{{C}_{i}}{{o}_{i}}: i \leq 2n+2 \right\rangle\right) \restrict n+1.
  \end{align*}
 \end{enumerate}
 Suppose for a moment that $\Sigma$ and $\Phi$ satisfying (1)--(4) can be defined.
 Since $\Sigma$ is not a winning strategy for Player I, there is a play $\left\langle \pr{{C}_{i}}{{o}_{i}}: i < \omega \right\rangle$ of ${\Game}^{\mathtt{SelStab}}\left(\UUU, \HHH\right)$ in which Player I follows $\Sigma$ and loses.
 Then $X = \{{o}_{2i+1}: i < \omega\} \in {\FIN}^{\[\omega\]}$ and $\[X\] \in \HHH$.
 Moreover, $X(i) = {o}_{2i+1}$.
 Also, there is a sequence $\seq{p}{i}{<}{\omega}$ such that for each $n \in \omega$, $\Phi\left(\left\langle \pr{{C}_{i}}{{o}_{i}}: i \leq 2n \right\rangle\right) = \seq{p}{i}{\leq}{n}$.
 Then $X$ and $\seq{p}{i}{<}{\omega}$ satisfy the hypotheses of Lemma \ref{lem:fusion}.
 To see this, first note that applying (3) with $n=i$ gives $X(i+1)={o}_{2n+3} \in {H}_{{p}_{n+1}, \max\left({o}_{2n+1}\right)+1} = {H}_{{p}_{i+1}, \max\left(X(i)\right)+1}$.
 Next, given $i \leq j < \omega$ and $e \in \lv{{T}_{{p}_{j}}}{\max\left(X(i)\right)} = \lv{{T}_{{p}_{j}}}{\max\left({o}_{2i+1}\right)}$, if $i=j$, then trivially ${\succc}_{{T}_{{p}_{i}}}(e) \subseteq {\succc}_{{T}_{{p}_{j}}}(e)$.
 If $i < j$, then applying (2)(b) with $n=j$ gives ${\succc}_{{T}_{{p}_{i}}}(e) \subseteq {\succc}_{{T}_{{p}_{j}}}(e)$.
 Therefore Lemma \ref{lem:fusion} applies and implies that $q = {T}_{q} = {\bigcap}_{i \in \omega}{{T}_{{p}_{i}}} \in \PP(\HHH)$.
 Further, $q \leq {p}_{0} \leq p$, and for each $i \in \omega$, since $q \leq {p}_{i}$, $q \; {\forces}_{\PP(\HHH)} \; {\mathring{x}\left({o}_{2i}\right) = 1}$.
 Since $A = \{{o}_{2i}: i < \omega\} \in \UUU$, $q$ and $A$ satisfy the conditions of (1) of Lemma \ref{lem:preserving}.
 Thus in both this case and in the case considered in the previous paragraph, (1) of Lemma \ref{lem:preserving} applies and it implies that $\PP(\HHH)$ preserves $\UUU$.
 
 $\Sigma$ and $\Phi$ are defined inductively.
 $\Sigma\left(\emptyset\right) = {E}_{p} \in \UUU$.
 If $\left\langle \pr{{C}_{0}}{{o}_{0}} \right\rangle$ is a partial run of ${\Game}^{\mathtt{SelStab}}\left(\UUU, \HHH\right)$ in which Player I has followed $\Sigma$, then ${C}_{0} = {E}_{p}$ and ${o}_{0} \in {C}_{0} = {E}_{p}$.
 Define $\Phi\left(\left\langle \pr{{C}_{0}}{{o}_{0}} \right\rangle\right) = \left\langle {p}_{0} \right\rangle$, where ${p}_{0} \leq p$ and ${p}_{0} \; {\forces}_{\PP(\HHH)} \; {\mathring{x}({o}_{0}) = 1}$.
 Such a ${p}_{0}$ exists by the definition of ${E}_{p}$.
 Define $\Sigma\left(\left\langle \pr{{C}_{0}}{{o}_{0}} \right\rangle\right) = \FIN \in \HHH$.
 It is clear that (1)--(4) are satisfied by these definitions.
 Now, assume that for some $n \in \omega$, $\left\langle \pr{{C}_{i}}{{o}_{i}}: i \leq 2n+1 \right\rangle$ is a partial run of ${\Game}^{\mathtt{SelStab}}\left(\UUU, \HHH\right)$ in which Player I has followed $\Sigma$ and that $\Phi\left(\left\langle \pr{{C}_{i}}{{o}_{i}}: i \leq 2n \right\rangle\right) = \seq{p}{i}{\leq}{n}$, and that these satisfy (1)--(4).
 Let $k = \max\left({o}_{2n+1}\right)+1$.
 For any $e \in \lv{{T}_{{p}_{n}}}{k}$, ${T}_{{p}_{n}}\langle e \rangle \leq {p}_{n} \leq {p}_{0} \leq p$.
 Therefore ${E}_{{T}_{{p}_{n}}\langle e \rangle} \in \UUU$.
 Define $\Sigma\left( \left\langle \pr{{C}_{i}}{{o}_{i}}: i \leq 2n+1 \right\rangle \right) = {C}_{2n+2} = \bigcap\left\{ {E}_{{T}_{{p}_{n}}\langle e \rangle}: e \in \lv{{T}_{{p}_{n}}}{k} \right\} \in \UUU$.
 If $\left\langle \pr{{C}_{i}}{{o}_{i}}: i \leq 2n+2 \right\rangle$ is a partial continuation of ${\Game}^{\mathtt{SelStab}}\left(\UUU, \HHH\right)$ in which Player I has followed $\Sigma$, then ${o}_{2n+2} \in {C}_{2n+2}$ and so there is a sequence $\left\langle {p}_{n, e}: e \in \lv{{T}_{{p}_{n}}}{k} \right\rangle$ such that for each $e \in \lv{{T}_{{p}_{n}}}{k}$, ${p}_{n, e} \leq {T}_{{p}_{n}}\langle e \rangle$ and ${p}_{n, e} \; {\forces}_{\PP(\HHH)} \; {\mathring{x}\left({o}_{2n+2}\right) = 1}$.
 Define ${p}_{n+1} = {T}_{{p}_{n+1}} = \bigcup\left\{{T}_{{p}_{n, e}}: e \in \lv{{T}_{{p}_{n}}}{k} \right\}$.
 By Lemma \ref{lem:amalgam}, ${p}_{n+1} \in \PP(\HHH)$ and ${p}_{n+1} \leq {p}_{n}$.
 Define $\Phi\left( \left\langle \pr{{C}_{i}}{{o}_{i}}: i \leq 2n+2 \right\rangle \right) = \seq{p}{i}{\leq}{n+1}$.
 Finally, define $\Sigma\left( \left\langle \pr{{C}_{i}}{{o}_{i}}: i \leq 2n+2 \right\rangle \right) = {C}_{2n+3} = {H}_{{p}_{n+1}, k} \in \HHH$, so that if ${o}_{2n+3} \in {C}_{2n+3}$, then ${o}_{2n+3} \in {H}_{{p}_{n+1}, \max\left({o}_{2n+1}\right)+1}$.
 It is clear that (1), (3), and (4) are satisfied.
 It is also clear that (2)(a) holds.
 For 2(b), suppose $i < j \leq n+1$.
 If $i < j \leq n$, then the induction hypothesis gives what is needed.
 So assume that $i < j=n+1$.
 Suppose that $e \in \lv{{T}_{{p}_{j}}}{\max\left({o}_{2i+1}\right)}$.
 Then $e \in \lv{{T}_{{p}_{n}}}{\max\left({o}_{2i+1}\right)}$.
 If $i=n$, then trivially ${\succc}_{{T}_{{p}_{i}}}(e) \subseteq {\succc}_{{T}_{{p}_{n}}}(e)$, while if $i < n$, then by the induction hypothesis, ${\succc}_{{T}_{{p}_{i}}}(e) \subseteq {\succc}_{{T}_{{p}_{n}}}(e)$.
 Thus in either case, ${\succc}_{{T}_{{p}_{i}}}(e) \subseteq {\succc}_{{T}_{{p}_{n}}}(e)$.
 Suppose $\sigma \in {\succc}_{{T}_{{p}_{n}}}(e)$.
 Then ${e}^{\frown}{\langle \sigma \rangle} \in {T}_{{p}_{n}}$, and $\dom({e}^{\frown}{\langle \sigma \rangle}) = \max\left({o}_{2i+1}\right)+1 \leq \max\left({o}_{2n+1}\right)+1 = k$.
 Choose some ${e}^{\ast} \in \lv{{T}_{{p}_{n}}}{k}$ with ${e}^{\frown}{\langle \sigma \rangle} \subseteq {e}^{\ast}$.
 Then ${e}^{\frown}{\langle \sigma \rangle} \in {T}_{{p}_{n, {e}^{\ast}}} \subseteq {T}_{{p}_{n+1}}$, whence $\sigma \in {\succc}_{{T}_{{p}_{n+1}}}(e)$.
 Therefore, ${\succc}_{{T}_{{p}_{i}}}(e) \subseteq {\succc}_{{T}_{{p}_{n}}}(e) \subseteq {\succc}_{{T}_{{p}_{n+1}}}(e) = {\succc}_{{T}_{{p}_{j}}}(e)$, as required for (2)(b).
 For (2)(c), suppose $r \leq {p}_{n+1}$.
 Choose $e \in \lv{{T}_{r}}{k}$.
 Then $e \in \lv{{T}_{{p}_{n}}}{k}$ and ${T}_{r}\langle e \rangle \leq {T}_{{p}_{n+1}}\langle e \rangle \leq {p}_{n, e}$.
 So ${T}_{r}\langle e \rangle \; {\forces}_{\PP(\HHH)} \; {\mathring{x}\left({o}_{2n+2}\right) = 1}$.
 Therefore, ${p}_{n+1} \; {\forces}_{\PP(\HHH)} \; {\mathring{x}\left({o}_{2n+2}\right) = 1}$, as needed for (2)(c).
 This concludes the definition of $\Sigma$ and $\Phi$ and the proof.
\end{proof}
\begin{Theorem} \label{thm:preservesselective}
 $\PP(\HHH)$ preserves all selective ultrafilters on $\omega$.
\end{Theorem}
\begin{proof}
 Let $\UUU$ be a selective ultrafilter on $\omega$.
 Either $\UUU \in {\C}_{1}(\HHH)$ or $\UUU \notin {\C}_{1}(\HHH)$.
 If $\UUU \notin {\C}_{1}(\HHH)$, then $\PP(\HHH)$ preserves $\UUU$ by Corollary \ref{cor:notc0-preserve}.
 If $\UUU \in {\C}_{1}(\HHH)$, then $\PP(\HHH)$ preserves $\UUU$ by Lemma \ref{lem:c1Hpreserved}.
\end{proof}
\section{A model with many selectives and no stable ordered-unions} \label{sec:manyselectives}
The main result of the paper is proved in this section.
It is shown that there is a model of set theory where ${2}^{{\aleph}_{0}} = {\aleph}_{2}$, there are ${2}^{{\aleph}_{0}}$ pairwise non-RK-isomorphic selective ultrafilters, but no stable ordered-union ultrafilters.
After all the work done in the previous sections, the proof is mostly a matter of combining the earlier lemmas with certain well-known iteration theorems.
Nevertheless, we provide a good amount of detail.
The first lemma is well-known.
A proof is included for completeness.
\begin{Lemma} \label{lem:selectivestay}
 Suppose $\UUU$ is a selective ultrafilter on $\omega$.
 Suppose $\PP$ is proper and $\BS$-bounding.
 If $\PP$ preserves $\UUU$, then
 \begin{align*}
  {\forces}_{\PP} \; {``\{A \subseteq \omega: \exists B \in \UUU\[B \subseteq A\]\} \ \text{is a selective ultrafilter on} \ \omega''}.
 \end{align*}
\end{Lemma}
\begin{proof}
 Let $G$ be $(\V, \PP)$-generic.
 In $\VG$, let ${\UUU}_{0} = \{A \subseteq \omega: \exists B \in \UUU\[B \subseteq A\]\}$.
 By hypothesis, ${\UUU}_{0}$ is an ultrafilter on $\omega$.
 Suppose $\{{A}_{n}: n \in \omega\} \subseteq {\UUU}_{0}$ is given and choose $\{{B}_{n}: n \in \omega\} \subseteq \UUU$ so that ${B}_{n} \subseteq {A}_{n}$.
 Since $\PP$ is proper, there is a set $X \in \V$ which is countable in $\V$ and satisfies $\{{B}_{n}: n \in \omega\} \subseteq X \subseteq \UUU$.
 As $\UUU$ is a P-point in $\V$, there is $B \in \UUU$ such that $B \; {\subseteq}^{\ast} \; A$, for every $A \in X$.
 Now $B \in {\UUU}_{0}$ and $B \; {\subseteq}^{\ast} \; {A}_{n}$, for every $n \in \omega$.
 So ${\UUU}_{0}$ is a P-point in $\VG$.
 
 Next, let $I = \seq{i}{n}{\in}{\omega}$ be an interval partition in $\VG$.
 Since $\PP$ is $\BS$-bounding, there is an interval partition $J = \seq{j}{n}{\in}{\omega}$ in $\V$ such that $\forall n \in \omega \exists l \in \omega\[{I}_{l} \subseteq {J}_{n}\]$.
 As $\UUU$ is a Q-point in $\V$, there is $B \in \UUU$ so that $\forall n \in \omega\[\lc B \cap {J}_{n} \rc \leq 1\]$.
 Now in $\VG$, $B \in {\UUU}_{0}$ and it is easily seen that $\forall n \in \omega\[\lc B \cap {I}_{n} \rc \leq 2\]$.
 Define $C = \left\{ \min\left( B \cap {I}_{n} \right): n \in \omega \wedge B \cap {I}_{n} \neq \emptyset \right\}$ and $D = B \setminus C$.
 It is clear that $\forall n \in \omega\[\lc C \cap {I}_{n} \rc \leq 1\]$ and that $\forall n \in \omega\[\lc D \cap {I}_{n} \rc \leq 1\]$.
 $C \in {\UUU}_{0}$ or $D \in {\UUU}_{0}$ because ${\UUU}_{0}$ is an ultrafilter in $\VG$.
 Therefore, ${\UUU}_{0}$ is a Q-point in $\VG$.
 Finally, it is well-known (see \cite{BJ}) that an ultrafilter on $\omega$ is selective if and only if it is both a P-point and a Q-point.
\end{proof}
\begin{Lemma} \label{lem:hausdorff}
 Let $\UUU$ and $\VVV$ be selective ultrafilters on $\omega$ so that $\UUU \; {\not\equiv}_{RK} \; \VVV$.
 For any $f: \omega \rightarrow \omega$, for any $\psi: \omega \rightarrow \omega$, and for any ${A}_{0} \in \UUU$, there exists a sequence $\seq{n}{i}{<}{\omega}$ such that:
 \begin{enumerate}
  \item
  ${n}_{i} \in \omega$, ${n}_{i} < {n}_{i+1}$;
  \item
  $\{{n}_{2j}: j \in \omega\} \in \UUU$, $\{{n}_{2j}: j \in \omega\} \subseteq {A}_{0}$, and $\{{n}_{2j+1}: j \in \omega\} \in \VVV$;
  \item
  $f({n}_{2i}) < {n}_{2i+1}$ and $\psi({n}_{2i+1}) < {n}_{2i+2}$.
 \end{enumerate}
\end{Lemma}
\begin{proof}
 Define a strategy $\Sigma$ for Player I in ${\Game}^{\mathtt{SelSel}}\left(\UUU, \VVV\right)$ as follows.
 $\Sigma(\emptyset) = {A}_{0} \in \UUU$.
 Suppose $i \in \omega$ and that $\left\langle \pr{{C}_{j}}{{n}_{j}}: j \leq 2i \right\rangle$ is a partial run of ${\Game}^{\mathtt{SelSel}}\left(\UUU, \VVV\right)$ in which Player I has followed $\Sigma$.
 Then define
 \begin{align*}
  \Sigma\left(\left\langle \pr{{C}_{j}}{{n}_{j}}: j \leq 2i \right\rangle\right) = \{n \in \omega: n > \max\{{n}_{2i}, f({n}_{2i})\}\} \in \VVV.
 \end{align*}
 Next, suppose $\left\langle \pr{{C}_{j}}{{n}_{j}}: j \leq 2i+1 \right\rangle$ is a partial run of ${\Game}^{\mathtt{SelSel}}\left(\UUU, \VVV\right)$ in which Player I has followed $\Sigma$.
 Define
 \begin{align*}
  \Sigma\left(\left\langle \pr{{C}_{j}}{{n}_{j}}: j \leq 2i+1 \right\rangle\right) = \{n \in {A}_{0}: n > \max\{{n}_{2i+1}, \psi({n}_{2i+1})\}\} \in \UUU.
 \end{align*}
 This concludes the definition of $\Sigma$.
 Since it is not a winning strategy for Player I (by Lemma \ref{lem:selsel}), there is a run $\left\langle \pr{{C}_{i}}{{n}_{i}}: i < \omega \right\rangle$ of ${\Game}^{\mathtt{SelSel}}\left(\UUU, \VVV\right)$ in which Player I has followed $\Sigma$ and lost.
 Now (1)--(3) are satisfied because Player II won and because of the way $\Sigma$ is defined.
\end{proof}
\begin{Cor} \label{cor:non-RK-preservation}
 Let $\UUU$ and $\VVV$ be selective ultrafilters on $\omega$ so that $\UUU \; {\not\equiv}_{RK} \; \VVV$.
 Suppose $\PP$ is proper and $\BS$-bounding.
 If $\PP$ preserves $\UUU$ and $\VVV$, then
 \begin{align*}
  {\forces}_{\PP} \; {``\{A \subseteq \omega: \exists B \in \UUU\[B \subseteq A\]\} \; {\not\equiv}_{RK} \; \{A \subseteq \omega: \exists B \in \VVV\[B \subseteq A\]\}.''}
 \end{align*}
\end{Cor}
\begin{proof}
 Let $G$ be any $(\V, \PP)$-generic filter.
 Work in $\VG$.
 Let ${\UUU}_{0} = \{A \subseteq \omega: \exists B \in \UUU\[B \subseteq A\]\}$ and ${\VVV}_{0} = \{A \subseteq \omega: \exists B \in \VVV\[B \subseteq A\]\}$.
 They are both selective ultrafilters on $\omega$.
 Suppose for a contradiction that $g: \omega \rightarrow \omega$ witnesses ${\VVV}_{0} \; {\leq}_{RK} \; {\UUU}_{0}$.
 Then for any $A \in {\UUU}_{0}$, $g\[A\] \in {\VVV}_{0}$.
 As ${\UUU}_{0}$ is a P-point, there is ${A}_{1} \in {\UUU}_{0}$ such that $g$ is either finite-to-one or constant on ${A}_{1}$.
 $g$ cannot be constant on ${A}_{1}$ because $g\[{A}_{1}\] \in {\VVV}_{0}$.
 Therefore, there is a function $\varphi: \omega \rightarrow \omega$ such that for any $k \in \omega$, for any $n \in {A}_{1}$, if $n > \varphi(k)$, then $g(n) > k$.
 Let ${A}_{0} \in \UUU$ with ${A}_{0} \subseteq {A}_{1}$.
 As $\PP$ is $\BS$-bounding, find $f$ and $\psi$ in $\V$ such that $f: \omega \rightarrow \omega$, $\psi: \omega \rightarrow \omega$, and $g(n) < f(n)$ and $\varphi(n) < \psi(n)$, for every $n \in \omega$.
 Applying Lemma \ref{lem:hausdorff} back in $\V$, find a sequence $\seq{n}{i}{<}{\omega}$ satisfying (1)--(3) of Lemma \ref{lem:hausdorff}.
 Let $A = \{{n}_{2i}: i \in \omega\} \in {\UUU}_{0}$ and $B = \{{n}_{2j+1}: j \in \omega\} \in {\VVV}_{0}$.
 Note $A \subseteq {A}_{0} \subseteq {A}_{1}$.
 It will be verified that $g\[A\] \cap B = \emptyset$.
 To this end, fix $i, j \in \omega$.
 If $i \leq j$, then $g({n}_{2i}) < f({n}_{2i}) < {n}_{2i+1} \leq {n}_{2j+1}$.
 If $j < i$, then $\varphi({n}_{2j+1}) < \psi({n}_{2j+1}) < {n}_{2j+2} \leq {n}_{2i}$, whence $g({n}_{2i}) > {n}_{2j+1}$.
 This proves $g\[A\] \cap B = \emptyset$, contradicting $g\[A\] \in {\VVV}_{0}$, and concluding the proof.
\end{proof}
\begin{Lemma}[Blass and Shelah~\cite{simple}] \label{lem:P-pointiteration}
 Suppose $\UUU$ is a P-point.
 Let $\gamma$ be a limit ordinal and let $\langle {\PP}_{\alpha}; {\mathring{\QQ}}_{\alpha}: \alpha \leq \gamma \rangle$ be a CS iteration such that $\forall \alpha < \gamma\[{\forces}_{\alpha}\;{{\mathring{\QQ}}_{\alpha} \ \text{is proper}}\]$.
 Suppose that for all $\alpha < \gamma$, ${\PP}_{\alpha}$ preserves $\UUU$.
 Then ${\PP}_{\gamma}$ preserves $\UUU$.
\end{Lemma}
\begin{Theorem} \label{thm:aleph2selectives}
 There is a model of $\ZFC$ with ${\aleph}_{2}$ pairwise non-RK-isomorphic selective ultrafilters on $\omega$ and no stable ordered-union ultrafilters on $\FIN$.
\end{Theorem}
\begin{proof}
 Put ${S}^{2}_{1} = \{\alpha < {\omega}_{2}: \cf(\alpha) = {\omega}_{1}\}$.
 Let $\V$ be a model satisfying $\CH$ and $\diamondsuit\left({S}^{2}_{1}\right)$.
 By $\CH$, fix a family $\{{\UUU}_{\alpha}: \alpha < {\omega}_{2}\}$ of pairwise non-RK-isomorphic selective ultrafilters on $\omega$ (for instance, see \cite{walterppoints} and \cite{Rudin:66}).
 Fixing some diamond sequence witnessing $\diamondsuit\left({S}^{2}_{1}\right)$, define a CS iteration $\left\langle {\PP}_{\alpha}; {\mathring{\QQ}}_{\alpha}: \alpha \leq {\omega}_{2} \right\rangle$ in $\V$ as follows.
 Assume $\alpha < {\omega}_{2}$ and that ${\PP}_{\alpha}$ is proper, $\BS$-bounding, satisfies the ${\aleph}_{2}$-c.c.\@, and that ${\forces}_{\alpha} \; {\CH}$.
 Observe that ${\forces}_{\alpha} \; {\text{``stable ordered-union ultrafilters exist''}}$ because ${\forces}_{\alpha} \; {\CH}$.
 If the diamond sequence at $\alpha$ codes a pair $\pr{\mathring{\GGG}}{p}$ such that $p \in {\PP}_{\alpha}$, $\mathring{\GGG}$ is a ${\PP}_{\alpha}$-name, and $p \; {\forces}_{\alpha} \; {``\mathring{\GGG} \ \text{is a stable ordered-union ultrafilter on} \ \FIN''}$, then choose a ${\PP}_{\alpha}$-name ${\mathring{\HHH}}_{\alpha}$ such that ${\forces}_{\alpha} \; {``{\mathring{\HHH}}_{\alpha} \ \text{is a stable ordered-union ultrafilter on} \ \FIN''}$ and $p \; {\forces}_{\alpha} \; {{\mathring{\HHH}}_{\alpha} = \mathring{\GGG}}$, and define ${\mathring{\QQ}}_{\alpha}$ to be a full ${\PP}_{\alpha}$-name so that ${\forces}_{\alpha} \; {{\mathring{\QQ}}_{\alpha} = \PP({\mathring{\HHH}}_{\alpha})}$.
 Otherwise choose an arbitrary ${\PP}_{\alpha}$-name ${\mathring{\HHH}}_{\alpha}$ with ${\forces}_{\alpha} \; {``{\mathring{\HHH}}_{\alpha} \ \text{is a stable ordered-union ultrafilter on} \ \FIN''}$, and define ${\mathring{\QQ}}_{\alpha}$ to be a full ${\PP}_{\alpha}$-name so that ${\forces}_{\alpha} \; {{\mathring{\QQ}}_{\alpha} = \PP({\mathring{\HHH}}_{\alpha})}$.
 Note that in both cases ${\forces}_{\alpha} \; {\lc {\mathring{\QQ}}_{\alpha} \rc = {\aleph}_{1}}$ because ${\forces}_{\alpha} \; {\CH}$.
 Standard arguments in the theory of proper forcing (see Shelah~\cite{PIF} or Abraham~\cite{Ab}) together with lemmas proved earlier therefore imply that for each $\delta \leq {\omega}_{2}$, ${\PP}_{\delta}$ is proper, $\BS$-bounding, and satisfies the ${\aleph}_{2}$-c.c.
 Furthermore, for each $\delta < {\omega}_{2}$, ${\forces}_{\delta} \; {\CH}$.
 
 Suppose for a contradiction that $\mathring{\HHH}$ is a ${\PP}_{{\omega}_{2}}$-name such that
 \begin{align*}
  p \; {\forces}_{{\omega}_{2}} \; {``\mathring{\HHH} \ \text{is a stable ordered-union ultrafilter on} \ \FIN''},
 \end{align*}
 for some $p \in {\PP}_{{\omega}_{2}}$.
 Then by a standard argument, there exists $\alpha \in {S}^{2}_{1}$ such that the diamond sequence at $\alpha$ codes a pair $\pr{\mathring{\GGG}}{p\restrict \alpha}$ such that $\mathring{\GGG}$ is a ${\PP}_{\alpha}$-name, $p \restrict \alpha \; {\forces}_{\alpha} \; {``\mathring{\GGG} \ \text{is a stable ordered-union ultrafilter on} \ \FIN''}$, and $p \; {\forces}_{{\omega}_{2}} \; \mathring{\GGG} \subseteq \mathring{\HHH}$.
 Let ${G}_{{\omega}_{2}}$ be a $(\V, {\PP}_{{\omega}_{2}})$-generic filter with $p \in {G}_{{\omega}_{2}}$.
 Let ${G}_{\alpha}$ denote its projection, that is ${G}_{\alpha} = \{q \restrict \alpha: q \in {G}_{{\omega}_{2}}\}$.
 In $\V\[{G}_{\alpha}\]$, ${\mathring{\HHH}}_{\alpha}\[{G}_{\alpha}\] = \mathring{\GGG}\[{G}_{\alpha}\]$ is a stable ordered-union ultrafilter on $\FIN$.
 Moreover, $\PP\left({\mathring{\HHH}}_{\alpha}\[{G}_{\alpha}\]\right)$ completely embeds into the completion of ${\PP}_{{\omega}_{2}} \slash {G}_{\alpha}$, and ${\PP}_{{\omega}_{2}} \slash {G}_{\alpha}$ is $\BS$-bounding.
 Therefore Theorem \ref{thm:noresseruction} implies that
 \begin{align*}
  {\forces}_{{\PP}_{{\omega}_{2}} \slash {G}_{\alpha}} \; {``\text{there is no stable ordered-union ultrafilter on} \ \FIN \ \text{extending} \ {\mathring{\HHH}}_{\alpha}\[{G}_{\alpha}\]''}.
 \end{align*}
 However ${G}_{{\omega}_{2}}$ is a $\left(\V\[{G}_{\alpha}\], {\PP}_{{\omega}_{2}} \slash {G}_{\alpha} \right)$-generic filter, $\V\[{G}_{\alpha}\]\[{G}_{{\omega}_{2}}\] = \V\[{G}_{{\omega}_{2}}\]$, and in $\V\[{G}_{{\omega}_{2}}\]$, $\mathring{\HHH}\[{G}_{{\omega}_{2}}\]$ is a stable ordered-union ultrafilter on $\FIN$ extending $\mathring{\GGG}\[{G}_{{\omega}_{2}}\] = \mathring{\GGG}\[{G}_{\alpha}\] = {\mathring{\HHH}}_{\alpha}\[{G}_{\alpha}\]$.
 This is a contradiction which shows that there are no stable ordered-union ultrafilters on $\FIN$ after forcing with ${\PP}_{{\omega}_{2}}$.
 
 An easy inductive argument using Theorem \ref{thm:preservesselective}, Corollary \ref{cor:proper+bounding}, Lemma \ref{lem:selectivestay}, and Lemma \ref{lem:P-pointiteration} shows that for every $\alpha < {\omega}_{2}$ and every $\delta \leq {\omega}_{2}$, ${\PP}_{\delta}$ preserves ${\UUU}_{\alpha}$.
 Let ${G}_{{\omega}_{2}}$ be a $(\V, {\PP}_{{\omega}_{2}})$-generic filter.
 In $\V\[{G}_{{\omega}_{2}}\]$, define ${\UUU}^{\ast}_{\alpha} = \left\{A \subseteq \omega: \exists B \in {\UUU}_{\alpha}\[B \subseteq A\]\right\}$, for every $\alpha < {\omega}_{2}$.
 By Lemma \ref{lem:selectivestay}, each ${\UUU}^{\ast}_{\alpha}$ is a selective ultrafilter on $\omega$, and by Corollary \ref{cor:non-RK-preservation}, ${\UUU}^{\ast}_{\alpha} \; {\not\equiv}_{RK} \; {\UUU}^{\ast}_{\beta}$, for every $\beta \neq \alpha$.
 Since ${\PP}_{{\omega}_{2}}$ preserves ${\aleph}_{2}$ because of the ${\aleph}_{2}$-c.c.\@, a model gotten by forcing with ${\PP}_{{\omega}_{2}}$ has all of the required properties.
\end{proof}
\section{Models with an intermediate number of selectives} \label{sec:intermediate}
This section introduces another partial order which allows us to control the number of selective ultrafilters in the final model.
Given a selective ultrafilter $\UUU$, we introduce a new partial order $\PP(\UUU)$ which is proper, $\BS$-bounding, and destroys $\UUU$ while preserving all selective ultrafilters that are not RK-isomorphic to $\UUU$.
By interleaving partial orders of the form $\PP(\UUU)$ with ones of the form $\PP(\HHH)$, it will be possible to produce models with no stable ordered-union ultrafilters and fewer than ${2}^{{\aleph}_{0}}$ RK-classes of selective ultrafilters.
Any cardinal strictly smaller than ${\aleph}_{2}$ can be obtained in this way.
Thus, for example, there is a model with precisely ${\aleph}_{0}$ distinct RK-classes of selective ultrafilters, but no stable ordered-union ultrafilters.
Furthermore, we can control exactly which RK-classes of selective ultrafilters from the ground model survive in the final forcing extension.

The definition of $\PP(\UUU)$ depends on the notion of a $k$-big set, which is similar to the notion of a $\pr{k}{s}$-big set, except that the semigroup operation $\cup$ plays no role here.
\begin{Def} \label{def:klbig}
 Let $k, l \in \omega$.
 $A \subseteq {2}^{l}$ is \emph{$k$-big} if for every $\sigma: k \rightarrow 2$, there exists $\tau \in A$ such that $\sigma \subseteq \tau$.
\end{Def}
\begin{Def} \label{def:PU}
 Define $\TT' = {\bigcup}_{l \in \omega}{{\prod}_{k \in l}{{2}^{k}}}$.
 $\pr{\TT'}{\subsetneq}$ is a tree.
 Let $\UUU$ be a selective ultrafilter on $\omega$.
 $p$ is called a \emph{$\UUU$-condition} if $p = {T}_{p} \subseteq \TT'$ is a subtree such that the following hold:
 \begin{enumerate}
  \item
  $\emptyset \in {T}_{p}$;
  \item
  $\forall f \in {T}_{p} \forall \dom(f) \leq n < \omega \exists g \in {T}_{p}\[f \subseteq g \wedge n \leq \dom(g)\]$;
  \item
  for each $k \in \omega$, ${H}_{p, k} \in \UUU$, where ${H}_{p, k} = $
  \begin{align*}
   \left\{ l \in \omega: \forall f \in \lv{{T}_{p}}{l}\[{\succc}_{{T}_{p}}(f) \ \text{is} \ k\text{-big}\] \right\}.
  \end{align*}
 \end{enumerate}
 Let $\PP(\UUU) = \left\{p: p \ \text{is a} \ \UUU\text{-condition} \right\}$.
 Define $q \leq p$ if and only if ${T}_{q} \subseteq {T}_{p}$, for all $p, q \in \PP(\UUU)$.
\end{Def}
The forcing $\PP(\UUU)$ may be seen as a tree version of the one Shelah used in \cite{MR1690694} to produce a model with no nowhere dense ultrafilters.

For now until the end of the proof of Lemma \ref{thm:noresseructionU}, $\UUU$ is a fixed selective ultrafilter on $\omega$.
The properties of $\PP(\UUU)$ are quite similar to those of $\PP(\HHH)$, and since the proofs of these properties are also similar, but easier, fewer details will be provided.
\begin{Lemma} \label{lem:oneU}
 $\TT' \in \PP(\UUU)$.
\end{Lemma}
\begin{proof}
 All of the requirements with the possible exception of (3) of Definition \ref{def:PU} are clear.
 To verify this, given $k \in \omega$, define $H = \{l \in \omega: l > k\}$, and note that $H \in \UUU$.
 Consider $l \in H$.
 Consider $f \in \lv{\TT'}{l}$.
 Then $A = {\succc}_{\TT'}(f) = {2}^{l}$.
 It is clear that $A$ is $k$-big as $k < l$.
 Thus $H \subseteq {H}_{\TT', k}$, and so ${H}_{\TT', k} \in \UUU$.
\end{proof}
\begin{Lemma} \label{lem:TfU}
 Suppose $p \in \PP(\UUU)$.
 Then for any $f \in {T}_{p}$, $q = {T}_{p}\langle f \rangle \in \PP(\UUU)$ and $q \leq p$.
\end{Lemma}
\begin{proof}
 The argument for all the requirements, with the possible exception of Clause (3) of Definition \ref{def:PU}, is identical to the proof of Lemma \ref{def:Tf}.
 Fix $k \in \omega$.
 Let $H = \{l \in {H}_{p, k}: l > \dom(f)\} \in \UUU$.
 Suppose $l \in H$ and $e \in \lv{{T}_{p}\langle f \rangle}{l}$.
 Then $l \in {H}_{p, k}$, $e \in \lv{{T}_{p}}{l}$, and $\dom(e) = l > \dom(f)$, whence $f \subseteq e$.
 Thus ${\succc}_{{T}_{p}\langle f \rangle}(e) = {\succc}_{{T}_{p}}(e)$ is $k$-big.
 Thus $H \subseteq {H}_{q, k}$, and so ${H}_{q, k} \in \UUU$.
 Thus $q \in \PP(\UUU)$ and $q \leq p$.
\end{proof}
\begin{Lemma} \label{lem:amalgamU}
 Let $p \in \PP(\UUU)$.
 Let $l \in \omega$, $1 \leq m < \omega$, and ${e}_{1}, \dotsc, {e}_{m} \in \lv{{T}_{p}}{l}$.
 If ${p}_{1}, \dotsc, {p}_{m} \in \PP(\UUU)$ are such that $\forall 1 \leq i \leq m\[{p}_{i} \leq {T}_{p}\langle {e}_{i} \rangle\]$, then $q = {T}_{q} = {\bigcup}_{1 \leq i \leq m}{{T}_{{p}_{i}}} \in \PP(\UUU)$, for each $1 \leq i \leq m$, ${p}_{i} \leq q$, and $q \leq p$.
\end{Lemma}
\begin{proof}
 Once again, only Clause (3) of Definition \ref{def:PU} may require some argument.
 For each $k \in \omega$, define $H = {H}_{{p}_{1}, k} \cap \dotsb \cap {H}_{{p}_{m}, k} \in \UUU$.
 Consider $l \in H$.
 Consider $f \in \lv{{T}_{q}}{l}$.
 Then for some $1 \leq i \leq m$, $f \in \lv{{T}_{{p}_{i}}}{l}$.
 Since ${\succc}_{{T}_{{p}_{i}}}(f)$ is $k$-big and ${\succc}_{{T}_{{p}_{i}}}(f) \subseteq {\succc}_{{T}_{q}}(f) \subseteq {2}^{l}$, ${\succc}_{{T}_{q}}(f)$ is also $k$-big.
 Thus $H \subseteq {H}_{q, k}$, and so ${H}_{q, k} \in \UUU$.
 Therefore $q \in \PP(\UUU)$.
 Finally, since ${T}_{q} \subseteq {T}_{p}$, $q \leq p$, and for all $1 \leq i \leq m$, since ${T}_{{p}_{i}} \subseteq {T}_{q}$, ${p}_{i} \leq q$.
\end{proof}
\begin{Lemma} \label{lem:fusionU}
 Suppose $\seq{p}{i}{\in}{\omega}$ and $\seq{l}{i}{\in}{\omega}$ satisfy:
 \begin{enumerate}
  \item
  $\forall i \in \omega\[{p}_{i} \in \PP(\UUU)\]$ and $\forall i \in \omega\[{p}_{i+1} \leq {p}_{i}\]$;
  \item
  $\seq{l}{i}{\in}{\omega} \in \BS$, $\forall i \in \omega\[{l}_{i} < {l}_{i+1}\]$, and $\{{l}_{i}: i \in \omega\} \in \UUU$;
  \item
  for each $i \in \omega$, ${l}_{i+1} \in {H}_{{p}_{i+1}, {l}_{i}+1}$;
  \item
  for each $i \leq j < \omega$ and for each $e \in \lv{{T}_{{p}_{j}}}{{l}_{i}}$, ${\succc}_{{T}_{{p}_{i}}}(e) \subseteq {\succc}_{{T}_{{p}_{j}}}(e)$.
 \end{enumerate}
 Then $q = {T}_{q} = {\bigcap}_{i \in \omega}{{T}_{{p}_{i}}} \in \PP(\UUU)$.
\end{Lemma}
\begin{proof}
 For ease of notation in this proof, the symbols ${T}_{i}$ will replace ${T}_{{p}_{i}}$, ${k}_{i}$ will denote ${l}_{i}+1$, and ${H}_{i, k}$ will be used for ${H}_{{p}_{i}, k}$, for all $i \in \omega$ and $k \in \omega$.
 With the exception of Clause (3) of Definition \ref{def:PU}, the arguments for all the other conditions are identical to the corresponding arguments in the proof of Lemma \ref{lem:fusion}.
 To see that ${T}_{q}$ satisfies (3) of Definition \ref{def:PU}, fix $k \in \omega$.
 Choose ${i}_{0} \in \omega$ such that $k < {l}_{{i}_{0}}$.
 Let $H = \left\{ {l}_{j}: \omega > j > {i}_{0} \right\} \in \UUU$.
 Suppose $l \in H$.
 Then $l = {l}_{i+1}$, where ${i}_{0} \leq i < \omega$.
 Note that $k < {l}_{{i}_{0}} \leq {l}_{i} < {l}_{i}+1={k}_{i}$.
 By hypothesis, ${l}_{i+1} \in {H}_{i+1, {k}_{i}}$.
 Now, fix $f \in \lv{{T}_{q}}{l}$.
 It needs to be seen that ${\succc}_{{T}_{q}}(f)$ is $k$-big.
 To this end, let $\sigma: k \rightarrow 2$ be fixed.
 Since $f \in \lv{{T}_{i+1}}{{l}_{i+1}}$, ${\succc}_{{T}_{i+1}}(f)$ is ${k}_{i}$-big.
 Now choose any ${\sigma}^{\ast}: {k}_{i} \rightarrow 2$ with $\sigma \subseteq {\sigma}^{\ast}$.
 Find $\tau \in {\succc}_{{T}_{i+1}}(f)$ such that ${\sigma}^{\ast} \subseteq \tau$.
 Consider any $i+1 \leq j < \omega$.
 Then $f \in \lv{{T}_{j}}{{l}_{i+1}}$, and the hypothesis is that ${\succc}_{{T}_{i+1}}(f) \subseteq {\succc}_{{T}_{j}}(f)$.
 Thus $\tau \in {\succc}_{{T}_{j}}(f)$, and so ${f}^{\frown}{\langle \tau \rangle} \in {T}_{j}$.
 Hence $\forall i+1 \leq j < \omega\[{f}^{\frown}{\langle \tau \rangle} \in {T}_{j}\]$, whence ${f}^{\frown}{\langle \tau \rangle} \in {T}_{q}$.
 Therefore $\tau \in {\succc}_{{T}_{q}}(f)$.
 Further, $\sigma \subseteq {\sigma}^{\ast} \subseteq \tau$.
 This proves that ${\succc}_{{T}_{q}}(f)$ is $k$-big.
 Thus $H \subseteq {H}_{q, k}$, and so ${H}_{q, k} \in \UUU$.
 This concludes the proof that $q \in \PP(\UUU)$.
\end{proof}
\begin{Def} \label{def:GcondU}
 Define the following two player game called the \emph{$\UUU$-condition game} and denoted \emph{${\Game}^{\mathtt{Cond}}\left(\UUU\right)$}.
 Players I and II alternatively choose $\pr{{p}_{i}}{{A}_{i}}$ and ${k}_{i}$ respectively, where
 \begin{enumerate}
  \item
  ${p}_{i} \in \PP(\UUU)$ and ${A}_{i} \in \UUU$;
  \item
  ${k}_{i} \in {A}_{i}$;
  \item
  there exists $\left\langle {p}_{i, e}: e \in \lv{{T}_{{p}_{i}}}{{k}_{i}+1} \right\rangle$ such that
  \begin{align*}
   \forall e \in \lv{{T}_{{p}_{i}}}{{k}_{i}+1}\[{p}_{i, e} \leq {T}_{{p}_{i}}\langle e \rangle\]
  \end{align*}
  and ${p}_{i+1} = {T}_{{p}_{i+1}} = \bigcup\left\{{T}_{{p}_{i, e}}: e \in \lv{{T}_{{p}_{i}}}{{k}_{i}+1}\right\}$.
 \end{enumerate}
 Together they construct the sequence
 \begin{align*}
  \pr{{p}_{0}}{{A}_{0}}, {k}_{0}, \pr{{p}_{1}}{{A}_{1}}, {k}_{1}, \dotsc,
 \end{align*}
 where each $\pr{{p}_{i}}{{A}_{i}}$ has been played by Player I and ${k}_{i}$ has been chosen by Player II in response subject to Conditions (1)--(3).
 Player II wins if and only if $\forall i < j < \omega\[{k}_{i} < {k}_{j}\]$, $\left\{{k}_{i}: i < \omega\right\} \in \UUU$, and $q = {T}_{q} = {\bigcap}_{i \in \omega}{{T}_{{p}_{i}}} \in \PP(\UUU)$.
\end{Def}
\begin{Lemma} \label{lem:condgameU}
 Player I does not have a winning strategy in ${\Game}^{\mathtt{Cond}}\left(\UUU\right)$.
\end{Lemma}
\begin{proof}
 Suppose for a contradiction that $\Sigma$ is a winning strategy for Player I in ${\Game}^{\mathtt{Cond}}\left(\UUU\right)$.
 Define $\Pi$ and $\Phi$ such that:
 \begin{enumerate}
  \item
  $\Pi$ is a strategy for Player I in ${\Game}^{\mathtt{Sel}}\left(\UUU\right)$;
  \item
  for each $n \in \omega$, if $\left\langle \pr{{B}_{i}}{{k}_{i}}: i \leq n \right\rangle$ is a partial run of ${\Game}^{\mathtt{Sel}}\left(\UUU\right)$ in which Player I has followed $\Pi$, then $\Phi\left(\left\langle \pr{{B}_{i}}{{k}_{i}}: i \leq n \right\rangle\right) = \left\langle \pr{{p}_{i}}{{A}_{i}}: i \leq n \right\rangle$ and
  \begin{align*}
   \left\langle \pr{\pr{{p}_{i}}{{A}_{i}}}{{k}_{i}}: i \leq n \right\rangle
  \end{align*}
  is a partial play of ${\Game}^{\mathtt{Cond}}\left(\UUU\right)$ in which Player I has followed $\Sigma$ and it has the property that $\forall i < n\[ {B}_{i+1} = {A}_{i+1} \cap {H}_{{p}_{i+1}, {k}_{i}+1} \cap \{k \in \omega: k > {k}_{i} \} \]$;
  \item
  for each $n \in \omega$, if $\left\langle \pr{{B}_{i}}{{k}_{i}}: i \leq n+1 \right\rangle$ is a partial run of ${\Game}^{\mathtt{Sel}}\left(\UUU\right)$ in which Player I has followed $\Pi$, then
  \begin{align*}
   \Phi\left( \left\langle \pr{{B}_{i}}{{k}_{i}}: i \leq n \right\rangle \right) = \Phi\left( \left\langle \pr{{B}_{i}}{{k}_{i}}: i \leq n+1 \right\rangle \right)\restrict n+1.
  \end{align*}
 \end{enumerate}
 $\Pi$ and $\Phi$ will be defined inductively.
 Let $\Sigma(\emptyset) = \pr{{p}_{0}}{{A}_{0}}$ define $\Pi(\emptyset) = {B}_{0} = {A}_{0}$.
 As ${B}_{0} \in \UUU$, this is a valid move for Player I in ${\Game}^{\mathtt{Sel}}\left(\UUU\right)$.
 Note that every partial run of ${\Game}^{\mathtt{Sel}}\left(\UUU\right)$ of length $1$ in which Player I has followed $\Pi$ will have the form $\pr{{B}_{0}}{{k}_{0}}$, where ${k}_{0} \in {B}_{0} = {A}_{0}$.
 For any such $\pr{{B}_{0}}{{k}_{0}}$, define $\Phi(\pr{{B}_{0}}{{k}_{0}}) = \langle \pr{{p}_{0}}{{A}_{0}} \rangle$, and note that $\pr{\pr{{p}_{0}}{{A}_{0}}}{{k}_{0}}$ is a partial run of ${\Game}^{\mathtt{Cond}}\left(\UUU\right)$ in which Player I has followed $\Sigma$.
 Now suppose $n \in \omega$, $\left\langle \pr{{B}_{i}}{{k}_{i}}: i \leq n \right\rangle$ is a partial run of ${\Game}^{\mathtt{Sel}}\left(\UUU\right)$ in which Player I has followed $\Pi$, $\Phi\left(\left\langle \pr{{B}_{i}}{{k}_{i}}: i \leq n \right\rangle\right) = \left\langle \pr{{p}_{i}}{{A}_{i}}: i \leq n \right\rangle$,
 \begin{align*}
  \left\langle \pr{\pr{{p}_{i}}{{A}_{i}}}{{k}_{i}}: i \leq n \right\rangle
 \end{align*}
 is a partial run of ${\Game}^{\mathtt{Cond}}\left(\UUU\right)$ in which Player I has followed $\Sigma$, and
 \begin{align*}
  \forall i < n\[ {B}_{i+1} = {A}_{i+1} \cap {H}_{{p}_{i+1}, {k}_{i}+1} \cap \{k \in \omega: k > {k}_{i} \} \].
 \end{align*}
 Let $\Sigma\left( \left\langle \pr{\pr{{p}_{i}}{{A}_{i}}}{{k}_{i}}: i \leq n \right\rangle \right) = \pr{{p}_{n+1}}{{A}_{n+1}}$.
 Then ${p}_{n+1} \in \PP(\UUU)$ and ${A}_{n+1} \in \UUU$.
 Hence ${B}_{n+1} = {A}_{n+1} \cap {H}_{{p}_{n+1}, {k}_{n}+1} \cap \{k \in \omega: k > {k}_{n} \} \in \UUU$.
 Note that ${B}_{n+1}$ is therefore a legitimate move for Player I in ${\Game}^{\mathtt{Sel}}\left(\UUU\right)$.
 Define $\Pi\left( \left\langle \pr{{B}_{i}}{{k}_{i}}: i \leq n \right\rangle \right) = {B}_{n+1}$.
 Note that any continuation of $\left\langle \pr{{B}_{i}}{{k}_{i}}: i \leq n \right\rangle$ to length $n+2$ in which Player I follows $\Pi$ must have the form $\left\langle \pr{{B}_{i}}{{k}_{i}}: i \leq n+1 \right\rangle$, where ${k}_{n+1} \in {B}_{n+1}$.
 Given any such $\left\langle \pr{{B}_{i}}{{k}_{i}}: i \leq n+1 \right\rangle$, define $\Phi\left( \left\langle \pr{{B}_{i}}{{k}_{i}}: i \leq n+1 \right\rangle \right) = \left\langle \pr{{p}_{i}}{{A}_{i}}: i \leq n+1 \right\rangle$.
 Note that $\left\langle \pr{\pr{{p}_{i}}{{A}_{i}}}{{k}_{i}}: i \leq n+1 \right\rangle$ is a partial run of ${\Game}^{\mathtt{Cond}}\left(\UUU\right)$ in which Player I has followed $\Sigma$ because of the definition of $\pr{{p}_{n+1}}{{A}_{n+1}}$ and ${k}_{n+1} \in {B}_{n+1} \subseteq {A}_{n+1}$.
 Further, by definition and by the induction hypothesis,
 \begin{align*}
  \forall i \leq n\[ {B}_{i+1} = {A}_{i+1} \cap {H}_{{p}_{i+1}, {k}_{i}+1} \cap \{k \in \omega: k > {k}_{i} \} \].
 \end{align*}
 Thus the inductive definition satisfies (1)--(3).
 This concludes the definition of $\Pi$ and $\Phi$.
 
 Since $\Pi$ is not a winning strategy for Player I in ${\Game}^{\mathtt{Sel}}\left(\UUU\right)$, there is a play $\left\langle \pr{{B}_{i}}{{k}_{i}}: i < \omega \right\rangle$ of ${\Game}^{\mathtt{Sel}}\left(\UUU\right)$ in which Player I follows $\Pi$ and loses.
 There exists $\left\langle \pr{{p}_{i}}{{A}_{i}}: i < \omega \right\rangle$ such that for each $n \in \omega$, $\Phi\left( \left\langle \pr{{B}_{i}}{{k}_{i}}: i \leq n \right\rangle \right) = \left\langle \pr{{p}_{i}}{{A}_{i}}: i \leq n \right\rangle$.
 Therefore, $\left\langle \pr{\pr{{p}_{i}}{{A}_{i}}}{{k}_{i}}: i < \omega \right\rangle$ is a play of ${\Game}^{\mathtt{Cond}}\left(\UUU\right)$ in which Player I has followed $\Sigma$, and $\forall i < \omega\[ {B}_{i+1} = {A}_{i+1} \cap {H}_{{p}_{i+1}, {k}_{i}+1} \cap \{k \in \omega: k > {k}_{i} \} \]$.
 Since Player II wins the play $\left\langle \pr{{B}_{i}}{{k}_{i}}: i < \omega \right\rangle$, $\left\{ {k}_{i}: i < \omega \right\} \in \UUU$.
 Lemma \ref{lem:fusionU} will be used to verify that $q = {T}_{q} = {\bigcap}_{i \in \omega}{{T}_{{p}_{i}}} \in \PP(\UUU)$.
 Note that by (3) of Definition \ref{def:GcondU} and by Lemma \ref{lem:amalgamU}, ${p}_{i+1} \leq {p}_{i}$, for all $i \in \omega$.
 For each $i \in \omega$, ${k}_{i+1} \in {B}_{i+1}$, and so ${k}_{i+1} \in {H}_{{p}_{i+1}, {k}_{i}+1}$, and ${k}_{i+1} > {k}_{i}$.
 Therefore $\forall i < j < \omega\[{k}_{i} < {k}_{j}\]$.
 Next, fix some $i < \omega$.
 It will be proved by induction on $j$ that for each $i \leq j < \omega$ and for each $e \in \lv{{T}_{{p}_{j}}}{{k}_{i}}$, ${\succc}_{{T}_{{p}_{i}}}(e) \subseteq {\succc}_{{T}_{{p}_{j}}}(e)$.
 This is clear when $i=j$.
 Assume this is true for some $i \leq j$.
 Fix $e \in \lv{{T}_{{p}_{j+1}}}{{k}_{i}}$ and consider $\sigma \in {\succc}_{{T}_{{p}_{i}}}(e)$.
 Since ${T}_{{p}_{j+1}} \subseteq {T}_{{p}_{j}}$, $e \in \lv{{T}_{{p}_{j}}}{{k}_{i}}$.
 So by the induction hypothesis, $\sigma \in {\succc}_{{T}_{{p}_{j}}}(e)$.
 Therefore ${e}^{\frown}{\langle \sigma \rangle} \in {T}_{{p}_{j}}$ and $\dom\left( {e}^{\frown}{\langle \sigma \rangle} \right) = {k}_{i}+1 \leq {k}_{j}+1$.
 Choose ${e}^{\ast}$ such that ${e}^{\frown}{\langle \sigma \rangle} \subseteq {e}^{\ast}$ and ${e}^{\ast} \in \lv{{T}_{{p}_{j}}}{{k}_{j}+1}$.
 By (3) of Definition \ref{def:GcondU}, there exists ${p}_{j, {e}^{\ast}} \leq {T}_{{p}_{j}}\langle {e}^{\ast} \rangle$ such that ${T}_{{p}_{j, {e}^{\ast}}} \subseteq {T}_{{p}_{j+1}}$.
 Since ${e}^{\frown}{\langle \sigma \rangle} \subseteq {e}^{\ast}$, ${e}^{\frown}{\langle \sigma \rangle} \in {T}_{{p}_{j, {e}^{\ast}}}$.
 Therefore ${e}^{\frown}{\langle \sigma \rangle} \in {T}_{{p}_{j+1}}$, whence $\sigma \in {\succc}_{{T}_{{p}_{j+1}}}(e)$, as required.
 This concludes the induction.
 Thus the hypotheses of Lemma \ref{lem:fusionU} are all satisfied, and so $q = {T}_{q} = {\bigcap}_{i \in \omega}{{T}_{{p}_{i}}} \in \PP(\UUU)$.
 However, this means that Player II wins the play $\left\langle \pr{\pr{{p}_{i}}{{A}_{i}}}{{k}_{i}}: i < \omega \right\rangle$ of ${\Game}^{\mathtt{Cond}}\left(\UUU\right)$ even though Player I has followed $\Sigma$ during this play, contradicting the hypothesis that $\Sigma$ is a winning strategy for Player I in ${\Game}^{\mathtt{Cond}}\left(\UUU\right)$.
\end{proof}
\begin{Lemma} \label{lem:winning1U}
 Suppose $\left\langle \pr{\pr{{p}_{i}}{{A}_{i}}}{{k}_{i}}: i \in \omega \right\rangle$ is a run of ${\Game}^{\mathtt{Cond}}\left(\UUU\right)$ which is won by Player II.
 If $q = {T}_{q} = {\bigcap}_{i \in \omega}{{T}_{{p}_{i}}}$, then for each $i \in \omega$,
 \begin{align*}
  \forall f \in {T}_{{p}_{i}}\[\dom(f) \leq {k}_{i}+1 \implies f \in {T}_{q}\].
 \end{align*}
\end{Lemma}
\begin{proof}
 Recall that by Clause (3) of Definition \ref{def:GcondU} and by Lemma \ref{lem:amalgamU}, ${p}_{i+1} \leq {p}_{i}$, for all $i \in \omega$.
 Fix $i \in \omega$ and $f \in {T}_{{p}_{i}}$ with $\dom(f) \leq {k}_{i}+1$.
 It will be proved by induction on $j$ that $\forall i \leq j < \omega\[f \in {T}_{{p}_{j}}\]$.
 When $j=i$, there is nothing to prove.
 Suppose the statement holds for some $i \leq j < \omega$.
 So $f \in {T}_{{p}_{j}}$ and $\dom(f) \leq {k}_{i} + 1 \leq {k}_{j}+1$.
 Choose $e \in {T}_{{p}_{j}}$ such that $f \subseteq e$ and $\dom(e) = {k}_{j}+1$.
 By Clause (3) of Definition \ref{def:GcondU}, there exists ${p}_{j, e} \leq {T}_{{p}_{j}} \langle e \rangle$ so that ${T}_{{p}_{j, e}} \subseteq {T}_{{p}_{j+1}}$.
 As $f \subseteq e$, $f \in {T}_{{p}_{j, e}}$, and so $f \in {T}_{{p}_{j+1}}$.
 This concludes the induction.
 Thus, $\forall i \leq j < \omega\[f \in {T}_{{p}_{j}}\]$, whence $f \in {\bigcap}_{l \in \omega}{{T}_{{p}_{l}}} = {T}_{q}$.
\end{proof}
\begin{Lemma} \label{lem:winning2U}
 Suppose $\left\langle \pr{\pr{{p}_{i}}{{A}_{i}}}{{k}_{i}}: i \in \omega \right\rangle$ is a run of ${\Game}^{\mathtt{Cond}}\left(\UUU\right)$ which is won by Player II.
 Let $i \in \omega$ and $q = {T}_{q} = {\bigcap}_{j \in \omega}{{T}_{{p}_{j}}}$.
 If ${k}_{i+1} \in {H}_{{p}_{i+1}, {k}_{i}+1}$, then ${k}_{i+1} \in {H}_{q, {k}_{i}+1}$.
\end{Lemma}
\begin{proof}
 Assume that ${k}_{i+1} \in {H}_{{p}_{i+1}, {k}_{i}+1}$.
 Consider some $e \in \lv{{T}_{q}}{{k}_{i+1}}$.
 Then $e \in \lv{{T}_{{p}_{i+1}}}{{k}_{i+1}}$ and so ${\succc}_{{T}_{{p}_{i+1}}}(e)$ is ${k}_{i}+1$-big.
 If $\sigma \in {\succc}_{{T}_{{p}_{i+1}}}(e)$, then ${e}^{\frown}{\langle \sigma \rangle} \in {T}_{{p}_{i+1}}$ and $\dom\left({e}^{\frown}{\langle \sigma \rangle}\right) = {k}_{i+1}+1$.
 Therefore, by Lemma \ref{lem:winning1U}, ${e}^{\frown}{\langle \sigma \rangle} \in {T}_{q}$, whence $\sigma \in {\succc}_{{T}_{q}}(e)$.
 Thus ${\succc}_{{T}_{{p}_{i+1}}}(e) \subseteq {\succc}_{{T}_{q}}(e) \subseteq {2}^{{k}_{i+1}}$, and so ${\succc}_{{T}_{q}}(e)$ is ${k}_{i}+1$-big.
 This shows ${k}_{i+1} \in {H}_{q, {k}_{i}+1}$.
\end{proof}
\begin{Lemma} \label{lem:continuousU}
 Let $\theta$ be a sufficiently large regular cardinal.
 Assume $M \prec H(\theta)$ is countable and that $M$ contains all relevant parameters.
 Suppose $f: \omega \rightarrow M$ is such that $\forall n \in \omega\[f(n) \in {\V}^{\PP(\UUU)} \ \text{and} \ {\forces}_{\PP(\UUU)}\;{f(n) \in \V}\]$.
 Let $p \in \PP(\UUU) \cap M$.
 Then there exist $q, \seq{k}{i}{\in}{\omega}$, and $F$ such that:
 \begin{enumerate}
  \item
  $q \leq p$, $\seq{k}{i}{\in}{\omega} \in \BS$, $\forall i < j < \omega\[{k}_{i} < {k}_{j}\]$, $F$ is a function, $\dom(F) =$
  \begin{align*}
   \left\{ \pr{i}{e}: i \in \omega \wedge e \in \lv{{T}_{q}}{{k}_{i}+1} \right\};
  \end{align*}
  \item
  $\{{k}_{i}: i \in \omega\} \in \UUU$, $\forall i \in \omega\[{k}_{i+1} \in {H}_{q, {k}_{i}+1}\]$;
  \item
  for each $i \leq j < \omega$ and each $e \in \lv{{T}_{q}}{{k}_{j}+1}$,
  \begin{align*}
   {T}_{q}\langle e \rangle \; {\forces}_{\PP(\UUU)} \; {f(i) = F\left(\pr{i}{e\restrict {k}_{i}+1}\right)};
  \end{align*}
  \item
  for any $\pr{i}{e} \in \dom(F)$, $F(\pr{i}{e}) \in M$.
 \end{enumerate}
\end{Lemma}
\begin{proof}
 Define a strategy $\Sigma$ for Player I in ${\Game}^{\mathtt{Cond}}\left(\UUU\right)$ as follows.
 ${p}_{0} = p \in \PP(\UUU)$ and ${A}_{0} = \omega \in \UUU$.
 Note that ${p}_{0} \in M$.
 Define $\Sigma(\emptyset) = \pr{{p}_{0}}{{A}_{0}}$.
 Now suppose that $i \in \omega$ and that $\left\langle \pr{\pr{{p}_{j}}{{A}_{j}}}{{k}_{j}}: j \leq i \right\rangle$ is a partial run of ${\Game}^{\mathtt{Cond}}\left(\UUU\right)$ in which Player I has followed $\Sigma$ and that ${p}_{i} \in M$.
 Let $l = {k}_{i}+1$, for ease of notation.
 For any $e \in \lv{{T}_{{p}_{i}}}{l}$, ${T}_{{p}_{i}}\langle e \rangle \in \PP(\UUU)$ and by hypothesis, ${\forces}_{\PP(\UUU)}\;{f(i) \in \V}$.
 Since ${p}_{i} \in M$ and $f(i) \in M$, there exist sequences $\left\langle {p}_{i, e}: e \in \lv{{T}_{{p}_{i}}}{l} \right\rangle \in M$ and $\left\langle {x}_{i, e}: e \in \lv{{T}_{{p}_{i}}}{l} \right\rangle \in M$ such that
 \begin{align*}
  \forall e \in \lv{{T}_{{p}_{i}}}{l}\[{p}_{i, e} \leq {T}_{{p}_{i}}\langle e \rangle \ \text{and} \ {p}_{i, e} \; {\forces}_{\PP(\UUU)} \; {f(i) = {x}_{i, e}}\].
 \end{align*}
 Define ${p}_{i+1} = {T}_{{p}_{i+1}} = \bigcup{\left\{ {T}_{{p}_{i, e}}: e \in \lv{{T}_{{p}_{i}}}{l} \right\}}$.
 Note that ${p}_{i+1} \in M$ and that ${p}_{i+1} \in \PP(\UUU)$ by Lemma \ref{lem:amalgamU}.
 Hence ${A}_{i+1} = {H}_{{p}_{i+1}, l} \in \UUU$.
 Define
 \begin{align*}
  \Sigma\left(\left\langle \pr{\pr{{p}_{j}}{{A}_{j}}}{{k}_{j}}: j \leq i \right\rangle\right) = \pr{{p}_{i+1}}{{A}_{i+1}}.
 \end{align*}
 This completes the definition of $\Sigma$.
 Note that by definition, if $\left\langle \pr{\pr{{p}_{i}}{{A}_{i}}}{{k}_{i}}: i < \omega \right\rangle$ is a run of ${\Game}^{\mathtt{Cond}}\left(\UUU\right)$ in which Player I has followed $\Sigma$, then ${p}_{i} \in M$, for all $i < \omega$.
 
 Since $\Sigma$ is not a winning strategy for Player I, there is a run $\left\langle \pr{\pr{{p}_{i}}{{A}_{i}}}{{k}_{i}}: i < \omega \right\rangle$ of ${\Game}^{\mathtt{Cond}}\left(\UUU\right)$ in which Player I has followed $\Sigma$ and lost.
 Then $\seq{k}{i}{\in}{\omega} \in \BS$, $\forall i < j < \omega\[{k}_{i} < {k}_{j}\]$, and $\{{k}_{i}: i \in \omega \} \in \UUU$.
 Again, for ease of notation, denote ${k}_{i}+1$ by ${l}_{i}$.
 Let $q = {T}_{q} = {\bigcap}_{i \in \omega}{{T}_{{p}_{i}}}$.
 Then $q \in \PP(\UUU)$ and $q \leq {p}_{0} = p$.
 Moreover, for any $i \in \omega$, ${k}_{i+1} \in {A}_{i+1} = {H}_{{p}_{i+1}, {l}_{i}}$, and so by Lemma \ref{lem:winning2U}, ${k}_{i+1} \in {H}_{q, {l}_{i}}$, as required for (2).
 Next, by the definition of $\Sigma$, for each $i \in \omega$, ${p}_{i} \in M$ and there exist sequences $\left\langle {p}_{i, e}: e \in \lv{{T}_{{p}_{i}}}{{l}_{i}} \right\rangle \in M$ and $\left\langle {x}_{i, e}: e \in \lv{{T}_{{p}_{i}}}{{l}_{i}} \right\rangle \in M$ as described in the previous paragraph.
 Given $i \in \omega$ and $e \in \lv{{T}_{q}}{{l}_{i}}$, then $e \in \lv{{T}_{{p}_{i}}}{{l}_{i}}$, and since $e \in M$, so ${x}_{i, e} \in M$.
 Define $F(\pr{i}{e}) = {x}_{i, e}$.
 The argument for (3) is identical to the corresponding argument in the proof of Lemma \ref{lem:continuous}.
\end{proof}
\begin{Cor} \label{cor:proper+boundingU}
 $\PP(\UUU)$ is proper and ${\omega}^{\omega}$-bounding.
\end{Cor}
\begin{proof}
 The proof is identical to the proof of Corollary \ref{cor:proper+bounding}, using Lemma \ref{lem:continuousU} in place of Lemma \ref{lem:continuous}.
\end{proof}
The construction of Lemma \ref{lem:continuousU} can be easily modified to show that $\PP(\UUU)$ has the weak Sacks property.
Taken in conjunction with Lemmas \ref{lem:c1HpreservedU} and \ref{thm:noresseructionU}, this shows that $\PP(\UUU)$ is a forcing with the weak Sacks property which does not add an independent real and yet does not preserve P-points.
Forcings with these properties were considered by Zapletal in \cite{MR2520149}
\begin{Lemma} \label{lem:c1HpreservedU}
 Suppose $\VVV$ is a selective ultrafilter on $\omega$.
 Suppose $\VVV \; {\not\equiv}_{RK} \; \UUU$.
 Then $\PP(\UUU)$ preserves $\VVV$.
\end{Lemma}
\begin{proof}
 Let $\mathring{x}$ be a $\PP(\UUU)$-name such that ${\forces}_{\PP(\UUU)} \; {\mathring{x}: \omega \rightarrow 2}$ and let $p \in \PP(\UUU)$.
 
 For each $q \leq p$ let ${E}_{q} = \{n \in \omega: \exists r \leq q \[r \; {\forces}_{\PP(\UUU)} \; {\mathring{x}(n) = 1}\]\}$.
 Suppose first that there exists $q \leq p$ such that ${E}_{q} \notin \VVV$.
 Then $A = \omega \setminus {E}_{q} \in \VVV$, and for all $n \in A$, $q \; {\forces}_{\PP(\UUU)} \; {\mathring{x}(n) = 0}$.
 This is as required by (1) of Lemma \ref{lem:preserving}.
 
 It will henceforth be assumed that for all $q \leq p$, ${E}_{q} \in \VVV$.
 Define $\Sigma$ and $\Phi$ such that:
 \begin{enumerate}
  \item
  $\Sigma$ is a strategy for Player I in ${\Game}^{\mathtt{SelSel}}\left(\VVV, \UUU\right)$;
  \item
  for each $n \in \omega$, if $\left\langle \pr{{C}_{i}}{{o}_{i}}: i \leq 2n \right\rangle$ is a partial run of ${\Game}^{\mathtt{SelSel}}\left(\VVV, \UUU\right)$ in which Player I has followed $\Sigma$, then $\Phi\left(\left\langle \pr{{C}_{i}}{{o}_{i}}: i \leq 2n \right\rangle\right) = \seq{p}{i}{\leq}{n}$ and
  \begin{enumerate}
   \item[(a)]
   $\forall i \leq n\[{p}_{i} \in \PP(\UUU)\]$, $\forall i < n\[{p}_{i+1} \leq {p}_{i}\]$, ${p}_{0} \leq p$;
   \item[(b)]
   for each $i < j \leq n$ and for each $e \in \lv{{T}_{{p}_{j}}}{{o}_{2i+1}}$, ${\succc}_{{T}_{{p}_{i}}}(e) \subseteq {\succc}_{{T}_{{p}_{j}}}(e)$;
   \item[(c)]
   for each $i \leq n$, ${p}_{i} \; {\forces}_{\PP(\UUU)} \; {\mathring{x}({o}_{2i}) = 1}$;
  \end{enumerate}
  \item
  for each $n \in \omega$, if $\left\langle \pr{{C}_{i}}{{o}_{i}}: i \leq 2n+3 \right\rangle$ is a partial run of ${\Game}^{\mathtt{SelSel}}\left(\VVV, \UUU\right)$ in which Player I has followed $\Sigma$ and if
  \begin{align*}
   \Phi\left(\left\langle \pr{{C}_{i}}{{o}_{i}}: i \leq 2n+2 \right\rangle\right) = \seq{p}{i}{\leq}{n+1},
  \end{align*}
  then ${o}_{2n+3} \in {H}_{{p}_{n+1}, {o}_{2n+1}+1}$ and ${o}_{2n+3} > {o}_{2n+1}$;
  \item
  for each $n \in \omega$, if $\left\langle \pr{{C}_{i}}{{o}_{i}}: i \leq 2n+2 \right\rangle$ is a partial run of ${\Game}^{\mathtt{SelSel}}\left(\VVV, \UUU\right)$ in which Player I has followed $\Sigma$, then
  \begin{align*}
   \Phi\left(\left\langle \pr{{C}_{i}}{{o}_{i}}: i \leq 2n \right\rangle\right) = \Phi\left(\left\langle \pr{{C}_{i}}{{o}_{i}}: i \leq 2n+2 \right\rangle\right) \restrict n+1.
  \end{align*}
 \end{enumerate}
 Suppose for a moment that $\Sigma$ and $\Phi$ satisfying (1)--(4) can be defined.
 Since $\Sigma$ is not a winning strategy for Player I, there is a play $\left\langle \pr{{C}_{i}}{{o}_{i}}: i < \omega \right\rangle$ of ${\Game}^{\mathtt{SelSel}}\left(\VVV, \UUU\right)$ in which Player I follows $\Sigma$ and loses.
 Then $\langle {o}_{2i+1}: i < \omega \rangle \in \BS$ and $\{{o}_{2i+1}: i < \omega\} \in \UUU$.
 Also, there is a sequence $\seq{p}{i}{<}{\omega}$ such that for each $n \in \omega$, $\Phi\left(\left\langle \pr{{C}_{i}}{{o}_{i}}: i \leq 2n \right\rangle\right) = \seq{p}{i}{\leq}{n}$.
 Then $\langle {o}_{2i+1}: i < \omega \rangle$ and $\seq{p}{i}{<}{\omega}$ satisfy the hypotheses of Lemma \ref{lem:fusionU}.
 To see this, first note that applying (3) with $n=i$ gives ${o}_{2i+3}={o}_{2n+3} \in {H}_{{p}_{n+1}, {o}_{2n+1}+1} = {H}_{{p}_{i+1}, {o}_{2i+1}+1}$ and ${o}_{2i+3} > {o}_{2i+1}$.
 Hence $\forall i < j < \omega\[{o}_{2i+1} < {o}_{2j+1}\]$.
 Next, given $i \leq j < \omega$ and $e \in \lv{{T}_{{p}_{j}}}{{o}_{2i+1}}$, if $i=j$, then trivially ${\succc}_{{T}_{{p}_{i}}}(e) \subseteq {\succc}_{{T}_{{p}_{j}}}(e)$.
 If $i < j$, then applying (2)(b) with $n=j$ gives ${\succc}_{{T}_{{p}_{i}}}(e) \subseteq {\succc}_{{T}_{{p}_{j}}}(e)$.
 Therefore Lemma \ref{lem:fusionU} applies and implies that $q = {T}_{q} = {\bigcap}_{i \in \omega}{{T}_{{p}_{i}}} \in \PP(\UUU)$.
 Further, $q \leq {p}_{0} \leq p$, and for each $i \in \omega$, since $q \leq {p}_{i}$, $q \; {\forces}_{\PP(\UUU)} \; {\mathring{x}\left({o}_{2i}\right) = 1}$.
 Since $A = \{{o}_{2i}: i < \omega\} \in \VVV$, $q$ and $A$ satisfy the conditions of (1) of Lemma \ref{lem:preserving}.
 Thus in both this case and in the case considered in the previous paragraph, (1) of Lemma \ref{lem:preserving} applies and it implies that $\PP(\UUU)$ preserves $\VVV$.
 
 $\Sigma$ and $\Phi$ are defined inductively.
 $\Sigma\left(\emptyset\right) = {E}_{p} \in \VVV$.
 If $\left\langle \pr{{C}_{0}}{{o}_{0}} \right\rangle$ is a partial run of ${\Game}^{\mathtt{SelSel}}\left(\VVV, \UUU\right)$ in which Player I has followed $\Sigma$, then ${C}_{0} = {E}_{p}$ and ${o}_{0} \in {C}_{0} = {E}_{p}$.
 Define $\Phi\left(\left\langle \pr{{C}_{0}}{{o}_{0}} \right\rangle\right) = \left\langle {p}_{0} \right\rangle$, where ${p}_{0} \leq p$ and ${p}_{0} \; {\forces}_{\PP(\UUU)} \; {\mathring{x}({o}_{0}) = 1}$.
 Such a ${p}_{0}$ exists by the definition of ${E}_{p}$.
 Define $\Sigma\left(\left\langle \pr{{C}_{0}}{{o}_{0}} \right\rangle\right) = \omega \in \UUU$.
 It is clear that (1)--(4) are satisfied by these definitions.
 Now, assume that for some $n \in \omega$, $\left\langle \pr{{C}_{i}}{{o}_{i}}: i \leq 2n+1 \right\rangle$ is a partial run of ${\Game}^{\mathtt{SelSel}}\left(\VVV, \UUU\right)$ in which Player I has followed $\Sigma$ and that $\Phi\left(\left\langle \pr{{C}_{i}}{{o}_{i}}: i \leq 2n \right\rangle\right) = \seq{p}{i}{\leq}{n}$, and that these satisfy (1)--(4).
 Let $k = {o}_{2n+1}+1$.
 For any $e \in \lv{{T}_{{p}_{n}}}{k}$, ${T}_{{p}_{n}}\langle e \rangle \leq {p}_{n} \leq {p}_{0} \leq p$.
 Therefore ${E}_{{T}_{{p}_{n}}\langle e \rangle} \in \VVV$.
 Define $\Sigma\left( \left\langle \pr{{C}_{i}}{{o}_{i}}: i \leq 2n+1 \right\rangle \right) = {C}_{2n+2} = \bigcap\left\{ {E}_{{T}_{{p}_{n}}\langle e \rangle}: e \in \lv{{T}_{{p}_{n}}}{k} \right\} \in \VVV$.
 If $\left\langle \pr{{C}_{i}}{{o}_{i}}: i \leq 2n+2 \right\rangle$ is a partial continuation of ${\Game}^{\mathtt{SelSel}}\left(\VVV, \UUU\right)$ in which Player I has followed $\Sigma$, then ${o}_{2n+2} \in {C}_{2n+2}$ and so there is a sequence $\left\langle {p}_{n, e}: e \in \lv{{T}_{{p}_{n}}}{k} \right\rangle$ such that for each $e \in \lv{{T}_{{p}_{n}}}{k}$, ${p}_{n, e} \leq {T}_{{p}_{n}}\langle e \rangle$ and ${p}_{n, e} \; {\forces}_{\PP(\UUU)} \; {\mathring{x}\left({o}_{2n+2}\right) = 1}$.
 Define ${p}_{n+1} = {T}_{{p}_{n+1}} = \bigcup\left\{{T}_{{p}_{n, e}}: e \in \lv{{T}_{{p}_{n}}}{k} \right\}$.
 By Lemma \ref{lem:amalgamU}, ${p}_{n+1} \in \PP(\UUU)$ and ${p}_{n+1} \leq {p}_{n}$.
 Define $\Phi\left( \left\langle \pr{{C}_{i}}{{o}_{i}}: i \leq 2n+2 \right\rangle \right) = \seq{p}{i}{\leq}{n+1}$.
 Finally, define $\Sigma\left( \left\langle \pr{{C}_{i}}{{o}_{i}}: i \leq 2n+2 \right\rangle \right) = {C}_{2n+3} = {H}_{{p}_{n+1}, k} \cap \{o \in \omega: o > {o}_{2n+1}\} \in \UUU$, so that if ${o}_{2n+3} \in {C}_{2n+3}$, then ${o}_{2n+3} \in {H}_{{p}_{n+1}, {o}_{2n+1}+1}$ and ${o}_{2n+3} > {o}_{2n+1}$.
 It is clear that (1), (3), and (4) are satisfied.
 It is also clear that (2)(a) holds.
 For 2(b), suppose $i < j \leq n+1$.
 If $i < j \leq n$, then the induction hypothesis gives what is needed.
 So assume that $i < j=n+1$.
 Suppose that $e \in \lv{{T}_{{p}_{j}}}{{o}_{2i+1}}$.
 Then $e \in \lv{{T}_{{p}_{n}}}{{o}_{2i+1}}$.
 If $i=n$, then trivially ${\succc}_{{T}_{{p}_{i}}}(e) \subseteq {\succc}_{{T}_{{p}_{n}}}(e)$, while if $i < n$, then by the induction hypothesis, ${\succc}_{{T}_{{p}_{i}}}(e) \subseteq {\succc}_{{T}_{{p}_{n}}}(e)$.
 Thus in either case, ${\succc}_{{T}_{{p}_{i}}}(e) \subseteq {\succc}_{{T}_{{p}_{n}}}(e)$.
 Suppose $\sigma \in {\succc}_{{T}_{{p}_{n}}}(e)$.
 Then ${e}^{\frown}{\langle \sigma \rangle} \in {T}_{{p}_{n}}$, and $\dom({e}^{\frown}{\langle \sigma \rangle}) = {o}_{2i+1}+1 \leq {o}_{2n+1}+1 = k$.
 Choose some ${e}^{\ast} \in \lv{{T}_{{p}_{n}}}{k}$ with ${e}^{\frown}{\langle \sigma \rangle} \subseteq {e}^{\ast}$.
 Then ${e}^{\frown}{\langle \sigma \rangle} \in {T}_{{p}_{n, {e}^{\ast}}} \subseteq {T}_{{p}_{n+1}}$, whence $\sigma \in {\succc}_{{T}_{{p}_{n+1}}}(e)$.
 Therefore, ${\succc}_{{T}_{{p}_{i}}}(e) \subseteq {\succc}_{{T}_{{p}_{n}}}(e) \subseteq {\succc}_{{T}_{{p}_{n+1}}}(e) = {\succc}_{{T}_{{p}_{j}}}(e)$, as required for (2)(b).
 For (2)(c), suppose $r \leq {p}_{n+1}$.
 Choose $e \in \lv{{T}_{r}}{k}$.
 Then $e \in \lv{{T}_{{p}_{n}}}{k}$ and ${T}_{r}\langle e \rangle \leq {T}_{{p}_{n+1}}\langle e \rangle \leq {p}_{n, e}$.
 So ${T}_{r}\langle e \rangle \; {\forces}_{\PP(\UUU)} \; {\mathring{x}\left({o}_{2n+2}\right) = 1}$.
 Therefore, ${p}_{n+1} \; {\forces}_{\PP(\UUU)} \; {\mathring{x}\left({o}_{2n+2}\right) = 1}$, as needed for (2)(c).
 This concludes the definition of $\Sigma$ and $\Phi$ and the proof.
\end{proof}
\begin{Lemma} \label{lem:genericrealU}
 Suppose $G$ is $(\V, \PP(\UUU))$-generic.
 In $\VG$, there exists a function $F \in {\prod}_{k \in \omega}{{2}^{k}}$ such that $\{F\} = \bigcap{\left\{ \[{T}_{p}\]: p \in G \right\}}$.
\end{Lemma}
\begin{proof}
 Similar to Lemma \ref{lem:genericreal}.
\end{proof}
\begin{Def} \label{def:CGU}
 Suppose $G$ is $(\V, \PP(\UUU))$-generic.
 In $\VG$, let ${F}_{G} \in {\prod}_{k \in \omega}{{2}^{k}}$ denote the unique function such that $\{{F}_{G}\} = \bigcap\left\{\[{T}_{p}\]: p \in G \right\}$.
 Define ${c}_{G}: \pc{\omega}{2} \rightarrow 2$ as follows.
 For $\{k, l\} \in \pc{\omega}{2}$ with $k < l$, define ${c}_{G}(\{k, l\}) = {F}_{G}(l)(k) \in 2$.
 In $\V$, let ${\mathring{F}}_{G}$ and ${\mathring{c}}_{G}$ be $\PP(\UUU)$-names that are forced by every condition to denote ${F}_{G}$ and ${c}_{G}$ respectively.
\end{Def}
\begin{Lemma} \label{lem:bigrestrictionU}
 Suppose $k \leq k' \leq l < \omega$.
 Suppose $A \subseteq {2}^{l}$ is $k'$-big.
 Suppose $t \in 2$.
 Then $B = \{\tau \in A: \forall n \in \omega\[k \leq n < k' \implies \tau(n) = 1-t\]\} \subseteq A \subseteq {2}^{l}$ is $k$-big.
\end{Lemma}
\begin{proof}
 Let $\sigma: k \rightarrow 2$ be given.
 Define $\sigma': k' \rightarrow 2$ by stipulating that for any $n \in k'$,
 \begin{align*}
  \sigma'(n) = \begin{cases}
   \sigma(n) &\ \text{if} \ n \in k\\
   1-t       &\ \text{if} \ n \notin k.
  \end{cases}
 \end{align*}
 As $A$ is $k'$-big, there exists $\tau \in A$ such that $\sigma' \subseteq \tau$.
 $\tau \in B$ because for any $n \in \omega$, if $k \leq n < k'$, then $\tau(n) = \sigma'(n) = 1-t$.
 Further, $\sigma \subseteq \sigma' \subseteq \tau$.
 As $\sigma$ was arbitrary, $B$ is $k$-big.
\end{proof}
\begin{Lemma} \label{lem:extend1-tU}
 Suppose $k \leq k' \leq l < \omega$.
 Suppose $p \in \PP(\UUU)$ and $t \in 2$.
 Assume that $l \in {H}_{p, k'}$.
 Then there exists $q \leq p$ such that:
 \begin{enumerate}
  \item
  $\forall i \in \omega \setminus \{l\} \forall e \in \lv{{T}_{q}}{i}\[{\succc}_{{T}_{q}}(e) = {\succc}_{{T}_{p}}(e)\]$;
  \item
  $l \in {H}_{q, k}$;
  \item
  for each $n \in \omega$, if $k \leq n < k'$, then $q \; {\forces}_{\PP(\UUU)} \; {{\mathring{c}}_{G}(\{n, l\}) = 1-t}$.
 \end{enumerate}
\end{Lemma}
\begin{proof}
 Since $l \in {H}_{p, k'}$, for each $f \in \lv{{T}_{p}}{l}$, ${A}_{f} = {\succc}_{{T}_{p}}(f) \subseteq {2}^{l}$ is $k'$-big.
 Thus, by applying Lemma \ref{lem:bigrestrictionU}, it is seen that for each $f \in \lv{{T}_{p}}{l}$, ${B}_{f} = \{\tau \in {A}_{f}: \forall n \in \omega\[k \leq n < k' \implies \tau(n)=1-t\]\}$ is $k$-big.
 In particular, each ${B}_{f}$ is non-empty.
 Since $\lv{{T}_{p}}{l}$ is non-empty,
 \begin{align*}
  B = \left\{{f}^{\frown}{\langle \tau \rangle}: f \in \lv{{T}_{p}}{l} \wedge \tau \in {B}_{f} \right\}
 \end{align*}
 is a non-empty subset of $\lv{{T}_{p}}{l+1}$.
 Therefore by defining
 \begin{align*}
  q = \bigcup\left\{ {T}_{p}\langle {f}^{\frown}{\langle \tau \rangle} \rangle: f \in \lv{{T}_{p}}{l} \wedge \tau \in {B}_{f} \right\},
 \end{align*}
 Lemma \ref{lem:amalgamU} insures that $q \in \PP(\UUU)$ and $q \leq p$.
 
 The argument for (1) is similar to the corresponding argument in the proof of Lemma \ref{lem:extend1-t}.
 
 For (2), consider $f \in \lv{{T}_{q}}{l}$.
 It needs to be seen that ${\succc}_{{T}_{q}}(f) \subseteq {2}^{l}$ is $k$-big.
 Indeed $f \in \lv{{T}_{p}}{l}$ and for any $\tau \in {B}_{f}$, ${f}^{\frown}{\langle \tau \rangle} \in {T}_{q}$, whence $\tau \in {\succc}_{{T}_{q}}(f)$.
 So ${B}_{f} \subseteq {\succc}_{{T}_{q}}(f)$, and so ${\succc}_{{T}_{q}}(f)$ is $k$-big.
 
 Finally for (3), fix $n \in \omega$ such that $k \leq n < k'$.
 Now let $G$ be $(\V, \PP(\UUU))$-generic with $q \in G$.
 Since ${F}_{G} \in \[{T}_{q}\]$, ${F}_{G} \restrict l+1 \in {T}_{q}$, and so ${F}_{G} \restrict l+1 = {f}^{\frown}{\langle \tau \rangle}$, for some $f \in \lv{{T}_{p}}{l}$ and $\tau \in {B}_{f}$.
 Therefore, ${c}_{G}(\{n, l\}) = {F}_{G}(l)(n) = \tau(n) = 1-t$, by the definition of ${B}_{f}$.
\end{proof}
\begin{Lemma} \label{lem:1-tfusionU}
 Suppose $p \in \PP(\UUU)$ and $\psi \in \BS$ is such that for all $k \in \omega$, $k \leq \psi(k)$.
 Let $t \in 2$.
 Then there exist $q \leq p$ and $\seq{l}{i}{\in}{\omega} \in {\omega}^{\omega}$ such that $\forall i < j < \omega\[{l}_{i} < {l}_{j}\]$, $\{{l}_{i}: i \in \omega\} \in \UUU$, $\forall i \in \omega\[\psi\left({l}_{i}+1\right) \leq {l}_{i+1}\]$, and for each $i \in \omega$, for each $n \in \omega$, if ${l}_{i}+1 \leq n < \psi\left({l}_{i}+1\right)$, then $q \; {\forces}_{\PP(\UUU)} \; {{\mathring{c}}_{G}\left(\{n, {l}_{i+1}\}\right)=1-t}$.
\end{Lemma}
\begin{proof}
 For each $n \in \omega$, let ${A}_{n} = {H}_{p, \psi(n+1)} \cap \{l \in \omega: l > \psi(n+1)\} \in \UUU$.
 As $\UUU$ is selective, there exists $\seq{l}{i}{\in}{\omega} \in \BS$ such that $\forall i < j < \omega\[{l}_{i} < {l}_{j}\]$, $\{{l}_{i}: i \in \omega\} \in \UUU$, and for each $i \in \omega$, ${l}_{i+1} \in {A}_{{l}_{i}}$.
 For ease of notation, let ${k}_{i} = {l}_{i}+1$, and ${k}_{i}' = \psi\left({l}_{i}+1\right)$.
 Thus ${l}_{i+1} \in {H}_{p, {k}_{i}'}$, and ${k}_{i} \leq {k}_{i}' \leq {l}_{i+1} < \omega$.
 
 Lemma \ref{lem:fusionU} will be used to obtain $q$.
 To this end, construct a sequence $\seq{p}{i}{\in}{\omega}$ satisfying the following:
 \begin{enumerate}
  \item
  $\forall i \in \omega\[{p}_{i} \in \PP(\UUU)\]$, $\forall i \in \omega\[{p}_{i+1} \leq {p}_{i}\]$, ${p}_{0} = p$;
  \item
  for each $i \in \omega$, ${l}_{i+1} \in {H}_{{p}_{i+1}, {k}_{i}}$;
  \item
  for each $i \leq j < \omega$ and for each $e \in \lv{{T}_{{p}_{j}}}{{l}_{i}}$, ${\succc}_{{T}_{{p}_{i}}}(e) \subseteq {\succc}_{{T}_{{p}_{j}}}(e)$;
  \item
  for each $j \leq {j}^{\ast} < \omega$, ${l}_{{j}^{\ast}+1} \in {H}_{{p}_{j}, {k}_{{j}^{\ast}}'}$;
  \item
  for each $i \in \omega$, for each $n \in \omega$, if ${k}_{i} \leq n < {k}_{i}'$, then
  \begin{align*}
   {p}_{i+1} \; {\forces}_{\PP(\UUU)} \; {{\mathring{c}}_{G}\left(\{n, {l}_{i+1}\}\right)=1-t}.
  \end{align*}
 \end{enumerate}
 Suppose for a moment that such a sequence has been constructed.
 Then by Lemma \ref{lem:fusionU}, $q = {T}_{q} = {\bigcap}_{i \in \omega}{{T}_{{p}_{i}}} \in \PP(\UUU)$, $q \leq {p}_{0} = p$, and for each $i \in \omega$, since $q \leq {p}_{i+1}$, $q$ is as desired because of (5).
 
 The sequence $\seq{p}{i}{\in}{\omega}$ is constructed by induction.
 Define ${p}_{0} = p$ and notice that (4) is satisfied because for each ${j}^{\ast} < \omega$, ${l}_{{j}^{\ast}+1} \in {H}_{p, {k}_{{j}^{\ast}}'}$.
 Fix $j \in \omega$ and suppose that $\seq{p}{i}{\leq}{j}$ satisfying (1)--(5) is given.
 Applying (4) with $j = {j}^{\ast}$ yields ${l}_{j+1} \in {H}_{{p}_{j}, {k}_{j}'}$.
 Hence by Lemma \ref{lem:extend1-tU}, there exists ${p}_{j+1} \leq {p}_{j}$ satisfying (1)--(3) of Lemma \ref{lem:extend1-tU}.
 It is clear that (1), (2), and (5) are satisfied.
 And the verification of (3) and (4) is similar to the corresponding part of the proof of Lemma \ref{lem:1-tfusion}.
 This concludes the induction and the proof.
\end{proof}
\begin{Lemma} \label{thm:noresseructionU}
 Suppose $\QQ$ is an $\BS$-bounding forcing.
 If $\PP(\UUU)$ completely embeds into $\QQ$, then
 \begin{align*}
  {\forces}_{\QQ}\;{``\text{there is no selective ultrafilter on} \ \omega \ \text{extending} \ \UUU''}.
 \end{align*}
\end{Lemma}
\begin{proof}
 Let $\pi: \PP(\UUU) \rightarrow \QQ$ be a complete embedding, and let ${\pi}^{\ast}$ denote the associated map from $\PP(\UUU)$-names to $\QQ$-names.
 Let $\mathring{c}$ denote ${\pi}^{\ast}\left({\mathring{c}}_{G}\right)$.
 Then ${\forces}_{\QQ}\;{\mathring{c}: \pc{\omega}{2} \rightarrow 2}$.
 Suppose $t \in 2$ and that $\mathring{A}$ is a $\QQ$-name such that ${\forces}_{\QQ}\;{\mathring{A} \in \pc{\omega}{\omega}}$.
 Let ${p}^{\ast} \in \QQ$ be given.
 It will be shown that there exist ${q}^{\ast} \leq {p}^{\ast}$ and $B \in \UUU$ such that
 \begin{align*}
  \tag{$\ast$} \forall l \in B \forall {r}^{\ast} \leq {q}^{\ast} \exists n < l \exists {r}^{\ast}_{1} \leq {r}^{\ast}\[{r}^{\ast}_{1} \; {\forces}_{\QQ} \; {n \in \mathring{A}} \ \text{and} \ {r}^{\ast}_{1} \; {\forces}_{\QQ} \; {\mathring{c}\left(\{n, l\}\right) \neq t}\].
 \end{align*}
 
 Let $G$ be any $(\V, \QQ)$-generic filter with ${p}^{\ast} \in G$.
 Then $\mathring{A}\[G\] \in \pc{\omega}{\omega}$ in $\VG$.
 So in $\VG$, there is a function $\varphi: \omega \rightarrow \omega$ such that for every $k \in \omega$, there exists $n \in \mathring{A}\[G\]$ with $k \leq n < \varphi(k)$.
 As this holds for every $(\V, \QQ)$-generic $G$ with ${p}^{\ast} \in G$, there is a $\QQ$-name $\mathring{\varphi}$ in $\V$ such that ${\forces}_{\QQ} \; {\mathring{\varphi}: \omega \rightarrow \omega}$ and ${p}^{\ast} \; {\forces}_{\QQ} \; {\forall k \in \omega \exists n \in \mathring{A}\[k \leq n < \mathring{\varphi}(k)\]}$.
 
 Since $\QQ$ is $\BS$-bounding, there exist ${p}^{\ast}_{1} \leq {p}^{\ast}$ and $\psi: \omega \rightarrow \omega$ in $\V$ such that ${p}^{\ast}_{1} \; {\forces}_{\QQ} \; {\forall k \in \omega\[\mathring{\varphi}(k) < \psi(k)\]}$.
 Let $p \in \PP(\UUU)$ be a reduction of ${p}^{\ast}_{1}$ with respect to the complete embedding $\pi$.
 Applying Lemma \ref{lem:1-tfusionU} in $\V$, find $q \leq p$ and $\seq{l}{i}{\in}{\omega} \in \BS$ satisfying the conclusions of Lemma \ref{lem:1-tfusionU}.
 Let $B = \left\{ {l}_{j}: j > 0 \right\} \in \UUU$.
 By the choice of $p$, $\pi(q)$ is compatible with ${p}^{\ast}_{1}$ in $\QQ$.
 Choose any ${q}^{\ast} \leq \pi(q), {p}^{\ast}_{1}$.
 To see that ${q}^{\ast} \leq {p}^{\ast}$ and $B \in \UUU$ satisfy ($\ast$), fix some $l \in B$ and ${r}^{\ast} \leq {q}^{\ast}$.
 Then $l = {l}_{i+1}$, for some $i \in \omega$.
 Since ${r}^{\ast} \leq {p}^{\ast}_{1} \leq {p}^{\ast}$, there exist ${r}^{\ast}_{1} \leq {r}^{\ast}$ and $n$ such that ${l}_{i}+1 \leq n$, ${r}^{\ast}_{1} \; {\forces}_{\QQ} \; {n \in \mathring{A} \subseteq \omega}$, and ${r}^{\ast}_{1} \; {\forces}_{\QQ} \; {n < \mathring{\varphi}\left({l}_{i}+1\right) < \psi\left({l}_{i}+1\right) \leq {l}_{i+1} = l}$.
 By the choice of $q$, $q \; {\forces}_{\PP(\UUU)} \; {{\mathring{c}}_{G}(\{n, l\}) = {\mathring{c}}_{G}(\{n, {l}_{i+1}\}) = 1-t}$, and so
 \begin{align*}
  \pi(q) \; {\forces}_{\QQ} \; {\mathring{c}(\{n, l\}) = {\pi}^{\ast}\left({\mathring{c}}_{G}\right)(\{n, l\}) = 1-t}.
 \end{align*}
 Therefore, ${r}^{\ast}_{1} \; {\forces}_{\QQ} \; {\mathring{c}(\{n, l\}) = 1-t}$, as needed.
 
 To conclude, suppose for a contradiction that $G$ is some $(\V, \QQ)$-generic filter and that in $\VG$, there is a selective ultrafilter $\VVV$ on $\omega$ with $\VVV \supseteq \UUU$.
 Then there exist $A \in \VVV$ and $t \in 2$ so that $\mathring{c}\[G\]$ is constantly $t$ on $\pc{A}{2}$.
 Let $\mathring{A}$ be a $\QQ$-name in $\V$ so that ${\forces}_{\QQ} \; {\mathring{A} \in \pc{\omega}{\omega}}$ and $A = \mathring{A}\[G\]$.
 By what has been proved above, there exist ${q}^{\ast} \in G$ and $B \in \UUU$ satisfying ($\ast$).
 Choose any $l \in B \cap A$.
 By ($\ast$), there is $n < l$ such that $n \in \mathring{A}\[G\] = A$ and $\mathring{c}\[G\](\{n, l\}) \neq t$, contradicting the fact that $\mathring{c}\[G\]$ is constantly $t$ on $\pc{A}{2}$.
\end{proof}
It is worth noting here that by iterating partial orders of the form $\PP(\UUU)$, it is possible to get models with, for example, a unique selective ultrafilter.
Models like this were first produced by Shelah in \cite{PIF}.
However, the details seem to be simpler for $\PP(\UUU)$.
\begin{Theorem} \label{thm:intermediate}
 Put ${S}^{2}_{1} = \{\alpha < {\omega}_{2}: \cf(\alpha) = {\omega}_{1}\}$.
 Let $\V$ be a transitive model of a sufficiently large fragment of $\ZFC$ which satisfies $\CH$ and $\diamondsuit\left({S}^{2}_{1}\right)$.
 Let $\kappa < {\aleph}_{2}$ be a cardinal.
 Suppose $\left\{ {\UUU}_{\alpha}: \alpha < \kappa \right\}$ is a family of selective ultrafilters on $\omega$ such that ${\UUU}_{\alpha} \; {\not\equiv}_{RK} \; {\UUU}_{\beta}$, for $\alpha \neq \beta$.
 Then there is a cardinal preserving forcing extension of $\V$ in which:
 \begin{enumerate}
  \item
  there are no stable ordered-union ultrafilters on $\FIN$;
  \item
  each ${\UUU}_{\alpha}$ generates a selective ultrafilter on $\omega$;
  \item
  the ultrafilter generated by ${\UUU}_{\alpha}$ is not RK-isomorphic to the ultrafilter generated by ${\UUU}_{\beta}$, for $\alpha \neq \beta$;
  \item
  if $\VVV$ is any selective ultrafilter on $\omega$, then $\VVV \; {\equiv}_{RK} \; {\UUU}_{\alpha}$, for some $\alpha < \kappa$.
 \end{enumerate}
\end{Theorem}
\begin{proof}
 Fixing some diamond sequence witnessing $\diamondsuit\left({S}^{2}_{1}\right)$, define a CS iteration $\left\langle {\PP}_{\alpha}; {\mathring{\QQ}}_{\alpha}: \alpha \leq {\omega}_{2} \right\rangle$ in $\V$ as follows.
 Assume $\alpha < {\omega}_{2}$ and that ${\PP}_{\alpha}$ is proper, $\BS$-bounding, satisfies the ${\aleph}_{2}$-c.c.\@, that ${\forces}_{\alpha} \; {\CH}$, and that ${\PP}_{\alpha}$ preserves ${\UUU}_{\xi}$, for all $\xi < \kappa$.
 Let $\{{\mathring{\UUU}}^{\alpha}_{\xi}: \xi < \kappa\}$ be a family of ${\PP}_{\alpha}$-names so that for each $\xi < \kappa$, ${\forces}_{\alpha} \; {{\mathring{\UUU}}^{\alpha}_{\xi} = \left\{A \subseteq \omega: \exists B \in {\UUU}_{\xi}\[B \subseteq A\] \right\}}$.
 Then ${\forces}_{\alpha} \; {``{\mathring{\UUU}}^{\alpha}_{\xi} \ \text{is a selective ultrafilter on} \ \omega''}$, for every $\xi < \kappa$, and ${\forces}_{\alpha} \; { {\mathring{\UUU}}^{\alpha}_{\xi} \; {\not\equiv}_{RK} \; {\mathring{\UUU}}^{\alpha}_{\zeta} }$, for $\xi \neq \zeta$.
 
 Observe that ${\forces}_{\alpha} \; {\text{``stable ordered-union ultrafilters exist''}}$ because ${\forces}_{\alpha} \; {\CH}$.
 And ${\forces}_{\alpha} \; {``\exists \VVV\[\VVV \ \text{is a selective ultrafilter on} \ \omega \ \text{and} \ \forall \xi < \kappa\[{\mathring{\UUU}}^{\alpha}_{\xi} \; {\not\equiv}_{RK} \; \VVV\]\]''}$ due to the fact that $\kappa < {\aleph}_{2}$ and ${\forces}_{\alpha} \; {\CH}$.
 
 If the diamond sequence at $\alpha$ codes a pair $\pr{\mathring{\GGG}}{p}$ such that $p \in {\PP}_{\alpha}$, $\mathring{\GGG}$ is a ${\PP}_{\alpha}$-name, and $p \; {\forces}_{\alpha} \; {``\mathring{\GGG} \ \text{is a stable ordered-union ultrafilter on} \ \FIN''}$, then choose a ${\PP}_{\alpha}$-name ${\mathring{\HHH}}_{\alpha}$ such that ${\forces}_{\alpha} \; {``{\mathring{\HHH}}_{\alpha} \ \text{is a stable ordered-union ultrafilter on} \ \FIN''}$ and $p \; {\forces}_{\alpha} \; {{\mathring{\HHH}}_{\alpha} = \mathring{\GGG}}$, and define ${\mathring{\QQ}}_{\alpha}$ to be a full ${\PP}_{\alpha}$-name so that ${\forces}_{\alpha} \; {{\mathring{\QQ}}_{\alpha} = \PP({\mathring{\HHH}}_{\alpha})}$.
 
 If the diamond sequence at $\alpha$ codes a pair $\pr{\mathring{\GGG}}{p}$ such that $p \in {\PP}_{\alpha}$, $\mathring{\GGG}$ is a ${\PP}_{\alpha}$-name, $p \; {\forces}_{\alpha} \; {``\mathring{\GGG} \ \text{is a selective ultrafilter on} \ \omega''}$, and for each $\xi < \kappa$, $p \; {\forces}_{\alpha} \; {\mathring{\GGG} \; {\not\equiv}_{RK} \; {\mathring{\UUU}}^{\alpha}_{\xi}}$, then choose a ${\PP}_{\alpha}$-name ${\mathring{\HHH}}_{\alpha}$ such that ${\forces}_{\alpha} \; {``{\mathring{\HHH}}_{\alpha} \ \text{is a selective ultrafilter on} \ \omega''}$, and for each $\xi < \kappa$, ${\forces}_{\alpha} \; {{\mathring{\HHH}}_{\alpha} \; {\not\equiv}_{RK} \; {\mathring{\UUU}}^{\alpha}_{\xi}}$, and define ${\mathring{\QQ}}_{\alpha}$ to be a full ${\PP}_{\alpha}$-name so that ${\forces}_{\alpha} \; {{\mathring{\QQ}}_{\alpha} = \PP({\mathring{\HHH}}_{\alpha})}$.
 
 If neither of these occurs, then choose an arbitrary ${\PP}_{\alpha}$-name ${\mathring{\HHH}}_{\alpha}$ such that ${\forces}_{\alpha} \; {``{\mathring{\HHH}}_{\alpha} \ \text{is a stable ordered-union ultrafilter on} \ \FIN''}$, and define ${\mathring{\QQ}}_{\alpha}$ to be a full ${\PP}_{\alpha}$-name so that ${\forces}_{\alpha} \; {{\mathring{\QQ}}_{\alpha} = \PP({\mathring{\HHH}}_{\alpha})}$.
 Note that in all cases ${\forces}_{\alpha} \; {\lc {\mathring{\QQ}}_{\alpha} \rc = {\aleph}_{1}}$ because ${\forces}_{\alpha} \; {\CH}$.
 Standard arguments in the theory of proper forcing (see Shelah~\cite{PIF} or Abraham~\cite{Ab}) together with lemmas proved earlier therefore imply that for each $\delta \leq {\omega}_{2}$, ${\PP}_{\delta}$ is proper, $\BS$-bounding, and satisfies the ${\aleph}_{2}$-c.c.
 In particular, this implies that the extension by ${\PP}_{{\omega}_{2}}$ preserves all cardinals.
 Furthermore, for each $\delta < {\omega}_{2}$, ${\forces}_{\delta} \; {\CH}$.
 An easy inductive argument using the earlier lemmas shows that for each $\xi < \kappa$ and for each $\delta \leq {\omega}_{2}$, ${\PP}_{\delta}$ preserves ${\UUU}_{\xi}$.
 Therefore, for each $\xi < \kappa$, ${\forces}_{{\omega}_{2}} \; {``{\mathring{\UUU}}^{{\omega}_{2}}_{\xi} \ \text{is a selective ultrafilter on} \ \omega''}$, and ${\forces}_{{\omega}_{2}} \; {{\mathring{\UUU}}^{{\omega}_{2}}_{\xi} \; {\not\equiv}_{RK} \; {\mathring{\UUU}}^{{\omega}_{2}}_{\zeta}}$, for $\zeta \neq \xi$.
 An argument identical to the one in the proof of Theorem \ref{thm:aleph2selectives} shows that ${\forces}_{{\omega}_{2}} \; {``\text{there are no stable ordered-union ultrafilters on} \ \FIN''}$.
 
 Next, suppose for a contradiction that $\mathring{\VVV}$ is a ${\PP}_{{\omega}_{2}}$-name such that for some $p \in {\PP}_{{\omega}_{2}}$, $p \; {\forces}_{{\omega}_{2}} \; {``\mathring{\VVV} \ \text{is a selective ultrafilter on} \ \omega''}$, and for each $\xi < \kappa$, $p \; {\forces}_{{\omega}_{2}} \; {\mathring{\VVV} \; {\not\equiv}_{RK} \; {\mathring{\UUU}}^{{\omega}_{2}}_{\xi}}$.
 Then by a standard argument, there exists $\alpha \in {S}^{2}_{1}$ such that the diamond sequence at $\alpha$ codes a pair $\pr{\mathring{\GGG}}{p \restrict \alpha}$ such that $\mathring{\GGG}$ is a ${\PP}_{\alpha}$-name,
 \begin{align*}
  p \restrict \alpha \; {\forces}_{\alpha} \; {``\mathring{\GGG} \ \text{is a selective ultrafilter on} \ \omega''},
 \end{align*}
 for each $\xi < \kappa$, $p \restrict \alpha \; {\forces}_{\alpha} \; {\mathring{\GGG} \; {\not\equiv}_{RK} \; {\mathring{\UUU}}^{\alpha}_{\xi}}$, and $p \; {\forces}_{{\omega}_{2}} \; {\mathring{\GGG} \subseteq \mathring{\VVV}}$.
 Let ${G}_{{\omega}_{2}}$ be a $(\V, {\PP}_{{\omega}_{2}})$-generic filter with $p \in {G}_{{\omega}_{2}}$, and let ${G}_{\alpha}$ denote its projection to ${\PP}_{\alpha}$.
 In $\V\[{G}_{\alpha}\]$, ${\mathring{\HHH}}_{\alpha}\[{G}_{\alpha}\] = \mathring{\GGG}\[{G}_{\alpha}\]$ is a selective ultrafilter on $\omega$ such that for every $\xi < \kappa$, ${\mathring{\HHH}}_{\alpha}\[{G}_{\alpha}\] \; {\not\equiv}_{RK} \; {\mathring{\UUU}}^{\alpha}_{\xi}\[{G}_{\alpha}\]$.
 Moreover, $\PP\left({\mathring{\HHH}}_{\alpha}\[{G}_{\alpha}\]\right)$ completely embeds into the completion of ${\PP}_{{\omega}_{2}} \slash {G}_{\alpha}$, and ${\PP}_{{\omega}_{2}} \slash {G}_{\alpha}$ is $\BS$-bounding.
 Therefore Lemma \ref{thm:noresseructionU} implies that
 \begin{align*}
  {\forces}_{{\PP}_{{\omega}_{2}} \slash {G}_{\alpha}} \; {``\text{there is no selective ultrafilter on} \ \omega \ \text{extending} \ {\mathring{\HHH}}_{\alpha}\[{G}_{\alpha}\]''}.
 \end{align*}
 However ${G}_{{\omega}_{2}}$ is a $\left(\V\[{G}_{\alpha}\], {\PP}_{{\omega}_{2}} \slash {G}_{\alpha}\right)$-generic filter, $\V\[{G}_{\alpha}\]\[{G}_{{\omega}_{2}}\] = \V\[{G}_{{\omega}_{2}}\]$, and in $\V\[{G}_{{\omega}_{2}}\]$, $\mathring{\VVV}\[{G}_{{\omega}_{2}}\]$ is a selective ultrafilter on $\omega$ extending $\mathring{\GGG}\[{G}_{{\omega}_{2}}\] = \mathring{\GGG}\[{G}_{\alpha}\] = {\mathring{\HHH}}_{\alpha}\[{G}_{\alpha}\]$.
 This is a contradiction which shows that in the extension by ${\PP}_{{\omega}_{2}}$, every selective ultrafilter on $\omega$ is RK-isomorphic to the ultrafilter generated by ${\UUU}_{\xi}$, for some $\xi < \kappa$.
\end{proof}
\def\polhk#1{\setbox0=\hbox{#1}{\ooalign{\hidewidth
  \lower1.5ex\hbox{`}\hidewidth\crcr\unhbox0}}}
\providecommand{\bysame}{\leavevmode\hbox to3em{\hrulefill}\thinspace}
\providecommand{\MR}{\relax\ifhmode\unskip\space\fi MR }
\providecommand{\MRhref}[2]{%
  \href{http://www.ams.org/mathscinet-getitem?mr=#1}{#2}
}
\providecommand{\href}[2]{#2}


\begin{thebibliography}{10}

\bibitem{Ab}
U.~Abraham, \emph{Proper forcing}, Handbook of set theory. {V}ols. 1, 2, 3,
  Springer, Dordrecht, 2010, pp.~333--394.

\bibitem{BJ}
T.~Bartoszy{\'n}ski and H.~Judah, \emph{Set theory: On the structure of the
  real line}, A K Peters Ltd., Wellesley, MA, 1995.

\bibitem{MR354394}
J.~E. Baumgartner, \emph{A short proof of {H}indman's theorem}, J.
  Combinatorial Theory Ser. A \textbf{17} (1974), 384--386.

\bibitem{MR2757532}
V.~Bergelson, \emph{Ultrafilters, {IP} sets, dynamics, and combinatorial number
  theory}, Ultrafilters across mathematics, Contemp. Math., vol. 530, Amer.
  Math. Soc., Providence, RI, 2010, pp.~23--47.

\bibitem{blassrk}
A.~Blass, \emph{The {R}udin-{K}eisler ordering of {$P$}-points}, Trans. Amer.
  Math. Soc. \textbf{179} (1973), 145--166.

\bibitem{blass-hindman}
\bysame, \emph{Ultrafilters related to {H}indman's finite-unions theorem and
  its extensions}, Logic and combinatorics ({A}rcata, {C}alif., 1985), Contemp.
  Math., vol.~65, Amer. Math. Soc., Providence, RI, 1987, pp.~89--124.

\bibitem{Bl}
\bysame, \emph{Selective ultrafilters and homogeneity}, Ann. Pure Appl. Logic
  \textbf{38} (1988), no.~3, 215--255.

\bibitem{MR906807}
A.~Blass and N.~Hindman, \emph{On strongly summable ultrafilters and union
  ultrafilters}, Trans. Amer. Math. Soc. \textbf{304} (1987), no.~1, 83--97.

\bibitem{simple}
A.~Blass and S.~Shelah, \emph{There may be simple {$P_{\aleph_1}$}- and
  {$P_{\aleph_2}$}-points and the {R}udin-{K}eisler ordering may be downward
  directed}, Ann. Pure Appl. Logic \textbf{33} (1987), no.~3, 213--243.

\bibitem{carlosstevo}
C.~A. Di~Prisco and S.~Todorcevic, \emph{Souslin partitions of products of
  finite sets}, Adv. Math. \textbf{176} (2003), no.~1, 145--173.

\bibitem{MR101283}
R.~Ellis, \emph{Distal transformation groups}, Pacific J. Math. \textbf{8}
  (1958), 401--405.

\bibitem{MR284352}
R.~L. Graham and B.~L. Rothschild, \emph{Ramsey's theorem for {$n$}-parameter
  sets}, Trans. Amer. Math. Soc. \textbf{159} (1971), 257--292.

\bibitem{MR307926}
N.~Hindman, \emph{The existence of certain ultra-filters on {$N$} and a
  conjecture of {G}raham and {R}othschild}, Proc. Amer. Math. Soc. \textbf{36}
  (1972), 341--346.

\bibitem{MR349574}
\bysame, \emph{Finite sums from sequences within cells of a partition of
  {$N$}}, J. Combinatorial Theory Ser. A \textbf{17} (1974), 1--11.

\bibitem{MR2991425}
N.~Hindman, J.~Stepr\={a}ns, and D.~Strauss, \emph{Semigroups in which all
  strongly summable ultrafilters are sparse}, New York J. Math. \textbf{18}
  (2012), 835--848.

\bibitem{MR2893605}
N.~Hindman and D.~Strauss, \emph{Algebra in the {S}tone-\v{C}ech
  compactification}, De Gruyter Textbook, Walter de Gruyter \& Co., Berlin,
  2012, Theory and applications, Second revised and extended edition.

\bibitem{MR2330595}
J.~G. Mijares, \emph{A notion of selective ultrafilter corresponding to
  topological {R}amsey spaces}, MLQ Math. Log. Q. \textbf{53} (2007), no.~3,
  255--267.

\bibitem{MR373906}
K.~R. Milliken, \emph{Ramsey's theorem with sums or unions}, J. Combinatorial
  Theory Ser. A \textbf{18} (1975), 276--290.

\bibitem{pomegaembed}
D.~Raghavan and S.~Shelah, \emph{On embedding certain partial orders into the
  {P}-points under {R}udin--{K}eisler and {T}ukey reducibility}, Trans. Amer.
  Math. Soc. \textbf{369} (2017), no.~6, 4433--4455.

\bibitem{tukey}
D.~Raghavan and S.~Todorcevic, \emph{Cofinal types of ultrafilters}, Ann. Pure
  Appl. Logic \textbf{163} (2012), no.~3, 185--199.

\bibitem{ramsey}
F.~P. Ramsey, \emph{On a {P}roblem of {F}ormal {L}ogic}, Proc. London Math.
  Soc. (2) \textbf{30} (1929), no.~4, 264--286.

\bibitem{creaturebook}
A.~Ros{\l}anowski and S.~Shelah, \emph{Norms on possibilities. {I}. {F}orcing
  with trees and creatures}, Mem. Amer. Math. Soc. \textbf{141} (1999),
  no.~671, xii+167.

\bibitem{Rudin:66}
M.~E. Rudin, \emph{Types of ultrafilters}, Topology {S}eminar ({W}isconsin,
  1965), Ann. of Math. Studies, No. 60, Princeton Univ. Press, Princeton, N.J.,
  1966, pp.~147--151.

\bibitem{maryellenorder}
\bysame, \emph{Partial orders on the types in {$\beta N$}}, Trans. Amer. Math.
  Soc. \textbf{155} (1971).

\bibitem{walterppoints}
W.~Rudin, \emph{Homogeneity problems in the theory of \v {C}ech
  compactifications}, Duke Math. J. \textbf{23} (1956), 409--419.

\bibitem{PIF}
S.~Shelah, \emph{Proper and improper forcing}, second ed., Perspectives in
  Mathematical Logic, Springer-Verlag, Berlin, 1998.

\bibitem{MR1690694}
\bysame, \emph{There may be no nowhere dense ultrafilter}, Logic {C}olloquium
  '95 ({H}aifa), Lecture Notes Logic, vol.~11, Springer, Berlin, 1998,
  pp.~305--324.

\bibitem{MR424571}
A.~D. Taylor, \emph{A canonical partition relation for finite subsets of
  {$\omega$}}, J. Combinatorial Theory Ser. A \textbf{21} (1976), no.~2,
  137--146.

\bibitem{introramsey}
S.~Todorcevic, \emph{Introduction to {R}amsey spaces}, Annals of Mathematics
  Studies, vol. 174, Princeton University Press, Princeton, NJ, 2010.

\bibitem{MR2520149}
J.~Zapletal, \emph{Preserving {$P$}-points in definable forcing}, Fund. Math.
  \textbf{204} (2009), no.~2, 145--154.

\end{thebibliography}
\end{document}